\documentclass{article}
\usepackage[numbers]{natbib}
\usepackage[greek, english]{babel}
\usepackage{amsbsy,amssymb,amscd,amsmath,amsthm,amsfonts}
\usepackage{eurosym}
\usepackage{epsf}
\usepackage{epsfig}

\newfont{\gothic}{eufm10}

\def\Z{{\mathbb{Z}}}                   \def\R{{\RR}}
\def\RR{{\mathbb{R}}}        \def\N{{\mathbb{N}}}        \def\Q{{\mathbb{Q}}}

        \newtheorem{theorem}{Theorem}[section]
\newtheorem{lemma}[theorem]{Lemma}
\newtheorem{proposition}[theorem]{Proposition}
\newtheorem{corollary}[theorem]{Corollary}
\newtheorem{definition1}[theorem]{Definition}
\newenvironment{definition}{\begin{definition1}\rm}{\end{definition1}}

\newtheorem{remark1}[theorem]{Remark}
\newenvironment{remark}{\begin{remark1}\rm}{\end{remark1}}

\newtheorem{example1}[theorem]{Example}
\newenvironment{example}{\begin{example1}\rm}{\end{example1}}

\def\barray{\begin{eqnarray*}}             \def\earray{\end{eqnarray*}}
\def\beq{\begin{equation}} \def\eeq{\end{equation}}

\makeatletter \title{Ghost circles in lattice Aubry-Mather theory}  
\author{Blaz Mramor\thanks{Department of Mathematics, VU University Amsterdam, The Netherlands, {\tt bmramor@few.vu.nl}.} \ and Bob Rink\thanks{Department of Mathematics, VU University Amsterdam, The Netherlands, {\tt brink@few.vu.nl}.}\ . }
\begin{document}  \hyphenation{boun-da-ry mo-no-dro-my sin-gu-la-ri-ty ma-ni-fold ma-ni-folds re-fe-rence se-cond se-ve-ral dia-go-na-lised con-ti-nuous thres-hold re-sul-ting fi-nite-di-men-sio-nal ap-proxi-ma-tion pro-per-ties ri-go-rous mo-dels mo-no-to-ni-ci-ty}
\newcommand{\X}{\mathbb{X}}

\newcommand{\p}{\partial}
\maketitle
\noindent

\abstract{Monotone lattice recurrence relations such as the Frenkel-Kontorova lattice, arise in Hamiltonian lattice mechanics, as models for ferromagnetism and as discretization of elliptic PDEs. Mathematically, they are a multi-dimensional counterpart of monotone twist maps. \\
\indent Such recurrence relations often admit a variational structure, so that the solutions $x:\Z^d\to \R$ are the stationary points of a formal action function $W(x)$. Given any rotation vector $\omega\in \R^d$, classical Aubry-Mather theory establishes the existence of a large collection of solutions of $\nabla W(x)=0$ of rotation vector $\omega$. For irrational $\omega$, this is the well-known Aubry-Mather set. It consists of global minimizers and it may have gaps. \\
\indent In this paper, we study the parabolic gradient flow $\frac{dx}{dt} = - \nabla W(x)$ and we will prove that every Aubry-Mather set can be interpolated by a continuous gradient-flow invariant family, the so-called `ghost circle'. The existence of these ghost circles is known in dimension $d=1$, for rational rotation vectors and Morse action functions. The main technical result of this paper is therefore a compactness theorem for lattice ghost circles, based on a parabolic Harnack inequality for the gradient flow. This implies the existence of lattice ghost circles of arbitrary rotation vectors and for arbitrary actions. \\ 
\indent As a consequence, we can give a simple proof of the fact that when an Aubry-Mather set has a gap, then this gap must be filled with minimizers, or contain a non-minimizing solution.}


\section{Introduction and outline}\label{intro} 
In this paper we are interested in variational monotone lattice recurrence relations. Before introducing such recurrence relations in full generality, let us discuss as an example the so-called $d$-dimensional Frenkel-Kontorova lattice. Here, the goal is to find a $d$-dimensional ``lattice configuration'' $x:\Z^d\to \R$ that satisfies
\begin{align}\label{RR}  
V'(x_i) - (\Delta x)_i = 0 \  \ \mbox{for all} \ i\in \mathbb{Z}^d.
\end{align} 
In the equation above, the smooth function $V: \mathbb{R} \to \mathbb{R}$ satisfies $V(\xi+1)=V(\xi)$ for all $\xi\in\R$. It has the interpretation of a periodic onsite potential. Setting $||i||:=\sum_{k=1}^{d}|i_k|$, the discrete Laplace operator $\Delta:\mathbb{R}^{\mathbb{Z}^d}\to  \mathbb{R}^{\mathbb{Z}^d}$ is defined as 
\begin{align}\label{Lap}
(\Delta x)_i := \frac{1}{2d} \sum_{||j-i||=1} \!\! (x_j - x_i) \ \mbox{for all} \ i \in \Z^d .
\end{align}
One could think of equation (\ref{RR}) as a naive discretization of the nonlinear elliptic partial differential equation $V'(u) - \Delta u=0$ for a function $u: \mathbb{R}^d\to \mathbb{R}$ and $x_i = u(i)$.
\\
\indent At the same time, equation (\ref{RR}) is relevant for statistical mechanics, because it is related to the Frenkel-Kontorova Hamiltonian lattice differential equation
\begin{align} \label{FKHam}
\frac{d^2 x_i}{dt^2} + V'(x_i) - (\Delta x)_i = 0 \ \mbox{for all} \ i\in\mathbb{Z}^d.
\end{align}
This differential equation describes the motion of particles under the competing influence of an onsite periodic potential field and nearest neighbor attraction. Obviously, equation (\ref{RR}) describes its stationary solutions. 
\\ \indent Finally, in dimension $d=1$, the solutions of equation (\ref{RR}) correspond to orbits of the famous Chirikov standard map $T_V$ 
of the annulus. 
This correspondence is explained in some detail in Appendix A. 
\\ \mbox{}\\
The Frenkel-Kontorova problem (\ref{RR}) is an example from a quite general class of lattice recurrence relations to which the results of this paper apply. These are recurrence relations for which there exists, for every $j\in \Z^d$, a real-valued ``local potential'' function $S_j:\R^{\Z^d}\to \R$ so that the relation can be written in the form 
\begin{align}\label{RRR1}
\sum_{j\in \Z^d} \partial_iS_j(x) = 0\ \mbox{for all} \ i\in \mathbb{Z}^d .
\end{align}
It turns out that for the Frenkel-Kontorova problem (\ref{RR}), such local potentials exist and it is easy to check that they are given by
 \begin{align}\label{FKpotential}
 S_j(x):= V(x_j) + \frac{1}{8d} \sum_{||k-j||=1}(x_k-x_j)^2.
 \end{align}
For the general problem (\ref{RRR1}), the functions $S_j(x)$ will be required to satisfy some rather restrictive hypotheses that will be explained in detail in Section \ref{problemsetup}. Physically, the most important of these hypotheses is the {\it monotonicity} condition. It is a discrete analogue of ellipticity for a PDE. Among the more technical hypotheses is one that guarantees that the sums in expression (\ref{RRR1}) are finite. For the purpose of this introduction, it probably suffices to say that the potentials (\ref{FKpotential}) of Frenkel-Kontorova are prototypical for the $S_j(x)$ that we have in mind. \\ \indent
It is important to observe that the solutions of (\ref{RRR1}) are precisely the stationary points of the formal sum
\begin{align}\label{potentialW}
W(x):=\sum_{j\in\Z^d} S_j(x).
\end{align}
This follows because differentiation of (\ref{potentialW}) with respect to $x_i$ produces exactly equation (\ref{RRR1}) and it explains why solutions to (\ref{RRR1}) are sometimes called {\it stationary} configurations. 
\\ \mbox{} \\ \noindent
In the case that the periodic onsite potential $V(\xi)$ vanishes, the Frenkel-Kontorova equation (\ref{RR}) reduces to the discrete Laplace equation $\Delta x= 0$, for which it is easy to point out solutions. For instance, when $\xi\in \mathbb{R}$ is an arbitrary number and $\omega \in \mathbb{R}^d$ is an arbitrary vector, then the linear functions $x^{\omega, \xi}: \mathbb{Z}^d\to\mathbb{R}$ defined by
$$x_i^{\omega, \xi} :=\xi + \langle \omega, i\rangle $$
obviously satisfy $\Delta x = 0$. It moreover turns out that the $x^{\omega, \xi}$ are {\it action-minimizers}, in the sense that for every finite subset $B\subset \Z^d$ and every $y:\Z^d\to \R$ with support in $B$, it holds that 
$$\sum_{j\in \Z^d} \left( S_j(x^{\omega,\xi}+ y)- S_j(x^{\omega, \xi}) \right) \geq 0 \ .$$
Note that this sum is actually finite and can be interpreted as $W(x^{\omega, \xi}+y)-W(x^{\omega,\xi})$. 
\begin{definition}
Let $x:\mathbb{Z}^d\to\mathbb{R}$ be a $d$-dimensional configuration. We say that $\omega\in \mathbb{R}^d$ is the {\it rotation vector} of $x$ if for all $i\in \mathbb{Z}^d$, the limit 
$$\lim_{n\to \infty} \frac{x_{ni}}{n} \ \mbox{exists and is equal to} \ \langle \omega, i\rangle.$$
\end{definition}
\noindent Clearly, the rotation vector of $x^{\omega, \xi}$ is equal to $\omega$. On the other hand, in dimension $d\neq 1$, a solution to (\ref{RR}) does not necessarily have a rotation vector. An example is the hyperbolic configuration $x^h$ defined by $x^h_i=i_1 i_2\cdots i_{d-1} i_d$ which solves $\Delta x =0$. \\ 
\mbox{} \\
\noindent  In Aubry-Mather theory, one is interested, among others, in answering the following questions: given a collection of local potentials $S_j(x)$ satisfying the assumptions of Section \ref{problemsetup}, a number $\xi\in \mathbb{R}$ and a vector $\omega \in \mathbb{R}^d$, does there always exist a solution $x$ to equation (\ref{RRR1}) with rotation vector $\omega$ and initial condition $x_0=\xi$? And if so, what is the structure of the solution set? \\ \indent
A rather complete answer to these questions is known. It turns out that solutions to (\ref{RRR1}) of all rotation vectors $\omega\in\R^d$ exist. For example, it was shown by Bangert \cite{bangert87}, that when $\omega\in\R^d \backslash \Q^d$ is irrational, then there exists a unique nonempty collection of ``recurrent'' action-minimizers of rotation vector $\omega$. This is the Aubry-Mather set of rotation vector $\omega$.  It is totally ordered, but may contain ``gaps''. That is, given an arbitrary $\xi\in \R$, it may happen that the Aubry-Mather set of rotation vector $\omega$ does not contain any configuration $x$ satisfying the initial condition $x_0=\xi$. It is known that in this case, the Aubry-Mather set is actually a Cantor set. \\ 
\indent The basics of this classical theory will be reviewed in Sections \ref{birkhoffsection} and \ref{classAM} of this paper. In Section \ref{birkhoffsection}, we will study Birkhoff configurations, examples of which are the action-minimizing configurations of the Aubry-Mather sets. In Section \ref{birkhoffpersection}, we will moreover prove some new results for $d$-dimensional periodic Birkhoff configurations. In Section \ref{classAM}, we will examine minimizing configurations and for completeness, we will reprove the classical result that global minimizers of every rotation vector exist and we will examine the properties of the Aubry-Mather set. \\
\indent To investigate the existence of stationary configurations in the gaps of the Aubry-Mather sets, we propose to study the gradient flow of the formal action function, i.e. the flow of the differential equation
$$\frac{dx}{dt} = - \nabla W(x)\ .$$ 
It was shown by Gol\'e \cite{gole91} that this flow is well-defined on a suitable subspace $\X\subset \R^{\Z^d}$ of configurations that contains all Birkhoff configurations. We will prove some regularity results for the gradient flow in Section \ref{flow} and we will discuss some of its qualitative properties in Section \ref{properties_flow}. The most notable of these is a strong monotonicity property or strong parabolic comparison principle, see Theorem \ref{monotonicity psi}. \\
\indent
The principal goal of this paper is then to prove the existence of a continuous one-dimensional gradient-flow invariant family of configurations that contains the Aubry-Mather set of rotation vector $\omega$. Such an interpolating family will be called a {\it ghost circle} and denoted $\Gamma_{\omega} \subset \R^{\Z^d}$. The precise definition of a ghost circle is given in Section \ref{GC}. \\
\indent Ghost circles were already constructed for twist maps by Gol\'e \cite{gole92}. Hence, they are well-known to exist in dimension $d=1$. Gol\'e starts his construction by assuming that $\omega =\frac{q}{p} \in\Q$ is rational and that an appropriate periodic action function $W_{p,q}(x)$ is a Morse function. Under these assumptions, the existence of a periodic ghost circle follows from a combination of topological arguments and the parabolic comparison principle of the gradient flow. In Section \ref{morse_existence}, we will imitate the construction of these periodic Morse ghost circles in dimension $d\neq 1$. Each of these periodic ghost circles contains at least one global minimizer. \\
\indent Our first main result is contained in Section \ref{morse_approximations}. It generalizes results of Gol\'e \cite{gole01} on twist maps and it roughly states that for every rational $\omega  \in \Q^d$ and for every collection of potentials $S_j(x)$, one can find arbitrarily small perturbations of the $S_j(x)$ that turn the periodic action $W_{p,q}(x)$ into a Morse function. This statement is nontrivial in dimension $d\neq 1$ and it holds because of group theoretic reasons that will explained in Section \ref{birkhoffpersection}. 
  \\ 
\indent The most important technical result of this paper is nevertheless a compactness theorem for ghost circles. It is presented in Section \ref{convergence}. It says that when the rotation vectors $\omega_n$ converge to a rotation vector $\omega_{\infty}$ and the local potentials $S_j^n$ converge to potentials $S_j^{\infty}$ and there exist ghost circles $\Gamma_n$ for the potentials $S_j^n$ of rotation vector $\omega_n$, then there is a ghost circle $\Gamma_{\infty}$ for the potentials $S_j^{\infty}$ of rotation vector $\omega_{\infty}$. Moreover, a subsequence of the $\Gamma_n$ actually converges to $\Gamma_{\infty}$ in a sense to be made precise. Together, all of the above shows that there are ghost circles of every rotation vector and for arbitrary potentials. Again, they contain at least one minimizer and hence the entire Aubry-Mather set of rotation vector $\omega$. \\ 
\indent A similar compactness result was proved by Gol\'e, see \cite{gole01}, for twist maps. The proof of this ``monotone convergence theorem for ghost circles'' relies on the fact that over time, two different solutions of the gradient flow must decrease their number of intersections. Hence, this proof is purely one-dimensional. Our proof, on the other hand, only depends on a quantitative version of the parabolic comparison principle, a so-called Harnack inequality. This inequality is stated and proved in Theorem \ref{UHIP}. \\
 \indent 
As a consequence, we show in Section \ref{gap_solutions} that when the Aubry-Mather set is a Cantor set, then its gaps must either be completely foliated by minimizers, or contain at least one non-minimizing solution to (\ref{RRR1}).
\subsection{Acknowledgement} We would like to thank our colleagues of the Department of Mathematics at VU University Amsterdam for their continuous support. A large part of this paper was written during the authors' visit to the Department of Mathematics and Statistics at Boston University. This research was funded by the Dutch Science Foundation NWO.

\section{Problem setup}\label{problemsetup}
Let us at this point introduce the generalized Frenkel-Kontorova lattice recurrence relations that we want to consider in this paper. \\ 
\indent As was discussed before, we will assume that for all $j\in \mathbb{Z}^d$ there is a function $S_j$ that assigns a real value to every $d$-dimensional configuration:
$$S_j:\mathbb{R}^{\mathbb{Z}^d} \to \mathbb{R} \ .$$
These functions are required to have the conditions A-E described below and are called {\it local potentials}. \\ \indent
To formulate the first condition, let us assume that a finite subset $B\subset \Z^d$ and an $m\geq 0$ times  continuously differentiable function $s:\R^B\to\R$ are given. Then we can define a function $S:\R^{\Z^d}\to\R$ by setting $S(x):=s(x|_B)$. This is just a way of saying that $S$ depends only on the finitely many variables $x_i$ for which $i\in B$. \\
\indent Such an $S$ has some convenient properties, most notably that it is continuous in the topology of pointwise convergence:  if $x^n, x^{\infty}\in\R^{\Z^d}$ and $\lim_{n\to\infty} x^n=x^{\infty}$ pointwise, then obviously also $\lim_{n\to\infty} S(x^n)=S(x^{\infty})$. \\
\indent Moreover, it makes sense to speak of the partial derivatives of the function $S$: if $j_1,\ldots, j_k\in\Z^d$, with $0\leq k\leq m$, is a collection of lattice points, then the partial derivative
$\p_{j_1, \ldots, j_k} S: \R^{\Z^d}\to\R$ can simply be defined as 
$$\p_{j_1, \ldots, j_k} S(x):=\left\{ \begin{array}{ll} (\p_{j_1, \ldots, j_k}s)(x|_B)& \mbox{if}\ j_1, \ldots, j_k \in B, \\ 0 & \mbox{otherwise}.\end{array} \right.$$
These partial derivatives are also continuous with respect to pointwise convergence. \\ 
\indent Finally, we recall the definition $||i||:=\sum_{k=1}^d |i_k|$, for $i\in \Z^d$, and define $B_j^r:=\{k\in \Z^d\ | \ ||k-j||\leq r\}$.\\
\indent
With all this in mind, we can formulate our first condition. 
\begin{itemize}
\item[{\bf A.}] The functions $S_j$ are twice continuously differentiable and of finite range. That is, there is an $0<r<\infty$ and for every $j\in\Z^d$ there is a twice continuously differentiable function $s_j: \R^{B_j^r}\to \R$ such that $S_j(x)=s_j(x|_{B_j^r})$. 
\end{itemize}
In other words, the function $S_j$ depends only on the finitely many variables $x_k$ with $||k-j||\leq r$. Hence, $S_j(x)$ has the interpretation of the ``local energy'' of the configuration $x$ at lattice site $j$ and we think of $r$ as the finite range of the interaction. \\ 
\indent
To formulate condition B, it is convenient to introduce an action of $\mathbb{Z}^d\times\mathbb{Z}$ on $\mathbb{R}^{\mathbb{Z}^d}$ by ``shifts'':
\begin{definition}
Let $k\in \mathbb{Z}^d$ and $l\in \mathbb{Z}$. Then the shift operator $\tau_{k,l}:\mathbb{R}^{\mathbb{Z}^d}\to \mathbb{R}^{\mathbb{Z}^d}$ is defined by
$$(\tau_{k,l} x)_i := x_{i+k}+l\ . $$
\end{definition}
\noindent Clearly, the graph of $\tau_{k,l}x$, viewed as a subset of $\mathbb{Z}^d\times \mathbb{R}$, is obtained by shifting the graph of $x$ over the integer vector $(-k,l)$. This explains why the $\tau_{k,l}$ are called {\it shift operators}.
\begin{itemize}
\item[{\bf B.}] The functions $S_j$ are shift-invariant: $S_j(\tau_{k,l} x) = S_{j+k}(x)$ for all $j$,  $k$ and $l$.
\end{itemize}
In fact, invariance of the $S_j$ under $\tau_{0,1}$ just means that $S_j(x)=S_j(x+1_{\Z^d})$ for all $j$, which means that $S_j$ descends to a function on $\mathbb{R}^{\mathbb{Z}^d}/\mathbb{Z}$. Invariance of the $S_j$ under the shifts $\tau_{k,0}$ expresses the maximal spatial homogeneity of the local potentials. In fact, once one of the $S_j$ is given, for instance $S_0$, then all the others are determined.  \\
\indent The next condition ensures the growth of the $S_j$ at infinity:
\begin{itemize}
\item[{\bf C.}] The functions $S_j$ are bounded from below and coercive in the following sense: for all $k$ with $||k-j||=1$, $$\lim_{|x_{k}-x_{j}| \to \infty}  S_j(x) = \infty\ .$$
\end{itemize}
Condition C says that every function $x\mapsto S_j(x)$ is as coercive as it can possibly be under the restriction that it satisfies the periodicity condition $S_j(\tau_{0,1}x)=S_j(x)$. 
\\ \indent The following condition D is the most essential one:
\begin{itemize}
\item[{\bf D.}] The functions $S_j$ satisfy the so-called {\it monotonicity condition}:
$$\partial_{i,k}S_j  \leq 0 \ \mbox{for all} \ j \ \mbox{and all} \ i\neq k, \mbox{while} \ \partial_{i,k}S_i  < 0 \ \mbox{for all} \ ||i-k||=1 \ .$$
\end{itemize}
Condition D is also called a {\it twist condition} or {\it ferromagnetic condition}. It says that all mixed derivatives of the local potentials are non-positive, while some of them are strictly negative.\\ 
\indent 
For technical reasons we will also assume:
\begin{itemize}
\item[{\bf E.}] The $S_j$ have uniformly bounded second derivatives: there is a constant $C$ such that
$$|\partial_{i, k}S_j | \leq C\ \mbox{for all} \ i, j, k. $$
\end{itemize}
As in Section \ref{intro}, we can now look for stationary configurations corresponding to these potentials. 
\begin{definition}\label{globalstat}
A configuration $x:\mathbb{Z}^d\to\mathbb{R}$ is called a {\it stationary} point for the local potentials $S_j$ if for every finite subset $B\subset \mathbb{Z}^d$ and every configuration $y$ with support in its $r$-interior $\mathring{B}^{(r)}:= \{i\in B\ | \ B_i^r\subset B \}$, it holds that
$$\left. \frac{d}{d\varepsilon} \right|_{\varepsilon = 0} W_B(x+\varepsilon y) = 0 , $$
where $W_B:\mathbb{R}^{\mathbb{Z}^d}\to \mathbb{R}$ is defined as
\begin{align}\label{finiteaction}
W_B(x) := \sum_{j\in B} S_j(x)  .
\end{align}
\end{definition}
\noindent In fact, by differentiating $W_B$ with respect to an $x_i$ with $i\in \mathring{B}^{(r)}$, one obtains that $x$ is a stationary point for the $S_j$ if and only if it satisfies the variational monotone recurrence relation
\begin{align}\label{RRR}
\sum_{||j-i||\leq r} \partial_iS_j(x) = 0\ \mbox{for all} \ i\in \mathbb{Z}^d .
\end{align}
The goal of this paper is to find solutions of (\ref{RRR}) and while doing so, we will exploit the variational principle that underlies it. \\ \indent
By the way, (\ref{RRR}) is called {\it monotone} because condition C guarantees that the derivative of the left hand side of (\ref{RRR}) with respect to any of the $x_k$ with $k\neq i$, is non-positive, while it is strictly negative if $||k-i||=1$. \\ 
\indent Definition \ref{globalstat} moreover inspires the definition of a special type of solutions to (\ref{RRR}):
\begin{definition}\label{globalmin}
A configuration $x:\mathbb{Z}^d\to\mathbb{R}$ is called a {\it global minimizer} or {\it ground state} for the potentials $S_j$ if for every finite subset $B\subset \mathbb{Z}^d$ and every $y:\mathbb{Z}^d\to\mathbb{R}$ with support in $\mathring{B}^{(r)}$,
$$W_B(x+y) - W_B(x) = \sum_{j\in\Z^d} \left( S_j(x+y)-S_j(x) \right) \geq 0 \ .$$
\end{definition}
\noindent Clearly, global minimizers are automatically stationary and hence satisfy the recurrence relation (\ref{RRR}).
\begin{example}
It is easy to check that the Frenkel-Kontorova potentials given in (\ref{FKpotential}) satisfy conditions A-E.   In fact, the range of interaction is $r=1$, and $\p_{j,k}S_j=-\frac{1}{4d}$ for $||j-k||=1$. \\ 
\indent 
In the particular case that $V(\xi)\equiv 0$ all solutions of (\ref{RR}) are actually global minimizers. This follows because every $y\mapsto W_B(x+y)$ is strictly convex if $V(\xi)\equiv 0$ and hence only has one stationary point, which is minimizing.
\end{example}

\section{Spaces of configurations}\label{birkhoffsection}
In this section, we introduce certain spaces of configurations that are often encountered in classical Aubry-Mather theory. We will moreover study some of their properties. Most of the definitions and
results in this section are standard, but to the best of our knowledge Lemma \ref{tauaction} and Theorem \ref{maxrelper} in Section \ref{birkhoffpersection} are new. We start by recalling the following definition:
\begin{definition}
Let $x:\mathbb{Z}^d\to\mathbb{R}$ be a $d$-dimensional configuration. We say that $\omega = \omega(x) \in \mathbb{R}^d$ is the {\it rotation vector} of $x$ if for all $i\in \mathbb{Z}^d$, the limit 
$$\lim_{n\to \infty} \frac{x_{ni}}{n} \ \mbox{exists and is equal to} \ \langle \omega, i\rangle.$$
The space of configurations with rotation vector $\omega$ is denoted 
$$\X_{\omega} : = \{x:\Z^d\to \R\ | \ \omega(x)=\omega\}\ .$$
\end{definition}
\subsection{Birkhoff configurations}\label{birkhoffsubsection}
We will now introduce the concept of a well-ordered lattice configuration.\begin{definition} 
On the configuration space $\R^{\Z^d}$ we define the relations $ \leq,  <$ and $\ll$ by
\begin{itemize}
\item $x \leq y$ if $x_i \leq y_i$ for every $i\in \Z^d$. 
\item $x<y$ if $x\leq y$, but $x \neq y$. 
\item $x \ll y$ if $x_i < y_i$ for every $i\in \Z^d$. 
\end{itemize}
Similarly for $\geq$, $>$ and $\gg$. 
\end{definition}
\noindent Recall the definition of the shift operators $\tau_{k,l}:\mathbb{R}^{\mathbb{Z}^d}\to \mathbb{R}^{\mathbb{Z}^d}$. The partial orderings defined above, now allow us to make the following definition, as in for instance \cite{blank} and \cite{llave-lattices}.
\begin{definition}\label{defbirkhoff}
A configuration $x \in \R^{\Z^d}$ is called a {\it Birkhoff configuration} or a {\it well-ordered configuration}, if for all $k\in \Z^d$ and $l\in \Z$,
\begin{equation}  \label{Birkhoff} \mbox{either}\ \tau_{k.l}x \geq x \;  \; \text{or} \;  \; \tau_{k,l}x \leq x . \end{equation}
\end{definition}
\noindent Definition \ref{defbirkhoff} says that the graph of a Birkhoff configuration $x$ does not cross any of its integer translates. The space of Birkhoff configurations will be denoted $\mathcal{B}\subset \R^{\Z^d}$ and it inherits the topology of pointwise convergence.
Birkhoff configurations will play an essential role in the remainder of this paper. Birkhoff configurations of every rotation vector exist: for every $\omega \in \R^d$ the linear configuration $x^{\omega}$ defined by $x_i^{\omega}:= \langle \omega, i\rangle$ is an example.
\begin{remark}
When $h:\R/\Z\to\R/\Z$ is an orientation preserving circle homeomorphism, then it admits a {\it lift} to a strictly increasing map $H:\R\to\R$ that satisfies $H(\xi+1)=H(\xi)+1$ and $H(\xi)\!\!\! \mod \! 1 = h(\xi \!\!\! \mod \! 1)$. \\
\indent Let us now denote by $x(\xi):\Z\to\R$ the $H$-orbit of $\xi\in \R$, defined by $x(\xi)_i:=H^i(\xi)$.
Then it is clear that for $\xi_1, \xi_2\in \R$ with $\xi_1<\xi_2$, one has that $x(\xi_1)\ll x(\xi_2)$. In turn this implies that each $x(\xi)$ is a Birkhoff sequence. Thus, ordering is a very natural concept in the theory of circle homeomorphisms. 
\end{remark}
\noindent The following result is folklore and it goes back to Poincar\'e, who proved it in the case $d=1$ and in the context of circle homeomorphisms, for which it implies that circle homeomorphisms have a unique rotation number. \\ \indent For $d=1$, the proof of Lemma \ref{sequence} can be found for instance in \cite{gole01}. For completeness, we include the proof for $d>1$ here. Lemma \ref{sequence} says that the graph of a Birkhoff configuration $x$ lies uniformly close to the graph of the affine configuration $i\mapsto x_0 + \langle \omega, i\rangle$.
\begin{lemma}\label{sequence}
Let $x \in \R^{\Z^d}$ be a Birkhoff configuration. Then $x$ has a rotation vector $\omega = \omega(x)$ and
$$|x_i - x_0-\langle \omega(x), i \rangle| \leq1\ .$$
Moreover, the map $x \mapsto \omega(x), \ \mathcal{B} \to \mathbb{R}^d$ is continuous with respect to pointwise convergence. We write 
$$\mathcal{B}_{\omega}:=\{x\in \mathcal{B} \ | \ \omega(x)=\omega\}\ . $$
\end{lemma}

\begin{proof}
We will assume that the result is true for $d=1$ and we choose $i, j \in \mathbb{Z}^d$. Then the sequence $n\mapsto x_{ni+j}$ is a one-dimensional Birkhoff sequence and hence its rotation number $\omega_{i,j}$ exists and is equal to $\lim_{n\to\infty} \frac{x_{ni+j}}{n}$. Moreover, $|x_{ni+j}- x_j - \langle \omega_{i,j},n \rangle| \leq 1$. We first of all remark that $\omega_{i,j}$ does not depend on $j$, and hence can be denoted $\omega_i$. This follows because the Birkhoff property of $x$ ensures that the sequences $n\mapsto x_{ni+j}=(\tau_{j,0}x)_{ni}$ and $n\mapsto x_{ni}$ do not cross. Now denote by $e_1=(1,0,\ldots, 0)$, $e_2=(0,1,0,\ldots, 0)$, etc. the standard basis of $\Z^d$ and define $\omega:=(\omega_{e_1}, \ldots, \omega_{e_d})$. Then, 
\begin{align}\nonumber
|x_i - x_0 -  \langle \omega, i \rangle| = |x_{(i_1, i_2, \ldots, i_d)} - x_{(0, i_2, \ldots, i_d)} - & i_1 \omega_{e_1} + \ldots + x_{(0,\ldots, 0, i_d)} - x_0 - i_d \omega_{e_d} | \leq \\
 |x_{(i_1, i_2, \ldots, i_d)} - x_{(0, i_2, \ldots, i_d)} - i_1 \omega_{e_1}| + \ldots & +  |x_{(0,\ldots, 0, i_d)} - x_0 - i_d \omega_{e_d} | \leq d\ . \nonumber
\end{align}
This clearly implies that $\lim_{n\to \infty} \frac{x_{ni}}{n} = \langle \omega,i\rangle$, while the Birkhoff property of the sequence $n\mapsto x_{ni}$ then implies that in fact, $|x_i - x_0 -  \langle \omega, i \rangle| \leq 1$.\\ 
\indent The continuity of $x\mapsto \omega(x)$ follows immediately from the continuity in the one-dimensional case.
\end{proof}
 \noindent The following proposition is equally standard. In particular, it will allow us to take limits of Birkhoff configurations with rational rotation vectors in order to produce Birkhoff configurations with irrational rotation vectors. \\ 
\indent Recall the action $\tau_{0,1}:x\mapsto x+1$ on $\mathbb{R}^{\Z^d}$. It can be used to identify sequences that differ by an integer. The quotient space is denoted $\R^{\Z^d} / \mathbb{Z}$. Note that every element $[x]$ in this quotient space has a unique representative $x$ with $x_0\in [0,1)$.

\begin{proposition}\label{compactness}
Let $K \subset \R^d$ be compact and let $\mathcal{B}_K := \bigcup_{\omega \in K} \mathcal{B}_{\omega}$. Then $\mathcal{B}_K /\Z$ is compact in the topology of pointwise convergence.
\end{proposition}

\begin{proof}
By definition, $\mathcal{B}$ is closed in the topology of pointwise convergence. Moreover, by Proposition \ref{sequence}, $\mathcal{B}_K/\Z$ is a closed subset of 
$$\{[x] \in \R^{\Z^d}/\Z\ | \ x_k= x_0 + k \cdot \omega + y_k \ \mbox{with} \ ([x]_0, \omega, y_k) \in \R/\Z \times K \times [-1,1]^{\Z^d}\}\ , $$ which is compact in the topology of pointwise convergence. This follows from Tychonov's theorem. 
\end{proof}

\noindent The following corollary of the compactness of $\mathcal{B}_K/\Z$ is trivial, but it has important implications.

\begin{corollary}\label{uniform_twist}
Let $K \subset \R^d$ be compact and let $B\subset \Z^d$ be a finite subset. Assume that $s:\R^{B}/\Z\to \mathbb{R}$ is a continuous function. Then $S:\R^{\Z^d}/\Z \to \R$ defined by $S(x)=s(x|_B)$ attains its maximum and minimum values on $\mathcal{B}_{K}/\Z$. 
\end{corollary} 
\begin{proof}
This follows because such a $S$ is continuous with respect to pointwise convergence and $\mathcal{B}_{K}/\Z$ is compact. 
\end{proof}

\noindent Applied to $S=\partial_{i,k}S_i$ with $||i-k||=1$, and recalling the twist condition $\partial_{i,k}S_i<0$, Corollary \ref{uniform_twist} implies that there is a $\lambda>0$ such that $\partial_{i,k}S_i(x) < -\lambda < 0$ for all $x\in \mathcal{B}_{K}$. In other words, the twist condition D is automatically uniform on $\mathcal{B}_K$. Similarly, even if one does not impose condition E, there is a constant $C>0$ such that $|\partial_{i,k}S_j |\leq C$ for all $i, k$ and $j$, uniformly on $\mathcal{B}_K$.
 
We finish this section with a simple and well-known proposition that expresses that the number theoretical properties of the rotation vector $\omega$ of a Birkhoff configuration $x$ decide to a large extent wether $\tau_{k,l}x> x$ or $\tau_{k,l}x < x$.
\begin{proposition}\label{numbertheory}
Let $\omega\in \R^d$ and $x \in \mathcal{B}_\omega$. If $\langle \omega, k\rangle+l>0$, then $\tau_{k.l}x > x$ and if $\langle \omega, k\rangle+l<0$, then $\tau_{k.l}x < x$.
\end{proposition}

\begin{proof}
Denote by $x^{\omega}\in \mathcal{B}_{\omega}$ the linear configuration defined by $x^{\omega}_i := \langle \omega, i\rangle$. Then $\tau_{k,l}x^{\omega}-x^{\omega} = \langle k, \omega\rangle + l$. Suppose for instance that $\langle k, \omega\rangle + l>0$, that is that $\tau_{k,l}x^{\omega} \gg x^{\omega}$, but assume on the other hand that $\tau_{k,l}x \leq x$. This means that $\tau_{k,l}x-x \leq 0$ and hence also that $\tau_{k,l}^2x-\tau_{k,l}x = \tau_{k,0}(\tau_{k,l}x-x) \leq 0$. Thus, $\tau_{k,l}^2x- x = (\tau_{k,l}^2x - \tau_{k,l}x)+(\tau_{k,l}x -x )\leq 0 $, i.e. $\tau^2_{k,l}x\leq x$. By induction we then find that
$\tau_{k,l}^n x \leq x$, for every $n\geq 1$.  On the other hand, $ \tau_{k,l}^n x^{\omega} = x^{\omega} + n(\langle k, \omega\rangle + l)$. This contradicts the fact that $\sup_i |\tau_{k,l}^n(x^{\omega}-x)_i| = \sup_i |(x^{\omega}-x)_i| \leq |x_0|+2$ is uniformly bounded in $n$.
\end{proof}

\subsection{Periodicity}\label{birkhoffpersection}
It turns out convenient to consider periodic configurations.
To define these, let $p_1, \ldots, p_c\in \mathbb{Z}^d$ be $0\leq c \leq d$ linearly independent integer vectors and let $q_1, \ldots, q_c \in \mathbb{Z}$ be $c$ integers. Then we set
$$\X_{p,q} := \{ x:  \Z^d   \to  \R \ | \ \tau_{p_j,q_j}x = x \ \mbox{for all} \ j=1,\ldots, c\ \}\ .$$
We say that a configuration $x\in \mathbb{X}_{p,q}$ is {\it periodic} with periods $(p_1, q_1),\ldots, (p_c, q_c)$. The collection of periods of $\X_{p,q}$ is a lattice of rank $c$, that we denote by
$$J_{p,q}: = \left\{\sum_{j=1}^{c}m_j(p_j, q_j)\ | \ m_j\in \Z\right\} \subset \Z^d\times \Z \ .$$
An element of $\X_{p,q}$ can have a rotation vector, but this rotation vector can not be arbitrary: when $x:\Z^d\to\R$ is a configuration of rotation vector $\omega$ and $\tau_{p_j,q_j}x=x$, then $x_{np_j} = x_0- nq_j$, so that $\lim_{n\to\infty} \frac{x_{np_j}}{n} =- q_j$, that is
$$\langle \omega, p_j\rangle+ q_j=0 \ \mbox{when} \ \X_{p,q} \cap \X_{\omega} \neq \emptyset \ . $$
\noindent Another way to express this is that when $\X_{p,q}\cap \X_{\omega}\neq \emptyset$, then $J_{p,q}\subset I_{\omega}$, where the lattice $I_{\omega}$ is defined as
$$I_{\omega}: = \{(k,l)\in \Z^d\times \Z\ | \ \langle k, \omega \rangle + l = 0 \} \subset \Z^d\times \Z \ . $$
On the other hand, when $x$ has rotation vector $\omega$ and $\langle \omega, k \rangle + l = 0$, then this does not imply that $\tau_{k,l}x=x$. We therefore define 
$$\overline{\X}_{\omega} := \{x\in \X_{\omega}\ | \ \tau_{k,l}x=x\ \mbox{when} \ \langle \omega, k \rangle + l = 0\ \}\ .$$
\noindent The elements of $\overline{\X}_{\omega}$ are called {\it maximally periodic} as they have all the periods that an element of $\X_{\omega}$ can possibly have. $\overline{\X}_{\omega}$ is nonempty because it contains the linear configuration $x^{\omega}$ defined by $x^{\omega}_i=\langle \omega, i\rangle$. This is true because $\tau_{k,l}x^{\omega} = x^{\omega} + \langle \omega, k \rangle + l$.
\begin{definition}
A $\Z$-basis $(p_1,q_1), \ldots, (p_c,q_c)$ of $I_{\omega}$ is called a collection of {\it principal periods} for $\omega$. That is, 
$(p_1,q_1), \ldots, (p_c,q_c)$ are principal periods for $\omega$ if and only if $\X_{p,q}\cap\X_{\omega}=\overline{\X}_{\omega}$.
\end{definition}
\noindent Of course, a set of principal periods for $\omega\in \R^d$ always exists, but it is not unique. \\
\indent At this point, let us make some group theoretic remarks. First of all, we remind the reader that we can think of the shift operators $\tau_{k,l}$ as defining a group action of $\Z^d\times \Z$ on the space of configurations:
$$\tau: (\Z^d\times \Z)\times \R^{\Z^d}\to \R^{\Z^d}, \ ((k,l),x)\mapsto \tau_{k,l}x\ . $$
Clearly, because $\Z^d\times \Z$ is Abelian, when $\tau_{p_j, q_j}x=x$, then also $\tau_{p_j,q_j}(\tau_{k,l}x)=\tau_{k, l}x$, and thus $\tau$ leaves $\X_{p,q}$ invariant. Moreover, because the elements of $J_{p,q}$ fix all elements of $\X_{p,q}$, we have that when $x\in \X_{p,q}$ and $(k,l)=(K,L)+ \sum_jm_j(p_j,q_j)$ for certain integers $m_j$, then  $\tau_{k,l}x=\tau_{K,L}x$. This shows that $\tau$ induces an action of $(\Z^d\times \Z)/J_{p,q}$ on $\X_{p,q}$. We recall that this action is called {\it free} if for every $(k,l)\notin J_{p,q}$ and every $x\in \X_{p,q}$ it holds that $\tau_{k,l}x\neq x$. We now have the following quite obvious characterization of $\overline{\X}_{\omega}$:
 \begin{lemma}\label{tauaction}
Assume that $\X_{p,q} \cap \X_{\omega} \neq \emptyset$. Then the $\tau$-action of $(\mathbb{Z}^d\times \Z)/ J_{p,q}$ on $\X_{p,q}\cap \X_{\omega}$ is free if and only if the $(p_j,q_j)$ are principal periods for $\omega$, i.e. if and only if $\X_{p,q}\cap\X_{\omega}=\overline{\X}_{\omega}$.
 \end{lemma} 
 \begin{proof} Let us start by assuming that the $(p_j, q_j)$ are principal periods for $\omega$, that is that
 $J_{p,q} = I_{\omega}$.   
 We want to show that then the action of $(\Z^d\times \Z)/I_{\omega}$ on $\X_{p,q}\cap \X_{\omega}$ is free. But a nontrivial equivalence class in $(\Z^d\times \Z)/I_{\omega}$ is represented by an element $(k,l)$ with $\langle \omega, k\rangle + l\neq 0$ and it is clear that this inequality implies that $\tau_{k,l}x\neq x$ if $x$ has rotation vector $\omega$.\\ 
 \indent In the other direction, suppose the action is not free. Then there is a $(k,l)\notin J_{p,q}$ and an $x\in \X_{p,q}\cap \X_{\omega}$ with $\tau_{k,l}x=x$. Clearly, such $(k,l)$ must satisfy $\langle \omega,k\rangle + l = 0$, that is $(k,l)\in I_{\omega}$. Thus, $J_{p,q}\neq I_{\omega}$. 
 \end{proof}
 
\noindent The case that $\omega\in \Q^d$ is especially nice. We have the following:
\begin{proposition}
$\omega\in \Q^d$ if and only if $I_{\omega}$ has rank $d$. When $(p_1, q_1), \ldots, (p_d, q_d) \in I_{\omega}$ are linearly independent, then $\X_{p,q}$ is finite-dimensional and $\X_{p,q}\subset \X_{\omega}$. In particular, when $(p_1, q_1), \ldots, (p_d, q_d)$ are principal periods, then $\X_{p,q} = \overline{\X}_{\omega}$. 
\end{proposition}
\begin{proof}
Let us suppose that $\omega \in \Q^d$, for instance $\omega = (\frac{a_1}{b_1}, \ldots, \frac{a_d}{b_d})$ for integers $a_j$ and $b_j$. Then $\langle \omega, P_j\rangle + Q_j =0$ for $P_j:=(0, \ldots, 0, b_j, 0, \ldots, 0)$ and $Q_j=-a_j$. This shows that $I_{\omega}$ has rank $d$.  \\
\indent On the other hand, when $I_{\omega}$ has rank $d$, then we can choose linearly independent $(p_1, q_1), \ldots, (p_d, q_d) \in I_{\omega}$. If we now denote by $p$ the $d\times d$-matrix with integer coefficients $(p_1, \ldots, p_d)$ and by $q=(q_1, \ldots, q_d)\in\Z^d$ the integer vector of length $d$, then we can write the equations $\langle p_j, \omega\rangle + q_j=0$ as the matrix equality $p^T\omega+q=0$, where $p^T$ denotes the transpose of the matrix $p$. This implies that the final column in the rank-$d$ matrix $(p^T, q)$ is degenerate, i.e. that $p^T$ is invertible. In particular, $\omega=-p^{-T}q \in \Q^d$, with $p^{-T}$ the inverse transpose of the matrix $p$. \\
\indent Moreover, the fact that $p$ is invertible implies that $\X_{p,q}$ is finite-dimensional. More precisely, let us define 
$$B_p:= p([0,1)^d)\cap \Z^d\ . $$ 
Then $B_p$ is a fundamental domain for $p$, that is for every $i\in \Z^d$ there is a unique $k \in B_p$ with $k = i \!\!  \mod  p(\Z^d)$. It is not hard to show that this implies that the map $x\mapsto x|_{B_p}$ from $\X_{p,q}$ to $\mathbb{R}^{B_p}$ is an isomorphism. \\ 
\indent Thus, we have that $\dim \R^{B_p} = |B_p|=\mbox{vol}_d(p[0,1)^d)=|\det p\ \! |$ and hence $\X_{p,q}$ is finite-dimensional. \\
\indent In turn this implies that $\X_{p,q}\subset \X_{\omega}$, because any $x\in \X_{p,q}$ satisfies $\sup_{i\in \Z^d}\{x_i-\langle \omega, i\rangle\} = \sup_{i\in B_p} \{x_i-\langle \omega, i\rangle\}<\infty$ and the configuration $i\mapsto \langle \omega, i\rangle$ has rotation vector $\omega$.  \\ 
\indent If $(p_1,q_1), \ldots, (p_d, q_d)$ are principal periods, then the above implies that $\overline{\X}_{\omega} = \X_{\omega}\cap \X_{p,q}=\X_{p,q}$.
\end{proof}

 \noindent After these general considerations, let us now return to the Birkhoff configurations defined in Section \ref{birkhoffsubsection}. Let us denote the set of maximally periodic Birkhoff configurations of rotation vector $\omega$ by
 $$\overline{\mathcal{B}}_{\omega}:=\{x\in \mathcal{B}_{\omega}\ | \ \tau_{k,l}x=x\ \mbox{when} \ \langle k, \omega \rangle + l=0\ \}\ .$$ 
\noindent The following theorem expresses that periodic Birkhoff configurations are automatically maximally periodic.  
\begin{theorem}\label{maxrelper} Let $\omega\in \R^d$, denote $c:={\rm rank}_{\Z}\ \! (I_{\omega})$ and let $(p_1,q_1),\ldots, (p_c, q_c)\in I_{\omega}$ be linearly independent. Then $$\X_{p,q}\cap \mathcal{B}_{\omega} = \overline{\mathcal{B}}_{\omega}\ . $$
 \end{theorem}
 \begin{proof}
 Let $x\in \X_{p,q}\cap \mathcal{B}_{\omega}$, that is $x$ is Birkhoff, has rotation vector $\omega$ and $\tau_{p_j,q_j}x=x$ for all $j=1,\ldots, c$. We need to show that whenever $\langle \omega, k \rangle + l = 0$, then $\tau_{k,l}x=x$. So let us assume that $\tau_{k,l}x\neq x$. Because $x$ is Birkhoff, we may assume that $\tau_{k,l}x > x$: the case $\tau_{k,l}x < x$ is similar. This assumption implies that $\tau_{nk, nl}x = \tau_{k,l}^nx > x$ as well, for every $n\geq 1$. We claim that this is not possible. \\ \indent To prove this claim, we remark that there must be an $n\in \N$ and $m_1, \ldots, m_c\in \Z$ so that $n(k,l)=\sum_j m_j(p_j, q_j)$.  This is because by assumption the $(p_j,q_j)$ span a sublattice of $I_{\omega}$ of maximal rank. 
We therefore have that $\tau_{nk,nl}x = (\tau_{p_1, q_1}^{m_1}\circ \ldots \circ \tau_{p_c,q_c}^{m_c})x = x$. This is a contradiction and hence, $\tau_{k,l}x=x$. 
 \end{proof}

\noindent In dimension $d=1$, Theorem \ref{maxrelper} simply says that a Birkhoff configuration of period $(np, nq)$ automatically has period $(p, q)$. That is, the period $(p,q)$ of a one-dimensional Birkhoff configuration can be chosen relatively prime. Theorem \ref{maxrelper} is the $d$-dimensional variant of this statement. \\ 
\indent In spite of Theorem \ref{maxrelper}, it should be remarked that in general, $\overline{\mathcal{B}}_{\omega}\neq \mathcal{B}_{\omega}$, that is not all Birkhoff configurations of rotation vector $\omega$ are periodic. Counterexamples are easy to find.

\section{Classical Aubry-Mather theory}\label{classAM}
We are now ready to discuss the most well-known results of classical Aubry-Mather theory in the context of lattice equations. The concepts and results of this section are widely known, but we chose to present them in a perhaps slightly unconventional manner.

\subsection{Fully periodic minimizers}\label{periodicminimizerssection}
Throughout Section \ref{periodicminimizerssection}, we will assume that $\omega\in \Q^d$ and $(p_1,q_1), \ldots, (p_d, q_d) \in \Z^d\times \Z$ are linearly independent. \\
\indent We are interested in solutions to (\ref{RRR}) that lie in $\X_{p,q}$. We start by noting the presence of a variational structure. Recalling the definition $B_{p}=\mathbb{Z}^d \cap p([0,1)^d)$, one has
\begin{proposition}
A configuration $x  \in \mathbb{X}_{p,q}$ solves (\ref{RRR}) if and only if it is a stationary point of the {\it periodic action function}
\begin{align}\label{WsumS}
W_{p,q}: \mathbb{X}_{p,q}\to \mathbb{R}\  \mbox{defined by}\ W_{p, q}(x) := W_{B_p}(x)= \!\! \sum_{j\in B_{p}} S_j(x) \ , 
\end{align}
with respect to variations in $\X_{p,q}$.
\end{proposition}
\begin{proof}
We start by recalling the shift-invariance of the local potentials, condition B, which says that $S_{j+k}(x)=S_j(\tau_{k,l}x)$ for all $k$ and $l$ and all $x\in\R^{\Z^d}$. Differentiation of this identity with respect to $x_i$ then gives that $\p_{i}S_{j+k}(x)=\p_{i-k}S_{j}(\tau_{k,l}x)$.  These equalities respectively imply that for $x\in\X_{p,q}$ it holds that $S_{j+pk}(x)=S_{j}(x)$ and $\p_iS_{j+pk}(x)=\p_{i-pk}S_j(x)$ for all $k\in \Z^d$. \\ 
\indent Now let $x\in\X_{p,q}$, choose an $i\in\Z^d$ and define $e_i\in\X_{p,0}$ by letting $(e_i)_j=1$ if $j=i \! \mod p(\Z^d)$ and $(e_i)_j=0$ otherwise. Then $x+e_i\in\X_{p,q}$ and
$$\left. \frac{d}{d\varepsilon}\right|_{\varepsilon=0} \! W_{p,q}(x+\varepsilon e_i) = \sum_{k\in \Z^d} \sum_{j\in B_p} \p_{i+pk} S_j(x) = \sum_{k\in \Z^d} \sum_{j\in B_p} \p_{i} S_{j-pk}(x) = \sum_{j\in\Z^d}\p_iS_j(x) . $$ 
Of course all these sums are finite. 
\end{proof}
 
\noindent Note that $W_{p,q}=W_{B_p}$ is actually well defined for any $x\in \R^{\Z^d}$, but in this section, we restrict it to a function on $\X_{p,q}$. As such, it is shift-invariant:
\begin{lemma}[Shift-invariance]\label{shift invariance}
The function $W_{p,q}$ is $\tau$-invariant: for $x\in\X_{p,q}$ and $(k, l)\in\Z^d\times \Z$ arbitrary, it holds that $W_{p.q}(\tau_{k, l}x)=W_{p.q}(x)$. 
\end{lemma}
\begin{proof}
In general, $S_{j+k}(x)=S_{j}(\tau_{k,l}x)$, so that if $x\in\X_{p,q}$, then $S_{j+pk}(x)=S_j(x)$ for all $k\in\Z^d$. Thus, $W_{p,q}(\tau_{k,l}x)=\sum_{j\in B_p} S_j(\tau_{k,l}x) = \sum_{j\in B_p} S_{j+k}(x) = \sum_{j\in k+B_p} S_{j}(x)= \sum_{j\in B_p} S_{j}(x) = W_{p,q}(x)$. The fourth equality follows as both $B_p$ and $k+B_p$ are fundamental domains of $\Z^d/p(\Z^d)$, so that for every $j\in k+B_p$ there is a unique $i\in  B_p$ for which $i=j\! \mod p(\Z^d)$. \end{proof}
 \begin{theorem}[Existence]\label{existence}
The action $W_{p,q}$ attains its minimum on $\mathbb{X}_{p,q}$.  
\end{theorem}
\begin{proof}
Since $W_{p,q}(\tau_{0,1}x)=W_{p,q}(x)$ for $x\in\X_{p,q}$, clearly $W_{p,q}$ descends to a function on $\mathbb{X}_{p,q}/\mathbb{Z}$. Every element in this quotient space has a representative $x$ with $x_0\in [0,1]$. \\
\indent Choose a cube $C_N=\{i\in\Z^d\ | \ |i_k|\leq N \ \mbox{for all} \ k=1,\ldots d\}$ that contains $B_p$ and choose a $k\in \Z^d$ and an $n\in\N$ such that $k+nB_p$ in turn contains $C_N$. Moreover, remember that $W_{k+nB_{p}}=n^dW_{p,q}$ on $\X_{p,q}$. \\
\indent The coercivity of the $S_i$, condition C, implies that for all $j$ with $||j||=1$, it holds that for $x\in\X_{p,q}$ with $x_0\in [0,1]$, we have that $\lim_{|x_j|\to \infty} W_{k+nB_{p}}(x) = \infty$. And hence by induction that for all $j\in C_N$ and $x\in\X_{p,q}$ with $x_0\in [0,1]$, it holds that $\lim_{|x_j|\to\infty} W_{k+nB_{p}}(x)=\infty$. Because $B_p\subset C_N$ and $W_{p,q}=n^{-d}W_{k+nB_{p}}$, this means in particular that for all $j\in B_p$ and $x\in \X_{p,q}$ with $x_0\in [0,1]$ it holds that $\lim_{|x_j|\to\infty} W_{p,q}(x)=\infty$. Hence, $W_{p,q}$ attains its minimum on $\mathbb{X}_{p,q}$. 
\end{proof}

\noindent The configurations that minimize $W_{p,q}$ on $\X_{p,q}$ will be called $p,q$-minimizers. Note that other extremal points of $W_{p,q}$ in $\mathbb{X}_{p,q}$, such as saddle points, may also exist. Under certain mild conditions their existence will be proved later in this paper. \\
\indent The following lemma is well-known. We took the proof from \cite{llave-valdinoci}.

\begin{lemma}[Minimum - maximum property]\label{maxminprop}
Assume the periodic configurations $x,y  \in \mathbb{X}_{p,q}$ are $p,q$-minimizers. Then also $m:=\min\{x,y\}$ and $M:=\max\{x,y\}$ are $p,q$-minimizers.
\end{lemma}

\begin{proof}
It is obvious that $m,M  \in \mathbb{X}_{p,q}$. Write $\alpha:=M-x$ and $\beta:=m-x$ and observe that $\alpha >0$, $\beta<0$, while $\mbox{supp}(\alpha) \cap \mbox{supp}(\beta) = \emptyset$ and $y=M+m-x=\alpha + m=\alpha + \beta + x$. The proof is done, if we show that $$W_{p,q}(x)+W_{p,q}(y)\geq W_{p,q}(M)+W_{p,q}(m).$$
This is the same as showing $$W_{p,q}(x+\alpha+\beta)-W_{p,q}(x+\alpha)-W_{p,q}(x+\beta)+W_{p,q}(x) \geq 0.$$
The left hand-side of this inequality can be put in integral form as 
 $$\int_0^1\int_0^1\frac{\p^2}{\p t\p s} W_{p,q}(x+\alpha t+\beta s)ds dt =  \sum_{i,k\in \Z^d}\sum_{j \in B_{p}}\left( \int_0^1\int_0^1 \partial_{i,k} S_j(x + t\alpha+ s\beta)dsdt \right) \alpha_i\beta_k \ .$$
Since $\mbox{supp}(\alpha) \cap \mbox{supp}(\beta) = \emptyset$, we have that $\alpha_i\beta_i=0$ for all $i$. Moreover, the twist condition $\partial_{i,k}S_{j}\leq 0$ for all $i\neq k$ and the inequalities $\alpha_i\beta_k\leq 0$, guarantee that the remaining terms in the sum are nonnegative.
\end{proof}
\noindent This is now used to prove the following famous lemma:
\begin{lemma}[Aubry's lemma]\label{Aubry's lemma}
Assume the configurations $x \neq y  \in \mathbb{X}_{p,q}$ are $p,q$-minimizers. Then either $x \ll y$ or $y \ll x$.
\end{lemma}

\begin{proof}
We pursue a proof by contradiction. Denote again $m:=\min\{x,y\}$. Suppose that for instance that $m<x$ but that is not true that $m \ll x$. The case that $m<y$ and not $m\ll y$ is similar. The assumption implies that there are indices $i,k \in \mathbb{Z}^d$ with $||i-k||=1$ such that $m_i = x_i$ and $m_k<x_k$. Now we compute
\begin{align}\nonumber
&\sum_{j\in\Z^d}\left(\p_iS_j(x) - \p_iS_j(m)\right) =  \sum_{j\in\Z^d}\int_0^1 \frac{d}{dt} \partial_iS_j(tx+(1-t)m) dt = \\ & \sum_{j,l \in\Z^d} \left( \int_0^1\partial_{i,l} S_j(tx + (1-t)m)dt \right)  (x_l - m_l)\ . \nonumber
\end{align}
Recall that $x_i=m_i$, while, by the twist condition, for every $l\neq i$, it holds that $\p_{i,l}S_j\leq 0$ and $(x_l-m_l)\geq 0$. Thus, every term in the above sum is nonpositive. But for the $k$ chosen above, $\partial_{i,k}S_i<0$, while $x_k-m_k>0$. This proves that $\sum_j\p_iS_j(x)\neq\sum_j\p_iS_j(m)$. This contradicts the fact that, by the lemma above, both $m$ and $x$ are $p,q$-minimizers and must therefore both be stationary.
\end{proof}

\begin{corollary}\label{Birkhoff minimizers}
Periodic minimizers are Birkhoff configurations.
\end{corollary}

\begin{proof}
Let $x  \in \mathbb{X}_{p,q}$ be a minimizer. Then for any $k \in \Z^d$ and $l \in \Z$, we have that
$W_{p,q}(x)=W_{p,q}(\tau_{k,l}x)$ by the invariance property of $W_{p,q}$. This shows that also
 $\tau_{k,l}x$ is a minimizer, whence, by the previous corollary, either $\tau_{k,l}x \ll x$, $\tau_{k,l}x=x$ or $\tau_{k,l}x \gg x$. In particular, $x$ is a Birkhoff configuration.
\end{proof}

\begin{lemma}\label{birkhoffperiodic}
Let $n\in\N$. Every $p,q$-minimizer is an $np,nq$-minimizer and vice versa.
\end{lemma}

\begin{proof}
Assume that $x$ is an $np,nq$-minimizer, that is a minimizer of $W_{np,nq}$ on $\X_{np,nq}$. Then, by Aubry's lemma, $x\in \mathcal{B}_{np,nq}$. Theorem \ref{maxrelper} now implies that $\mathcal{B}_{np,nq}=\mathcal{B}_{p,q}$, so actually $x\in \mathcal{B}_{p,q}\subset \X_{p,q}$. Note now that on $\X_{p,q}$ it holds that $W_{np,nq}=n^dW_{p,q}$ and let $y\in \X_{p,q}\subset \X_{np,nq}$. Then $W_{p,q}(x)=n^{-d}W_{np,nq}(x)\leq n^{-d}W_{np,nq}(y)=W_{p,q}(y)$. Thus, $x$ is a  $p,q$-minimizer. 
\\ \indent In the other direction, if $x$ is a $p,q$-minimizer and $y$ is an $np,nq$-minimizer, then $y\in \mathcal{B}_{p,q}$ and $W_{np,nq}(x)=n^dW_{p,q}(x)\leq n^dW_{p,q}(y)=W_{np,nq}(y)$, that is $x$ is an $np,nq$-minimizer.
\end{proof}

\noindent The following result shows that $p,q$-minimizers are global minimizers. Recall that $x$ is called a global minimizer if for every finite set $B\subset \Z^d$ and every $y$ with support in its $r$-interior $\mathring{B}^{(r)}$, one has that $W_B(x+y)\geq W_B(x)$, with $W_B(x):=\sum_{j\in B}S_j(x)$.

\begin{theorem}\label{global}
Periodic minimizers are global minimizers.
\end{theorem}

\begin{proof} Let $x \in \X_{p,q}$ be a $p,q$-minimizer. If $x$ is not a global minimizer, then there exists a finite set $B \subset \Z^d$ and a configuration $y$ with $\mbox{supp}(y) \subset \mathring{B}^{(r)}$, such that $W_B(x+y) < W_B(x)$. Since $B$ is finite, there exist a $k\in\Z^d$ and an $n \in \N$ such that $\mbox{supp}(y)\subset B \subset k+B_{np}$. Now define $\tilde y\in \X_{np,nq}$ by setting $\tilde y_i=y_j$ when $j$ is the unique point in $k+B_{np}$ for which $j=i\! \mod np(\Z^d)$. In other words, $\tilde y$ is the $np$-periodic extension of $y|_{k+B_{np}}$. Then we conclude that 
$$W_{np,nq}(x+\tilde y) - W_{np,nq}(x) = W_{k+nB_p}(x+y)-W_{k+nB_p}(x)= W_{B}(x+y)-W_B(x)<0, $$ 
so $x$ is not $np,nq$-minimizer. This contradicts Lemma \ref{birkhoffperiodic}.
\end{proof}

\noindent Perhaps surprisingly, to prove the converse one needs to be slightly more ingenious. We have not found this statement anywhere in the literature:

\begin{theorem}\label{converse}
If $x\in \X_{p,q}$ is a global minimizer, then it is a $p,q$-minimizer.
\end{theorem}
\begin{proof}
Suppose that $x\in \X_{p,q}$ is not a $p,q$-minimizer. We will prove that this implies that $x$ is not a global minimizer. Our assumption means that there is a $y\in \X_{p,q}$ for which $0 < \varepsilon:= W_{p,q}(x) -W_{p,q}(y)$. This in turn implies that $W_{np,nq}(x)-W_{np,nq}(y) = n^d\varepsilon$. \\
\indent By periodicity, we may assume that $x\ll y$. Let us now define, for $n\in \N$, the configurations $x\leq y^n \leq y$ by
$$y^n_i=\left\{ \begin{array}{cc} y_i & \mbox{if} \ i\in \mathring{B}_{np}^{(r)} \\ x_i & \mbox{otherwise} \end{array} \right.$$
Here, $\mathring{B}_{np}^{(r)}$ is the $r$-interior of $B_{np}$. By definition, $y^n$ is a variation of $x$ with support in this $r$-interor. It now holds that 
\begin{align}\nonumber
W_{B_{np}}(x)-W_{B_{np}}(y^n) \!= & \!\! \sum_{j\in B_{np}} \!\! \left(S_j(x)-S_j(y^n)\right)\!\! =\!\! \sum_{j\in B_{np}}\!\! \left(S_j(x)-S_j(y)\right) + \left( S_j(y)-S_j(y^n)\right) \\
\nonumber =& \  n^d\varepsilon +\!\! \sum_{j\in B_{np}} \!\! \left( S_j(y)-S_j(y^n)\right)\ .
\end{align}
Because the support of $y-y^n$ is contained in $\Z^d\backslash \mathring{B}_{np}^{(r)}$ and the range of interaction of the $S_j$ is equal to $r$, the number of nonzero terms in the above sum is at most $(2r)^{2d+1}|\p B_{np}| \leq E n^{d-1}$, where $E$ is a constant depending only on $r, d$ and $p$.\\
\indent
Moreover, by compactness of $[x,y]:=\{z\ | \ x\leq z\leq y\}$, there is a constant $e>0$ so that $|S_j(y)|, |S_j(y^n)|<e$. This then implies that $$W_{B_{np}}(x)-W_{B_{np}}(y^n)\geq \varepsilon n^d - 2Ee \cdot n^{d-1}\ .$$ 
Choosing $n$ large enough, we see that $x$ is not a global minimizer.
\end{proof}

\subsection{Nonperiodic minimizers}\label{nonpermin}
In this section, we show that global minimizers of all rotation vectors exist. They are constructed as limits of periodic minimizers. Moreover, we show that they satisfy a certain pairwise regularity. The results in this section are standard.

\begin{lemma}\label{closed}
The set of global minimizers is closed in the topology of pointwise convergence.
\end{lemma}

\begin{proof} 
Assume that $x^{n}$ is a sequence of global minimizers converging pointwise to $x^{\infty}$. Let $B \subset \Z^d$ be a finite set and $y$ a configuration with support in $\mathring{B}^{(r)}$. Then
\begin{align}\label{minblabla}
W_B(x^{n}+y)-W_B(x^{n})\geq 0\ .
\end{align}
But $W_B$ is continuous with respect to pointwise convergence, so that taking the limit for $n\to \infty$ of equation (\ref{minblabla}), we find that $W_B(x^{\infty}+y)-W_B(x^{\infty})\geq 0$. So $x^{\infty}$ is a global minimizer.
\end{proof}

\begin{theorem}\label{quasi-periodic birkhoff minimizers}
For all rotation vectors $\omega \in \R^d$ and all local potentials $S_j$, there exists a global minimizer $x \in \overline{\mathcal{B}}_\omega$. 
\end{theorem}

\begin{proof}
For any $\omega \in \R^d$, we can take a sequence $\omega_{n} \in \mathbb{Q}^d$, such that $\lim_{n \to \infty} \omega_{n} = \omega$, while $\langle \omega_n, k\rangle+l=0$ for all the $k$ and $l$ for which $\langle \omega, k\rangle+l=0$. We take a corresponding sequence $(p_n, q_n)$ of principal periods for which $\omega_n:=-p_n^{-T}q_n$. By Theorems \ref{existence} and \ref{global}, there exists a global minimizer $x^{n} \in \mathcal{B}_{p_n, q_n} = \overline{\mathcal{B}}_{\omega_{n}}$. In particular, $x_n$ has rotation vector $\omega_n$ and satisfies $\tau_{k,l}x^{n}=x^{n}$ for all $k$ and $l$ for which $\langle \omega, k\rangle + l =0$. Because the $\omega_{n}$ and $\omega$ lie in some compact subset $K$ of $\R^d$, Proposition \ref{compactness}, guarantees that there is a subsequence  of the $x^{n}$ that converges pointwise to a Birkhoff configuration $x^{\infty}\subset \mathcal{B}_K$. By continuity of the rotation vector $x \mapsto \omega(x)$, see Proposition \ref{sequence}, $x^{\infty}$ actually has rotation vector $\omega$. Moreover, the limit $x^{\infty}$ will have the same periodicities: denoting the converging subsequence also by $x^n$, the continuity of $\tau_{k,l}$ implies that $\tau_{k,l}x^{\infty}=\tau_{k,l}(\lim_{n\to\infty} x^n) = \lim_{n\to\infty} \tau_{k,l} x^n = \lim_{n\to\infty} x^n = x^{\infty}$ for all $k$ and $l$ with $\langle \omega, k\rangle + l =0$. Finally, $x^{\infty}$ is a global minimizer by Theorem \ref{global} and Lemma \ref{closed}.
\end{proof}

\noindent The following result expresses the regularity of pairwise comparable stationary solutions. It is the analogue of a Harnack inequality for elliptic PDEs.
 \begin{theorem}[Elliptic Harnack inequality]\label{ellharn}
Let $x<y$ be two Birkhoff configurations with rotation vector in the compact set $K\subset \R^d$.  Suppose that $x$ and $y$ are stationary for the local potentials $S_j$. Then there is a constant $\delta$, depending only on $K$ and $||i-k||$, such that for all $i$ and $k$,
$$(y_k - x_k) \leq \delta (y_i-x_i) \ .$$
In particular, if $x<y$, then $x\ll y$.
 \end{theorem}
 \begin{proof}
By interpolation: let $x$ and $y$ be stationary and Birkhoff and let $i, k\in \Z^d$ and assume first that $||i-k||=1$. Choose a $B$ with $i \in \mathring{B}^{(r)}$ and recall the definition $W_B(x)=\sum_{j\in B}S_j(x)$. Then, by stationarity,
\begin{align}\nonumber
0 =& \partial_iW_B(y)-\partial_i W_B(x) =  \int_0^1\frac{d}{d\tau} \left( \sum_{||i-j||\leq r} \partial_iS_j(\tau y+(1-\tau)x)\right) d\tau = \\ \nonumber
 & \sum_{||i-j||\leq r, ||j-l||\leq r} \!\! \left( \int_0^1\partial_{i,l}S_j(\tau y+(1-\tau)x)d\tau \right) (y_l-x_l) \ .
\end{align}
Since, by the twist condition C, the only possibly positive terms on the right hand side are the $\left(\partial_{i,i}S_j(\tau y+(1-\tau)x)\right) (y_i-x_i)$, the right hand side is less than or equal to
$$\left(\int_0^1\!\! \sum_{||i-j||\leq r}\!\! \partial_{i,i}S_j(\tau y \! + \! (1\! -\! \tau)x)d\tau \! \right) \!  (y_i-x_i) + \left(\int_0^1 \!\! \sum_{||k-i||= 1} \!\! \partial_{i,k}S_i(\tau y \! + \! (1\! - \! \tau)x)d\tau \! \right) \!  (y_k-x_k) \ .
$$
Now, because $x$ and $y$ are Birkhoff, so is every $\tau y + (1-\tau)x$ and hence by Corollary \ref{uniform_twist}, there are constants $\lambda, C>0$, depending only on the compact set $K$, such that for all $j$ and all $||i-k||=1$, it holds that $\partial_{i,k}S_i < -\lambda$, while  $\partial_{i,i}S_j<C$ for all $i$ and $j$. Thus,
$$0 \leq (2r)^dC(y_i-x_i) - 2d \lambda (y_k-x_k)\ .$$
This proves the theorem for $||i-k||=1$ with $\delta = \delta_1 :=(2r)^dC/2d\lambda$. For $||i-k||>1$, the result then follows by induction and it holds for $\delta = \delta_1^{||i-k||}$.
\end{proof}

\subsection{Aubry-Mather sets}
We make the following definition:
\begin{definition}\label{defmatherset}
An {\it Aubry-Mather set} $\mathcal{M}\subset \R^{\Z^d}$ is a collection of configurations with the following properties
\begin{itemize}
\item $\mathcal{M}$ is nonempty and closed under pointwise convergence
\item $\mathcal{M}$ is strictly ordered, i.e. for every $x,y \in \mathcal{M}$, $x \ll y$, $x=y$ or $x \gg y$
\item $\mathcal{M}$ is shift-invariant: if $x\in \mathcal{M}$, then for every $(k,l) \in \Z^d \times \Z$, also $\tau_{k,l}x \in \mathcal{M}$
\item Every $x\in \mathcal{M}$ is a global minimizer of the variational recurrence relation (\ref{RRR})
\item $\mathcal{M}$ does not contain any strictly smaller set with the properties listed above
\end{itemize} 
\end{definition}
\noindent The strict ordering and the shift-invariance of an Aubry-Mather set $\mathcal{M}$ imply that any configuration $x\in \mathcal{M}$ is Birkhoff and hence has a rotation vector $\omega=\omega(x)$. The ordering of $\mathcal{M}$ moreover implies that this rotation vector is independent of the choice of $x\in \mathcal{M}$, that is $\omega = \omega(\mathcal{M})$ and thus, $\mathcal{M} \subset \mathcal{B}_\omega$. 
\\ \indent Recall that Theorem \ref{quasi-periodic birkhoff minimizers} states that for every rotation vector $\omega$ there exists a minimizer $x\in \mathcal{B}_{\omega}$ for which $\tau_{k,l}x=x$ as soon as $\langle \omega, k\rangle + l =0$. This in fact implies that a certain Aubry-Mather set $\mathcal{M}(x)\subset \mathcal{B}_\omega$ exists. This $\mathcal{M}(x)$ is constructed as follows. 
\noindent One starts by defining the collection $\widetilde{\mathcal{M}}(x)\subset \mathcal{B}_{\omega}$ as the closure with respect to pointwise convergence of the set of translates of $x$:
$$\widetilde{\mathcal{M}}(x):=\overline{ \{ \tau_{k,l}x \ | \ (k,l)\in \Z^d\times \Z\ \} }. $$
This is almost an Aubry-Mather set:
\begin{lemma}\label{almostMatherset}
Let $x\in \mathcal{B}_{\omega}$ be an action-minimizer with the property that $\tau_{k,l}x=x$ when $\langle \omega, k\rangle + l =0$. Then $\widetilde{\mathcal{M}}(x)$ is nonempty, closed, strictly ordered, shift-invariant and consists of minimizers. Moreover, for every $y\in \widetilde{\mathcal{M}}(x)$ it holds that $\tau_{k,l}y= y$ as soon as $\langle \omega, k\rangle + l = 0$. When $\omega\in \Q^d$, then $\widetilde{\mathcal{M}}(x)$ is an Aubry-Mather set.
\end{lemma}
\begin{proof} By definition, $\widetilde{\mathcal{M}}(x)$ is nonempty and closed. \\ \indent We note that when $x$ is a minimizer, then so is $\tau_{k,l}x$ and because any pointwise limit of minimizers is a minimizer itself, by Lemma \ref{closed}, we see that $\widetilde{\mathcal{M}}(x)$ consists of minimizers only. \\ \indent When $y\in \widetilde{\mathcal{M}}(x)$, say $y = \lim_{n\to \infty} \tau_{k_n, l_n}x$, then the continuity of $\tau_{k,l}$ implies that $\tau_{k,l}y = \tau_{k,l}(\lim_{n\to\infty} \tau_{k_n, l_n}x) = \lim_{n\to \infty} \tau_{k,l} (\tau_{k_n, l_n}x) = \lim_{n\to \infty} \tau_{k+k_n, l+l_n}x$ and thus, $\widetilde{\mathcal{M}}(x)$ is shift-invariant. \\ \indent The fact that $x$ is a Birkhoff configuration means that the collection $\{\tau_{k,l}x\ | \ (k,l)\in\Z^d\times \Z\ \}$ is ordered. Now let $y$ and $z$ be elements of $\widetilde{\mathcal{M}}(x)$, say $y = \lim_{n\to\infty}\tau_{k_n,l_n}x$ and $z=\lim_{n\to\infty}\tau_{K_n,L_n}x$. We claim that $y\leq z$ or $z\leq y$. If not, then there are $i,k$ with $y_i<z_i$ and $y_k>z_k$. The pointwise convergence then implies that there are $n$ and $N$ so that $(\tau_{k_n,l_n}x)_i<(\tau_{K_N,L_N}x)_i$ and $(\tau_{k_n,l_n}x)_k>(\tau_{K_N,L_N}x)_k$. This is a contradiction. The second conclusion of Theorem \ref{ellharn} now implies that $y\ll z$, $y=z$ or $y\gg z$, that is $\widetilde{\mathcal{M}}(x)$ is strictly ordered. \\ \indent The penultimate conclusion of the lemma follows from the continuity of $\tau_{k,l}$ and the fact that $\tau_{k,l}x=x$ when $\langle \omega, k\rangle + l = 0$. Namely, for such $k$ and $l$ and for $y\in \widetilde{\mathcal{M}}(x)$, say $y=\lim_{n\to\infty}\tau_{k_n,l_n}x$, we have that $\tau_{k,l}y=\tau_{k,l}(\lim_{n\to\infty} \tau_{k_n,l_n}x) =  \lim_{n\to\infty} \tau_{k,l}(\tau_{k_n,l_n}x) = \lim_{n\to\infty} \tau_{k_n,l_n}(\tau_{k,l}x) = \lim_{n\to\infty} \tau_{k_n,l_n}x = y$.\\
\indent Finally, when $\omega\in \Q^d$, then our assumptions imply that $x$ is periodic, say $x\in\X_{p,q}$, with $(p,q)$ a collection of principal periods for $\omega$. This implies that the $\tau$-orbit of $x$ is finite. Thus, $\widetilde{\mathcal{M}}(x)$ is equal to this single $\tau$-orbit and cannot contain any proper nonempty $\tau$-invariant subset.
\end{proof}

\noindent It is clear from the proof of Lemma \ref{almostMatherset}, that when $\omega\in \Q^d$, then every Aubry-Mather set is finite and consists of the translates of one periodic minimizer. Thus, the Aubry-Mather sets of rational rotation vector do not need to be unique. \\ \indent On the other hand, when $\omega\in \R^d\backslash \Q^d$ is irrational, then $\widetilde{\mathcal{M}}(x)$ may fail to be an Aubry-Mather set. Then one replaces $\widetilde{\mathcal{M}}(x)$ by its {\it recurrent} subset
$$\mathcal{M}(x) := \{ y \in \widetilde{\mathcal{M}}(x) \ | y = \lim_{n\to\infty} \tau_{k_n,l_n}y \ \mbox{for a sequence} \ (k_n, l_n)\ \mbox{with} \ \langle \omega, k_n\rangle + l_n \neq 0\ \}\ .$$
Before proving that this $\mathcal{M}(x)$ is indeed an Aubry-Mather set, let us define for a configuration $y\in \widetilde{\mathcal{M}}(x)$, the configurations
$$y^-:=\sup \{\tau_{k,l}y\ll y\}\ \mbox{and}\  y^+:=\inf \{ \tau_{k,l}y \gg y\} \ .$$
We remark that, by definition, $y\in \mathcal{M}(x)$ if and only if $y=y^-$ or $y=y^+$, or both. We now have the following technical result:
\begin{proposition}\label{yminus}
For $y, z\in \widetilde{\mathcal{M}}(x)$ it holds that $y^-= \sup \{\tau_{k,l}z\ll y^-\}$ and $y^+= \inf \{\tau_{k,l}z \gg y^+\}$. 
\end{proposition}
\begin{proof}
Let us prove the first equality: the proof of the second one is similar. We denote $z^-(y^-):= \sup \{\tau_{k,l}z\ll y^-\}$ and we argue by contradiction. That is, we suppose that $z^-(y^-) \neq y^-$, and hence,  that $z^-(y)\ll y^-$. Then, because $y^-$ can be approximated from below by translates of $y$ by definition, there are $k$ and $l$ so that $z^-(y^-) \ll \tau_{k,l}y \ll y^-$. This implies that $\tau_{k,l}y\ll y$ and in view of Proposition \ref{numbertheory}, we must therefore have that $\langle \omega, k\rangle + l <0$. Applying $\tau_{-k,-l}$ to the inequality $z^-(y^-) \ll \tau_{k,l} y$, we obtain that $\tau_{-k,-l} z^-(y^-) \ll y$. But because $\langle \omega, -k\rangle -l >0$, we must also have that $z^-(y^-) \ll \tau_{-k,-l} z^-(y^-)$. Hence, $z^-(y^-) \ll \tau_{-k,-l} z^-(y^-) \ll y$. But this contradicts the definition of $z^-(y^-)$, because by continuity of $\tau_{-k,-l}$, if $z^-(y^-) = \lim_{n\to\infty} \tau_{k_n,l_n} z$, then also $\tau_{-k,-l} z^-(y^-) = \lim_{n\to\infty} \tau_{-k+k_n,-l+l_n} z$ is a limit of translates of $z$ that lie below $y^-$. 
\end{proof}
\noindent We are now ready to prove:
\begin{theorem}\label{mathersettheorem}
When $\omega\in\R^d/\Q^d$, then $\mathcal{M}(x)$ is the unique Aubry-Mather set contained in $\widetilde{\mathcal{M}}(x)$. 
\end{theorem}
\begin{proof}
Proposition \ref{yminus} says that any $y\in \mathcal{M}(x)$ is a limit point of the $\tau$-orbit of any $z\in\widetilde{\mathcal{M}}(x)$. Thus, any nonempty, shift-invariant closed subset of $\widetilde{\mathcal{M}}(x)$ should contain $\mathcal{M}(x)$. It remains to show that $\mathcal{M}(x)$ is nonempty, shift-invariant and closed.\\ 
\indent First of all, Proposition \ref{yminus} applied to $z=y^-$ and $z=y^+$ respectively, says that $(y^-)^-=y^-$ and $(y^+)^+=y^+$, i.e. that $y^-$ and $y^+$ are recurrent. This shows that $\mathcal{M}(x)$ is nonempty. \\
\indent 
Shift-invariance of $\mathcal{M}(x)$ follows from the continuity of $\tau_{k,l}$: when $y=\lim_{n\to \infty}\tau_{k_n,l_n}y$, then $\tau_{k,l}y=\lim_{n\to\infty}\tau_{k_n,l_n}(\tau_{k,l}y)$.\\
\indent To prove that $\mathcal{M}(x)$ is closed, assume that $\lim_{n\to \infty} y_n =y$ pointwise for a sequence $y_n$ of recurrent configurations. When the limit $y$ is not recurrent, then $y^- \ll y \ll y^+$, so that there is an $n$ for which $y^- \ll y_n \ll y^+$. But $y_n$ is recurrent, hence $y_n\neq y$, while by Proposition \ref{yminus}, $y_n$ can be approximated by translates of $y$. Hence, there are $k$ and $l$ such that $y^-\ll \tau_{k,l}y\ll y^+$ and $\tau_{k,l}y\neq y$. This contradicts the definition of $y^-$ or $y^+$.
\end{proof}
\begin{remark} A theorem of Bangert \cite{bangert87} in the case of elliptic PDEs, states that when $\omega\in\R^d\backslash \Q^d$, then the recurrent subset actually does not depend on the choice of the Birkhoff minimizer $x\in \mathcal{B}_{\omega}$. In other words, that when $x, y\in \mathcal{B}_{\omega}$ are such that $\tau_{k,l}x=x$ and $\tau_{k,l}y=y$ whenever $\langle \omega, k\rangle + l=0$, then $\mathcal{M}(x)=\mathcal{M}(y)$. 
\\
\indent The proof of this theorem is nontrivial. The essence of it lies in proving an Aubry lemma for recurrent minimizers, that is to show that if $\tilde x \in \mathcal{M}(x)$ and $\tilde y\in \mathcal{M}(y)$ are recurrent, then $\tilde x\ll \tilde y$, $\tilde x=\tilde y$ or $\tilde x\gg \tilde y$. \\ 
\indent We claim that a similar theorem holds for lattices instead of PDEs, but we will not prove this, as it is not essential for the remainder of this paper.  As a result, the Aubry-Mather set of an irrational rotation vector is unique.
\end{remark}
\noindent The following well-known result shows that the set of recurrent minimizers can have a complicated topology. We recall that a topological space $\mathcal{C}$ is called a {\it Cantor set} if it is closed, perfect and totally disconnected. ``Perfect'' means that every element $c\in \mathcal{C}$ is a limit of points in $\mathcal{C}\backslash \{c\}$. ``Totally disconnected'' means that for any two elements $c_1, c_2\in \mathcal{C}$ one can decompose $\mathcal{C}$ as the disjoint union of closed sets $\mathcal{C}_1$  and $\mathcal{C}_2$ with $c_1\in \mathcal{C}_1$ and $c_2\in \mathcal{C}_2$. 
\begin{theorem}\label{Cantortheorem}
If $\omega\in \R^d\backslash \Q^d$, then $\mathcal{M}(x)$ is either connected or a Cantor set.
\end{theorem}
\begin{proof}
The recurrent subset is perfect by definition: for every $y\in \mathcal{M}(x)$, it holds that $y=\lim_{n\to\infty}\tau_{k_n,l_n}y$, where by Proposition \ref{numbertheory} the condition that $\langle \omega, k_n\rangle + l_n\neq 0$ guarantees that $\tau_{k_n,l_n}y\neq y$ for all $n$.\\
\indent We will now show that when $\mathcal{M}(x)$ is not connected, then there is a $y\in\mathcal{M}(x)$ so that $y^-\neq y^+$. So let's assume that $\mathcal{M}(x)$ is not connected and write $\mathcal{M}(x)=U\cup V$ for two nonempty closed subsets $U$ and $V$ with $U\cap V=\emptyset$. We may assume that there exist $u\in U$ and $v\in V$ so that $u\ll v$, whence we can define $y:=\sup \{u\in U\ |  u \ll v\}$. Clearly, $y\in U$, because $U$ is closed. Hence, $y\ll v$.
We claim that $y^+\neq y$. This is easily proved: if $y^+=y$, then $y=\lim_{n\to\infty}\tau_{k_n,l_n}y$ for a sequence with $y \ll \tau_{k_n,l_n}y\ll v$. By definition of $y$, it must hold that $\tau_{k_n, l_n}y \in V$. Hence, because $V$ is closed, also $y\in V$, which is a contradiction. \\
\indent The next step is to observe that an order interval $[y^-, y^+]:=\{z\in\R^{\Z^d} \ | \ y^- \leq z \leq y^+\}$ can never contain any recurrent elements other than $y^-$ and $y^+$. Namely, if $y^-\ll v \ll y^+$ were such a recurrent element, then by Proposition \ref{yminus}, it can be approximated by translates of $y$, so that there are $k$ and $l$ with $y^-\ll \tau_{k,l}y\ll y^+$. This contradicts the definition of $y^-$ or $y^+$. This is why we call the order interval $[y^-, y^+]$ a {\it gap} in the Aubry-Mather set.
\\
\indent Now we show that when $\mathcal{M}(x)$ is not connected, and hence contains at least one gap $[y^-, y^+]$, then between any two elements $w, z\in \mathcal{M}(x)$ there exists a gap.  Namely, for any given pair $w\ll z$, either $[w, z]$ is a gap, or there is a recurrent element $w \ll u  \ll z$. By Proposition \ref{yminus}, this $u$ can then be approximated by the $\tau$-orbit of $y^-$, which implies that there are $k$ and $l$ so that $w \ll \tau_{k,l}y^-\ll z$. But when $[y^-,y^+]$ is a gap, then so is $[\tau_{k,l}y^-, \tau_{k,l}y^+]$, since $\tau_{k,l}$ is order-preserving. We must therefore have that $w \ll \tau_{k,l}y^- \ll \tau_{k,l}y^+ \leq z$, i.e. that there is a gap between $w$ and $z$.\\
\indent This implies that $\mathcal{M}(x)$ is totally disconnected: if $w,z\in \mathcal{M}(x)$ with $w\ll z$, then there is a gap $[y^-,y^+]$ with $w\leq y^-\ll y^+\leq z$ and hence $\mathcal{M}(x)$ splits as the disjoint union of the closed sets $ \{u\in \mathcal{M}(x) \ | \ u \leq y^-\}$ and $\{v\in \mathcal{M}(x) \ | \ v\geq y^+\}$ that contain $w$ and $z$ respectively.
\end{proof}
\noindent The proof of Theorem \ref{Cantortheorem} shows that for any $y\in \widetilde{\mathcal{M}}(x)$, in the order interval 
$$[y^-, y^+]: = \{ z\in \R^{\Z^d} \ | \  y^- \leq z \leq y^+\}$$ 
only the elements $y^-$ and $y^+$ are recurrent. Hence, when $y^-\neq y^+$, then $[y^-, y^+]$ is called a {\it gap} in the Aubry-Mather set. Moreover, in the case that $\mathcal{M}(x)$ is not connected, then between any two recurrent configurations there exists such a gap.
\\
\indent 
When $\mathcal{M}(x)$ is connected, then we say that it forms a {\it foliation}: for every $i\in\Z^d$ and every $\xi\in \R$ there is a unique $y\in \mathcal{M}(x)$ so that $y_i=\xi$. In the case that $\mathcal{M}(x)$ is a Cantor set, one says that it forms a {\it lamination}: for every $i$ and every $\xi$ there is at most one $y$ so that $y_i=\xi$. 
\\
\indent Both foliations and laminations by minimizers occur in examples, for instance that of the Frenkel-Kontorova lattice (\ref{RR}). In fact, when $V(\xi)\equiv 0$, then the Aubry-Mather sets are all of the form $\mathcal{M}(x^{\omega,0}):=\{x^{\omega, \xi}\ | \ \xi \in\R\}$, where we recall that the linear configuration $x^{\omega,\xi}$ is defined by $x^{\omega,\xi}_i=\langle \omega, i\rangle +\xi$. These Aubry-Mather sets are clearly connected. \\
\indent On the other hand, the following theorem says that when the onsite potentials $V(\xi)$ are sufficiently oscillatory, then the Aubry-Mather sets must be Cantor sets:
\begin{theorem}\label{gaps}
Let $S_j$ be local potentials satisfying conditions A-E and let $K\subset \R^d$ be a compact set. Then there exists a number $M>0$, depending on the $S_j$ and on $K$, such that for every $1$-periodic twice continuously differentiable function $V=V(\xi)$ with ${\rm osc}\ V:=\max_{\xi,\nu\in\R}(V(\xi) - V(\nu))>M$, the collection of local potentials $\tilde S_j$ defined by $\tilde S_j(x)=S_j(x)+V(x_j)$ does not possess any connected, strictly ordered shift-invariant family of global minimizers of rotation vector $\omega\in K$.
\end{theorem}
\begin{proof}
Because $\mathcal{B}_{K}/\Z$ is compact and the functions $S_j$ are $\tau$-invariant and continuous, their oscillation over $\mathcal{B}_K$ is bounded and, say, equal to $N:=\mbox{osc}_{\ \mathcal{B}_K} S_{j}=\max_{x,y\in\mathcal{B}_K}(S_j(x)-S_j(y))$. Let $M>(2r+1)^dN$, where $r\geq 1$ is the finite interaction range of the local potentials $S_j$, and choose a smooth $1$-periodic onsite potential $V$ with oscillation larger than $M$. Assume for instance that $V(\xi)-V(\nu)>M$ for certain $\xi, \nu\in \R$.  \\ \indent 
We will now prove that if a configuration $x\in \mathcal{B}_K$ has $x_0=\xi$, then it can not be a global minimizer. In other words, that $x$ is a ``gap configuration''. This is easily shown by defining $y:\Z^d\to\R$ by setting $y_i=0$ for $i\neq 0$ and $y_0=\nu-\xi$. Now choose a finite subset $B\subset \Z^d$ such that $B^{r}_0\subset B$. Then $\mbox{supp}\ y = \{0\} \subset \mathring{B}^{(r)}$ and we compute that 
 \begin{align}\nonumber & \tilde W_B(x) -  \tilde W_B(x+y)=\sum_{j\in B}\tilde S_j(x) -\tilde S_j(x+y) =  \sum_{j\in B^r_0} \tilde S_j(x) -\tilde S_j(x+y) = \\
& V(\xi)-V(\nu) +  \sum_{j\in B^r_0} S_j(x+y)-S_j(x) > M - (2r+1)^dN>0. \nonumber
 \end{align}
 This shows that $x$ is not a global minimizer.
\end{proof}
\begin{example} 
For the Frenkel-Kontorova lattice, Theorem \ref{gaps} can be improved upon considerably. In fact, by Lemma \ref{sequence}, $\mbox{osc}_{\ \mathcal B_{\omega}} (x_j-x_k) \leq 2$, which is independent of $\omega$. Therefore, the oscillation over $\mathcal{B}$ of the interaction potential $\frac{1}{8d}\sum_{||j-k||=1}(x_j-x_k)^2$ is bounded above by $1$. Thus, for any onsite potential $V(\xi)$ with oscillation larger than $2d$, the Frenkel-Kontorova lattice with local potentials $S_j(x)  =\frac{1}{8d}\sum_{||k-j||=1}(x_j-x_k)^2 + V(x_j)$  does not have a connected family of global minimizers of any rotation vector at all. 
\\ \indent
The latter result for the Frenkel-Kontorova lattice is well-known in dimension $d=1$. It turns out that the one-dimensional Frenkel-Kontorova lattice is equivalent to the Chirikov standard map $T_V$, see Appendix A. As such, Theorem \ref{gaps} and the discussion above say that for any onsite potential $V$ with oscillation larger than $2$, the standard map $T_V$ has no rotational invariant curves. In the case that $V$ has the ``standard'' form $V(\xi)=\frac{k}{8\pi^2} \cos (2\pi \xi)$, so that $\mbox{osc}\ V = \frac{k}{4\pi^2}$, we obtain that there are no rotational invariant curves for $k>8\pi^2$. In fact, in this case the much stronger computer-proved bound $k> \frac{63}{64}$ is actually known, see \cite{percival}.
\end{example}

\section{A formal gradient flow}\label{flow}
The idea of studying globally stationary solutions by means of a formal gradient flow goes back to Gol\'e, see \cite{gole91}. We will review his ideas in this section. The new result is a parabolic Harnack inequality, see Theorem \ref{UHIP}. \\
\indent The study of the formal gradient flow starts with the observation that one can assign a meaning to the partial derivatives of the formal, and generally nonconvergent sum $W(x)=\sum_{j \in \Z^d} S_j(x)$, namely as follows. Since the potentials $S_j$ are of finite range, every variable $x_i$ appears only in finitely many terms of the formal series. Hence, we may write, with a slight abuse of notation, $$(\nabla W(x))_i:=\p_i W(x) = \sum_{||j-i||\leq r}\p_i S_j(x) \ .$$
Note that $\nabla W: \R^{\Z^d}\to \R^{\Z^d}$ is well-defined as soon as the $S_j$ are continuously differentiable and that $\nabla W$ is the formal gradient of $W$ with respect to the $l_2$-inner product $\langle x, y\rangle_{2} = \sum_{j\in\Z^d} x_jy_j$.\\ 
\indent
We remark that $x$ is globally stationary if and only if $\nabla W(x)=0$. In this section, we shall nevertheless view such $x$ as stationary points of the auxiliary differential equation
$$\frac{dx}{dt} = - \nabla W(x)\ .$$
This differential equation shall be defined on an appropriate Banach subspace $\X\subset \R^{\Z^d}$ of configurations, for which its initial value problem has existence and uniqueness of solutions. The corresponding flow is called the negative gradient flow of $W$. The motivation to study the negative gradient flow is simply that it will help us find globally stationary solutions. \\ 
\indent The Banach subspace we choose to work with is the exponentially weighted configuration space
$$\X:=\{x \in \R^{\Z^d} | \ ||x||_{\X}:= \sum_{i \in \Z^d} \frac{|x_i|}{2^{||i||}}<\infty\} \subset \R^{\Z^d}\ , $$ 
where we recall that $||i||:=\sum_{k=1}^d|i_k|$. First of all, the space of Birkhoff configurations is contained in $\X$:
\begin{lemma}\label{subspace}
$\mathcal{B}\subset \X$.
\end{lemma}

\begin{proof}
This follows because every $x\in \mathcal{B}$ has a rotation vector, say $\omega$, and $|x_i- x_0- \langle \omega,i\rangle| \leq 1$. This implies that $|x_i|\leq ||\omega|| \cdot ||i|| + |x_0| + 1$ and hence
 $$||x||_\X=\sum_{i \in \Z^d} \frac{|x_i|}{2^{\|i\|}}\leq \sum_{i \in \Z^d} \frac{\|i\|\|\omega\|+|x_0| + 1}{2^{\|i\|}} < \infty\ .$$ 
\end{proof}

\noindent We moreover note that the topology $\mathcal{B}$ inherits from $\X$ is exactly that of pointwise convergence:

\begin{proposition}\label{pointwiseinX}
Let $x\in \X$ and for all $n\in \N$, let $x_n\in \X$. Then $\lim_{n\to\infty} ||x_n-x||_{\X}=0$ if and only if $\lim_{n\to\infty} x_n = x$ pointwise. In particular, a sequence in $\mathcal{B}$ converges in $\X$ if and only if it converges pointwise.
\end{proposition}
\begin{proof}
The first claim is obvious. The second claim follows because $\mathcal{B}$ is a closed subset of $\X$.
\end{proof}

\noindent Before showing the existence of the negative gradient flow on $\X$, we need the following simple lemma, which shows that the shift maps $\tau_{k,l}$ are Lipschitz on $\X$:

\begin{lemma}\label{lipschitz translations}
Let $x, y \in \X$ and $(k,l) \in \Z^d \times \Z$. Then $\tau_{k,l}x \in \X$ and $\tau_{k,l}y \in \X$, while $\|\tau_{k,l}x - \tau_{k,l}y\|_\X\leq 2^{\|k\|}\|x-y\|_\X$.
\end{lemma}

\begin{proof}
First of all,
$$\|\tau_{k,0}x\|_\X = \sum_{i\in \Z^d} \frac{|x_{i+k}|}{2^{\|i\|}} = 2^{\|k\|} \sum_{i\in \Z^d} \frac{|x_{i+k}|}{2^{\|i\|+\|k\|}}  \leq 2^{\|k\|} \sum_{i\in \Z^d} \frac{|x_{i+k}|}{2^{\|i + k\|}} = 2^{||k||} ||x||_{\X}\ .$$
Therefore, $\|\tau_{k,l}x||_{\X}= ||\tau_{k,0}x + l||_{\X}\leq  ||\tau_{k,0}x||_{\X}+||l||_{\X}\leq 2^{||k||}||x||_{\X}+ ||l||_{\X} < \infty$ and similarly for $\tau_{k,l}y$. In particular, $\|\tau_{k,l}x - \tau_{k,l}y\|_\X = \|\tau_{k,0}(x -y)\|_\X \leq 2^{\|k\|}\|x-y\|_\X$.
\end{proof}

\noindent In particular, this means that $\tau_{k,l}:\mathbb{X}\to\mathbb{X}$ is continuous in the topology of pointwise convergence: if $x^n\to x^{\infty}$ pointwise, then $\tau_{k,l}x^n\to \tau_{k,l}x^{\infty}$ pointwise. Of course, this is also clear without Lemma \ref{lipschitz translations}.\\
\indent The main result of this section is the following theorem, which says that under the condition that the local potentials $S_j$ are twice continuously differentiable with uniformly bounded second derivatives, then $-\nabla W$ indeed defines a flow on $\X$. Moreover, this flow has the regularity properties one expects it to have.

\begin{theorem}\label{psi}
Assume the local potentials $S_j$ satisfy conditions A, B and E, that is they are twice continuously differentiable with uniformly bounded second derivatives, they depend on finitely many variables and are shift-invariant. Then the vector field $-\nabla W:\X\to\X$ is globally Lipschitz continuous, i.e. there is a constant $L>0$, depending only on the constant $C$ of condition E and the interaction range $r$ of condition A, such that for all $x, y\in\X$,
$$||\nabla W(x)-\nabla W(y)||_{\X} \leq L ||x-y||_{\X}.$$
Hence, the initial value problem $\frac{dx}{dt}=-\nabla W(x),\ x(0)=x_0$ on $\mathbb{X}$ has global-in-time existence and uniqueness of solutions and defines a complete flow $t\mapsto \Psi_t$ on $\X$. This flow is Lipschitz continuous, i.e. there are constants $L_t >0$, depending only on $L$, such that for all $t\in \R$ and $x,y\in \X$,  
$$||\Psi_tx-\Psi_ty||_{\X} \leq L_t ||x-y||_{\X}.$$ 
Moreover, this flow depends Lipschitz continuously on $-\nabla W$. This means that there are constants $\overline{L}_t>0$, depending only on $L$, such that for all $t\in \R$ and $x,y\in \X$ and for all $-\nabla W$ and $-\nabla \tilde W$ with Lipschitz constants $\leq L$ and respective complete flows $\Psi_t$ and $\tilde \Psi_t$,
$$\sup_{x\in \X} ||\Psi_t x - \tilde \Psi_t x||_{\X} \leq \overline{L}_t \sup_{x\in \X} ||\nabla W(x) - \nabla \tilde W(x)||_{\X}.$$
\end{theorem}

\begin{proof}
Using the uniform bound that $|\p_{i,k}S_j|\leq C$, see condition $E$, we will prove that $-\nabla W$ maps $\X$ to $\X$ and is globally Lipschitz continuous. The usual ODE theory then provides the existence of a complete flow $t\mapsto \Psi_t$ on $\X$.\\
\indent Thus, let $x,y \in \X$. Then first of all
\begin{align}\nonumber
|\! -\nabla W(x)_i \!+\! \nabla W(y)_i| \leq \!\!\! \sum_{||j-i||\leq r}\!\! |\p_i S_j(y) - \p_i S_j(x)| =\!\!\! \sum_{||j-i||\leq r} & \! \left| \int_0^1 \frac{d}{d\tau}\left( \p_{i} S_j(\tau y+(1-\tau) x)\right)  d\tau \right| \\
\nonumber \leq \sum_{||j-i|| \leq r} \left| \int_0^1 \sum_{||k-j||\leq r}\p_{i, k} S_j(\tau y+(1-\tau) x) d\tau\right| \cdot |y_k-x_k| & \leq \ C \!\!\sum_{||k-j||\leq r} \sum_{||j-i||\leq r}|x_k-y_k| \ .
\end{align}
But this implies that 
\begin{align}\nonumber
\|-\nabla W(x) + \nabla W(y)\|_\X  \leq C& \sum_{i\in\Z^d} 2^{-||i||}\!\! \sum_{||k-j||\leq r}\! \sum_{||j-i||\leq r}|x_k-y_k| = \\ C\! \sum_{||m||\leq r} \! \sum_{||n||\leq r} \sum_{i\in \Z^d} 2^{-||i||}  |x_{i+m+n}-y_{i+m+n}&| =  C\!\! \sum_{||m||\leq r} \sum_{||n||\leq r}||\tau_{m+n,0}x-\tau_{m+n,0}y||_{\X}
\nonumber
\end{align}
By Lemma \ref{lipschitz translations} and the fact that in the sum above $||m+n||\leq 2r$, we know that $||\tau_{m+n,0}x-\tau_{m+n,0}y||_{\X} \leq 2^{2r}||x-y||_{\X}$. Hence, noting that $|\{i\in\Z^d\ | \ ||i||\leq r\}|\leq (2r+1)^d$, we obtain that 
$$\|-\nabla W(x) + \nabla W(y)\|_\X  \leq L ||x-y||_{\X}\ ,$$
where $L:= 2^{2r}C(2r+1)^{2d}$. On the one hand, this shows that $-\nabla W$ is globally Lipschitz continuous. On the other hand, choosing $y= 0$, we see that $||-\nabla W(x) +\nabla W(0)||_{\X}\leq L ||x||_{\X}$, or $||-\nabla W(x)||_{\X} \leq L ||x||_{\X} + ||\nabla W(0)||_{\X}$, that is $-\nabla W$ maps $\X$ into $\X$. \\ 
\indent This implies the existence and uniqueness of solutions of the initial value problem $\frac{dx}{dt}=-\nabla W(x)$, $x(0)=x_0$ in $\X$, that is the existence of flow maps $\Psi_t:\X\to\X$ for all $t\in \R$. The Lipschitz continuity of $\Psi_t$ follows, as usual, from an application of Gronwall's inequality: first one shows that $||\Psi_tx-\Psi_ty||_{\X}\leq ||x-y||_{\X} + L\int_0^{|t|}||\Psi_{\tau}x-\Psi_{\tau}y||_{\X}d\tau$. This then implies that $||\Psi_tx - \Psi_t y||_{\X} \leq L_t ||x-y||_{\X}$, with $L_t=e^{L|t|}$.\\
\indent For the last part of the theorem, let $\nabla W$ and $\nabla \tilde W$ be two vector fields with Lipschitz constants $\leq L$ and complete flows $\Psi_t$ and $\tilde \Psi_t$ respectively. Call $x(t)=\Psi_t x$ and $\tilde x(t)=\tilde \Psi_t x$. We then have 
\begin{align}\nonumber
&||x(t)-\tilde x(t)||_{\X} \leq \int_0^{|t|}|| \nabla\tilde W(\tilde x(\tau)) - \nabla W(x(\tau))||_{\X}d\tau \leq   
 \int_0^{|t|} ||\nabla \tilde W(\tilde x(\tau)) \!-\!  \nabla W(\tilde x(\tau))||_{\X}d\tau + \\  &\int_0^{|t|} \! ||\nabla W(\tilde x(\tau)) - \nabla W(x(\tau))||_{\X}d\tau \! \nonumber
\leq \! |t| \sup_{x\in \X} ||\nabla W(x)\! -\! \nabla \tilde W(x)||_{\X} \! + \! L \int_0^{|t|} ||\tilde x(\tau) - x(\tau)||_{\X}  d\tau.
\end{align}
Thus, by Gronwall's inequality, $||\Psi_tx-\tilde \Psi_tx||_{\X}\leq \overline{L}_t \sup_{x\in \X} ||\nabla W(x) - \nabla \tilde W(x)||_{\X}$ with $\overline{L}_t=|t|e^{L|t|}$.
\end{proof}
\begin{remark}
It is not true in general that $-\nabla W:\X\to\X$ is a $C^1$ map. Hence, contrary to a claim made in \cite{gole01}, the $\Psi_t$ in general can not be assumed $C^1$ either. 
\end{remark}
\noindent By Proposition \ref{pointwiseinX}, the first part of Theorem \ref{psi} implies that $\nabla W: \X \to \X$ is continuous with respect to pointwise convergence: if $\lim_{n\to\infty} x^n=x^{\infty}$ pointwise, then $\lim_{n\to\infty}\nabla W(x^n)=\nabla W(x^{\infty})$ pointwise. \\ \indent Similarly, the second part of Theorem \ref{psi} implies that for every $t\in \R$ the flow map $\Psi_t: \X \to \X$ is continuous with respect to pointwise convergence. \\ \indent Part three of Theorem \ref{psi} implies that if $\nabla W^n, \nabla W^{\infty}:\X\to\X$ is a sequence of formal gradient vector fields with a uniform Lipschitz constant and corresponding flow maps $\Psi_t^n, \Psi_t^{\infty}:\X\to\X$ and such that $\nabla W^n\to \nabla W^{\infty}$ uniformly on $\X$, then for all $t$ also $\Psi_t^n\to\Psi_t^{\infty}$ uniformly on $\X$. \\
\indent In the remainder of this section, we will formulate a concept of convergence for a sequence of finite range potentials $S_j^n$ that guarantees that their corresponding gradient vector fields and flow maps converge uniformly. It turns out that it is enough to require the convergence of the gradients of the $S_j^n$. We will first of all need to define what it means for collections of gradients of finite range potentials to be ``close''. Remembering the definition in Section \ref{problemsetup} of the partial derivatives $\p_{j_1,\ldots, j_k}S$ of a $k$ times continuously differentiable function $S:\R^{\Z^d}\to\R$ of finitely many variables, we now define:
\begin{definition} Let $S:\R^{\Z^d}\to\R$ be an $m+1 \geq 0$ times continuously differentiable function of finitely many variables, that is $S(x)=s(x|_{B})$ for a certain finite subset $B\subset \Z^d$ and an $m+1$ times continuously differentiable function $s:\R^{B}\to\R$. Then we define the uniform $C^m(\R^{\Z^d})$ norm $||\nabla S||_{C^m(\R^{\Z^d})}\in [0,\infty)$ of the gradient of $S$ as the finite sum of suprema
$$||\nabla S||_{C^m(\R^{\Z^d})} :=\sum_{1\leq k\leq m+1}\ \sum_{j_1, \ldots, j_k\in \Z^d}\ \sup_{x\in\R^{\Z^d}} |\p_{j_1, \ldots, j_k}S(x)|. $$
\end{definition}
\noindent We note that if $S_j:\R^{Z^d}\to\R$ is a collection of $m+1$ times continuously differentiable, shift invariant finite range potentials, that is if $S_j(x)=s_j(x|_{B_j^r})$ for some $m+1$ times continuously differentiable function $s_j:\R^{B_j^r}\to\R$ and $S_j(\tau_{k,l}x)=S_{j+k}(x)$ for all $j, k$ and $l$, then $||\nabla S_i||_{C^m(\R^{\Z^d})}= ||\nabla S_j||_{C^m(\R^{\Z^d})}$ for all $i,j\in\Z^d$. With this in mind, we first of all prove:
\begin{proposition}\label{estimate}
Let $S_j, \tilde S_j:\R^{\Z^d}\to\R$ be two $m+1\geq 1$ times continuously differentiable and shift-invariant collections of finite range local potentials, say $S_j(x)=s_j(x|_{B_j^r})$ and $\tilde S_j(x)=\tilde s_j(x|_{B_j^r})$ and denote their corresponding gradient vector fields by $\nabla W$ and $\nabla \tilde W$. Then, there is a constant $\overline{L}>0$, depending only on the dimension $d$, such that
$$\sup_{x\in \X} ||\nabla W(x)-\nabla \tilde W(x)||_{\X}\leq \overline{L}||\nabla S_0-\nabla \tilde S_0||_{C^0(\R^{\Z^d})}.$$
\end{proposition}
\begin{proof}
We have that 
\begin{align}\nonumber
|| \nabla W(x)& - \nabla \tilde W(x)||_{\X} \leq \sum_{i\in\Z^d} 2^{-||i||} \!\! \sum_{||j-i||\leq r} |\p_i S_j(x) - \p_i \tilde S_j(x)|.  
\nonumber
\end{align}
By shift invariance, $\p_iS_j(x)=\p_{i-j}S_0(\tau_{j,0}x)$ and similarly for $\tilde S_j$, so that $\sup_x |\p_iS_j(x) - \p_i \tilde  S_j(x)|= \sup_x |\p_{i-j}S_0(x) - \p_{i-j}\tilde S_0(x)|$, and consequently
\begin{align}\nonumber
\sup_{x\in \X} &||\nabla W(x) - \nabla \tilde W(x)||_{\X} \leq \sum_{i\in\Z^d}\! 2^{-||i||}  \!\! \sum_{||j-i||\leq r} \sup_{x\in\X} |\p_{i-j}S_0(x)-\p_{i-j}\tilde S_0(x)| = \\ \nonumber
& \sum_{i\in\Z^d}\! 2^{-||i||} \!\! \sum_{||j||\leq r} \sup_{x\in\X} |\p_{j}S_0(x)-\p_{j}\tilde S_0(x)| \leq\overline{L}||\nabla S_0-\nabla \tilde S_0||_{C^0(\R^{\Z^d})}, 
\end{align}  
with $\overline{L}=\sum_{i\in\Z^d}2^{-||i||}$.
\end{proof}
\noindent We are now ready to define what it means for a sequence of local potentials to converge: 
\begin{definition}
Let $S_j^n, S_j^{\infty}: \R^{\Z^d}\to \R$ be a sequence of collections of $m+1\geq 1$ times continuously differentiable, shift-invariant functions of finite range $r$.   
Then we say that the $\nabla S_j^n$ converge to the $\nabla S_j^{\infty}$ uniformly in $C^m(\R^{\Z^d})$ as $n\to \infty$ if
$$\lim_{n\to\infty} ||\nabla S_0^n-\nabla S_0^{\infty}||_{C^m(\R^{\Z^d})} =0  .$$
\end{definition}
\noindent With this definition, we can then prove the following corollary of Theorem \ref{psi}. It trivially follows from our definitions, Theorem \ref{psi} and Proposition \ref{estimate}.
\begin{corollary}\label{continuousflow}
Let $S_j^n, S_j^{\infty}: \R^{\Z^d}\to \R$ be a sequence of continuously differentiable local potentials of finite range $r$, with corresponding gradient vector fields $\nabla W^n$ and $\nabla W^{\infty}$, and assume that $\nabla S_j^n\to \nabla S_j^{\infty}$ uniformly in $C^0(\R^{\Z^d})$. Then $\nabla W^n\to\nabla W^{\infty}$ uniformly, i.e. 
$$\lim_{n\to\infty} \sup_{x\in \X} ||\nabla W^n(x) - \nabla W^{\infty}(x)||_{\X} = 0.$$
Moreover, in the case that the $S_j^n$ and $S_j^{\infty}$ are twice continuously differentiable with uniformly bounded second derivatives, so that $-\nabla W^n$ and $-\nabla W^{\infty}$ have well defined flow maps $\Psi_t^n$ and $\Psi_t^{\infty}$, then it also holds for every $t\in\R$ that $\Psi^n_t\to \Psi_t^{\infty}$ uniformly, i.e.
$$\lim_{n\to\infty} \sup_{x\in \X} ||\Psi^n_t x - \Psi^{\infty}_t x||_{\X} =0.$$
\end{corollary}

\section{Properties of the gradient flow}\label{properties_flow}
In this section, we collect some qualitative properties of the formal negative gradient flow that was introduced in the previous section.\\
\indent First of all, not surprisingly, it is equivariant with respect to shifts:
\begin{proposition}\label{invariance}
Let $(k,l)\in \Z^d\times \Z$ and $t\in \R$. Then $\Psi_t \circ \tau_{k,l}=\tau_{k,l} \circ \Psi_t$.
\end{proposition}
\begin{proof}
By the shift-invariance of the local potentials $S_j$ of condition B above, we have that $S_{j-k}(\tau_{k,l}x)=S_{j}(x)$ for all $k,l$ and $j$. Differentiating this identity with respect to $x_{i+k}$, we find that $\p_{i}S_{j-k}(\tau_{k,l}x)=\p_{i+k}S_{j}(x)$. Assume now that $\frac{dx_i}{dt} = -\left(\nabla W(x)\right)_i$ for all $i$. Then,
\begin{align}\nonumber
\frac{d}{dt}\left(\tau_{k,l}x\right)_i = \frac{d x_{i+k}}{dt} = -\left(\nabla W(x)\right)_{i+k} & = -\!\!\!\!\!\! \sum_{||j-(i+k)||\leq r}\!\!\!\! \p_{i+k}S_j(x)  = \\ 
- \!\!\!\!\!\! \sum_{||(j-k)-i||\leq r} \!\!\!\! \p_{i}S_{j-k}(\tau_{k,l}x) =& -\left( \nabla W (\tau_{k,l}x)\right)_{i}  \ .\nonumber
\end{align}
In other words, when $t\mapsto x(t)$ is a solution of the negative gradient flow, then so is $t\mapsto \tau_{k,l}x(t)$.
\end{proof}
\noindent Proposition \ref{invariance} implies in particular that the spaces $\X_{p,q}$ of periodic configurations are invariant under the gradient flow.\\
\indent The following well-known property of the negative gradient flow is the analogue of the comparison principle for parabolic PDEs, cf. \cite{gole91} or \cite{llave-lattices}. It is a direct consequence of the monotonicity condition D.

\begin{theorem}[Strict monotonicity of the parabolic flow]\label{monotonicity psi}
Let $x,y \in \mathbb{X}$ such that $x<y$. Denote by $\Psi_t$ the time-$t$ flow of $\dot x = -\nabla W(x)$. Then for every $t>0$, $\Psi_tx \ll \Psi_ty$.
\end{theorem}
\begin{proof}
Denote $x(t)=\Psi_tx$ and $y(t)=\Psi_ty$ and define $u(t):=y(t)-x(t)$. Note that $u(0)> 0$ and that $u$ satisfies the following linear ODE: 
 \begin{align}\nonumber
\dot{u}_i(t) & =-\p_i W(y(t)) + \p_i W(x(t))=  \int_0^1\frac{d}{d\tau} \left( \sum_{||i-j||\leq r} \partial_iS_j(\tau x(t)+(1-\tau)y(t))\right) d\tau = \\ \nonumber
 & \sum_{||i-j||\leq r, ||j-k||\leq r} \!\! \left( \int_0^1 - \partial_{i,k}S_j(\tau x(t)+(1-\tau)y(t))d\tau \right) u_k(t) =: (H(t)u(t))_i\ .
\end{align}
Here, for every $t$, the operator $H(t)$ is Lipschitz from $\X$ to $\X$, by a proof similar to that of Theorem \ref{psi}. Recall that $\p_{i,k}S_j\leq 0$ when $i\neq k$, whereas $\p_{i,i}S_j<C$. This implies that there is a constant $M>0$ such that the operators $\tilde H(t):=H(t) + M\mbox{Id}: \mathbb{R}^{\Z^d}\to\R^{\Z^d}$ are positive: $u\geq 0$ implies $\tilde H(t)u\geq 0$.\\ 
\indent
Note moreover that both the $H(t)$ and the $\tilde H(t)$ are uniformly bounded operators, whence the ODEs $\dot u = H(t)u$ and $\dot v = \tilde H(t) v$ define well-posed initial value problems. More importantly, $u(t)$ solves $\dot{u}= H(t) u$ if and only if $v(t):=e^{Mt}u(t)$ solves $\dot{v}=\tilde H(t)v$.  We will now prove that for every $t>0$ and every $i$, $v_i(t) >0$. Then, obviously, $u_i(t)>0$ as well, which then proves the theorem. \\
\indent
To prove the claim on $v(t)$, we solve the initial value problem for $\dot v = \tilde H(t)v$ by Picard iteration, that is we write 
\begin{equation}\label{picard}v(t)= \left( \sum_{n=0}^{\infty} \tilde H^{(n)}(t) \right) v(0),\end{equation} 
where the $\tilde H^{(n)}(t)$ are defined inductively by
$$\tilde H^{(0)}(t)=\mbox{Id}\ \mbox{and} \ \tilde H^{(n)}(t):=\int_0^t \tilde H(\tilde t) \circ \tilde H^{(n-1)}(\tilde t) \ d\tilde t \ \mbox{for} \ n \geq 1\ .$$
Observe that the positivity of $\tilde H(t)$ implies that the $\tilde H^{(n)}(t)$ are positive as well. Because $v(0)=u(0)>0$, we can therefore estimate, for any $i, k\in\Z^d$ with $||i-k||=1$,
\begin{align}\nonumber
v_i(t) = \left(\sum_{n=0}^{\infty} \tilde H^{(n)}(t) v(0)\right)_{i} \geq (\tilde H(t) v(0))_i \geq \\ \left( \int_0^t\int_0^1- \partial_{i,k}S_i(\tau x(\tilde t) +(1-\tau)y(\tilde t))d\tau d\tilde t \right) v_k(0) \ .\label{qualitativeestimate}
\end{align}
Now choose a $k \in \Z^d$ such that $v_k(0)>0$ and recall that $\p_{i,k}S_i < 0$. Then from (\ref{qualitativeestimate}) it follows that if $||i-k||=1$, then for all $t>0$, $v_i(t) > 0$. \\ 
\indent To generalize to the case that $||i-k|| \neq 1$, let us choose a sequence of lattice points $k=i_0, \ldots, i_{N}=i$ such that $||i_n-i_{n-1}||=1$ and $N=||i-k||$. Then, by induction, $v\left(\frac{nt}{N}\right)_{i_n}>  0$ for all $n$. Thus, if $v_k(0)>0$ and $t>0$, then $v_i(t)> 0$.
\end{proof}
\noindent Theorem \ref{monotonicity psi} immediately gives us the following important corollary.
\begin{corollary} \label{Birkhoff invariance}
Let $\omega \in \R^d$. Then $\mathcal{B}_\omega$ is positively invariant under the negative gradient flow: $\Psi_t(\mathcal{B}_\omega)\subset \mathcal{B}_\omega$, for every $t>0$.
\end{corollary}
\noindent This just follows because the strict monotonicity of the parabolic flow implies that $\Psi_t$ preserves the inequalities that define $\mathcal{B}_{\omega}$. \\
\indent The following is a quantitative version of Theorem \ref{monotonicity psi}. It will be crucial in the remainder of this paper and we have not found it elsewhere in the literature.
\begin{theorem}[Parabolic Harnack inequality] \label{UHIP}
Let $t>0$, $K\subset \mathbb{R}^d$ a compact set and $x,y\in \mathcal{B}_K:= \cup_{\omega\in K} \mathcal{B}_{\omega}$ such that $x<y$. Then there exists a constant $L>0$, depending only on $K$, $||i-k||$ and $t$, such that for all $i,k \in \mathbb{Z}^d$,
$$(\Psi_t y)_i - (\Psi_t x)_i \geq L ( y_k-x_k )\ .$$
\end{theorem}
\begin{proof}
The proof is a quantitative variant of the proof of Theorem \ref{monotonicity psi}. We start by recalling that by Corollary \ref{uniform_twist}, there is a constant $\lambda>0$, depending only on $K$, such that $\p_{i,k}S_i(z)\leq -\lambda < 0$ for all $||i-k||=1$ and $z\in  \cup_{\omega\in K} \mathcal{B}_{\omega}$. Then (\ref{qualitativeestimate}) shows that if $||i-k||=1$, then $v_i(t) \geq \tilde L_1 v_k(0)$, with $\tilde L_1=t\lambda$. \\ 
\indent To generalize to the case that $||k-i|| \neq 1$, we again choose a sequence of lattice points $k=i_0, \ldots, i_{N}=i$ such that $||i_n-i_{n-1}||=1$ and $N=||i-k||$. Then there is a constant $\tilde L_N'=\frac{t\lambda}{N}$ depending only on $K$, $t$ and $||i-k||$ such that $v\left(\frac{nt}{N}\right)_{i_n} \geq L_N' v\left(\frac{(n-1)t}{N}\right)_{i_{n-1}}$ for all $n$. Thus, $v_i(t)\geq \tilde L_N v_k(0)$ with $\tilde L_N= (\tilde L_N')^N$. \\ \indent This proves that $u_i(t)\geq L u_k(0)$ with $L=e^{-Mt}\tilde L_N = e^{-Mt}(\lambda t/ ||i-k||)^{||i-k||}$. 
\end{proof}
\noindent Note that for Birkhoff configurations, both the strict monotonicity, Theorem \ref{monotonicity psi}, and the elliptic Harnack inequality, Theorem \ref{ellharn}, follow directly from this parabolic Harnack inequality. \\
\indent We moreover remark that under the uniform twist condition that $\p_{i,k}S_i(z)\leq-\lambda<0$ for all $z\in \mathbb{X}$ and $||i-k||=1$, the above parabolic Harnack inequality holds for all $x,y$ in $\mathbb{X}$ with $x<y$, i.e. it then holds irrespective of the Birkhoff property of $x$ and $y$. This uniform twist condition for instance holds for the Frenkel-Kontorova problem, see formula  (\ref{FKpotential}).\\
\indent To finish this section, let us for completeness include the following alternative existence proof for globally stationary Birkhoff solutions of arbitrary rotation vector. It was provided by Gol\'e in \cite{gole91} in dimension $d=1$. The below is a more or less trivial generalization to higher dimensions, see also \cite{llave-lattices}. As opposed to the results presented in Section \ref{nonpermin}, it also holds without Hypothesis C that requires that the $S_j$ are coercive. The proof presented here is slightly shorter and more direct than the proof in \cite{llave-lattices}.

\begin{theorem}\label{existenceGole}
Also without  the coercivity condition C, it holds that for every $\omega \in \R^d$, there exists an $x\in \mathcal{B}_{\omega}$ with $\nabla W(x)=0$. 
\end{theorem}
\begin{proof}
Recall that the conditions A, D and E alone guarantee that the compact set $\mathcal{B}_{\omega}$  is forward invariant under the negative gradient flow. Condition B will be used below. \\
\indent Now, for $B\subset \Z^d$ a finite subset, recall the definition of the finite action $W_B(x):=\sum_{j\in B}S_j(x)$. Then, for $i\in \mathring{B}^{(r)}$, it holds that $\p_iW_B(x)= \p_iW(x)$, whereas if $||i-B||:=\min_{j\in B}||j-i||> r$, it is true that $\p_iW_B(x)=0$. Thus, the time-derivative of $W_B$ along solutions of $\frac{dx}{dt}=-\nabla W(x)$ equals
\begin{align}\nonumber
\frac{d}{dt} W_B(x)  = - \sum_{i\in \Z^d} \p_iW_B(x)\p_iW(x) =   - \sum_{i\in\mathring{B}^{(r)}} \left(\p_iW(x)\right)^2 - \!\!\!\! \sum_{i\notin \mathring{B}^{(r)}, ||i-B||\leq r } \!\!\! \p_iW_B(x) \p_iW(x) \ .
\end{align}
We call $A_B(x):=\sum_{i\in\mathring{B}^{(r)}} (\p_iW(x))^2$. It is the square length of the gradient of the map $y\mapsto W_B(x+y)$ from $\R^{\mathring{B}^{(r)}}$ to $\R$ evaluated at $y=0$. With this definition, one checks that if $B_1\subset B_2$, then $A_{B_1}(x)\leq A_{B_2}(x)$. Moreover, if $\mathring{B_1}^{(r)}$ and $\mathring{B_2}^{(r)}$ are disjoint, then $A_{B_1\cup B_2}(x) = A_{B_1}(x) + A_{B_2}(x)$. \\
\indent
The second sum in the expression for $\frac{d}{dt}W_B$ consists of ``boundary terms''. We will call it $a_B(x):= \sum_{i\notin \mathring{B}^{(r)}} \p_iW_B(x) \p_iW(x)$. Because $\mathcal{B}_{\omega}$ is compact and $\p_iS_j = \p_{i-j}S_0\circ \tau_{j,0}$ for all $j\in \Z^d$, there is a constant $c>0$ with the property that $|\p_iS_j|\leq c$ for all $i,j \in \Z^d$ and uniformly on $\mathcal{B}_{\omega}$. This in turn implies the estimate $|a_B(x)| \leq c^2(2r+1)^d |\p B |$.  \\ 
\indent Assume now that there is no globally stationary point in $\mathcal{B}_{\omega}$. Then for every $x\in \mathcal{B}_{\omega}$ there is a finite subset $B_x\subset \Z^d$ such that $A_{B_x}(x) =2\varepsilon_x>0$. Moreover, because $\nabla W_{B_x}$ is continuous on $\R^{\Z^d}$, it holds that $x$ has an open neighborhood $U_x\subset \R^{\Z^d}$ on which $A_{B_x}>\varepsilon_x$. By compactness we can find a finite collection $x_1,\ldots, x_m\in \mathcal{B}_{\omega}$ such that $\mathcal{B}_{\omega}\subset \cup_{l=1}^m U_{x_l}$. Define $B:=\cup_{l=1}^m B_{x_l}$ and $\varepsilon:=\min_l \{\varepsilon_{x_l}\}>0$. Then every $x\in\mathcal{B}_{\omega}$ is in some $U_{x_l}$ and thus, $A_B(x) \geq A_{B_{x_l}}(x)>\varepsilon_{x_l} \geq \varepsilon >0$, that is $A_B>\varepsilon>0$ uniformly on $\mathcal{B}_{\omega}$. Moreover, translation invariance implies that for any $k\in\Z^d$, also $A_{B+k} > \varepsilon$ uniformly on $\mathcal{B}_{\omega}$.\\
\indent For $n\in \N$, define the ball $B(n):=\{j\in\Z^d \ | \ ||i||\leq n\} \subset \Z^d$ and let $N\in \N$ be such that the $B$ above is contained in $B(N)$. Then $A_{B(N)}(x)\geq A_B(x)>\varepsilon>0$. \\ \indent Let $m\geq 2$ be an integer. By translation invariance and the fact that $B(mN)$ contains at least $m^d$ translates of $B(N)$ with disjoint $r$-interiors, it holds that $A_{B(mN)} \geq m^dA_{B(N)}> m^d \varepsilon$. On the other hand, $|a_{mN}(x)|\leq c^2(2r+1)^d|\p B(mN)| = D_N^rm^{d-1}$ for some $D_N^r>0$. Thus, $\frac{d}{dt}W_{B(mN)}(x) \leq - m^d\varepsilon + D_N^rm^{d-1}$ and hence by choosing $m$ large enough, we can arrange that $\frac{d}{dt}W_{mN}(x)\leq -1$ uniformly on $\mathcal{B}_{\omega}$. \\ \indent Since $\mathcal{B}_{\omega}$ is forward invariant under the negative gradient flow, this implies that $W_{mN}$ is not bounded from below on $\mathcal{B}_{\omega}$. This contradicts the fact that $\mathcal{B}_{\omega}$ is compact and $W_{mN}$ is continuous. This proves that there must be a globally stationary point in $\mathcal{B}_{\omega}$.
\end{proof}

\section{Ghost circles}\label{GC}
In dimension $d=1$, the concept of a ghost circle was introduced by Gol\'e. We generalize this definition here to general dimensions. Note the similarity with Definition \ref{defmatherset} of an Aubry-Mather set.
\begin{definition}[Ghost Circle] \label{definition ghost circle}
A {\it ghost circle} $\Gamma \subset \R^{\Z^d}$ is a collection of configurations with the following properties
\begin{itemize}
	\item $\Gamma$ is nonempty, closed and connected
	\item $\Gamma$ is strictly ordered, i.e. for every $x,y \in \Gamma$, $x \ll y$, $x=y$ or $x \gg y$
	\item $\Gamma$ is invariant under shifts: if $x\in \Gamma$, then for every $(k,l) \in \Z^d \times \Z$, also $\tau_{k,l}x \in \Gamma$
	\item $\Gamma$ is invariant under the positive and negative gradient flow: for all $t \in \R$, $\Psi_t(\Gamma) = \Gamma$
\end{itemize}
\end{definition}
\noindent An example of a ghost circle are the connected Aubry-Mather sets of Theorem \ref{Cantortheorem}. \\ 
\indent The strict ordering and the shift-invariance of a ghost circle $\Gamma$ imply that any configuration $x\in \Gamma$ is Birkhoff and hence has a rotation vector $\omega=\omega(x)$. The ordering of $\Gamma$ moreover implies that this rotation vector is independent of the choice of $x\in \Gamma$, that is $\omega = \omega(\Gamma)$ and thus, $\Gamma \subset \mathcal{B}_\omega$. \\
\indent
Let $j \in \Z^d$. Recall the definition of the projection to the $j$-th factor $$\pi_j:\R^{\Z^d} \to \R, \ \pi_j(x)=x_j\  .$$ 
Each $\pi_j$ is continuous with respect to pointwise convergence. In fact, we can show that $\pi_j|_{\Gamma}: \Gamma \to \R$ is a homeomorphism: 

\begin{proposition}\label{projection}
Let $\Gamma$ be a ghost circle. Then, for every $j\in\Z^d$, the projection $\pi_j:\R^{\Z^d} \to \R$ induces a homeomorphism $\pi_j|_{\Gamma}:\Gamma \to \R$. 
\end{proposition}

\begin{proof}
Let $\Gamma \subset \R^{\Z^d}$ be a ghost circle. Clearly, $\pi_j|_{\Gamma}$ is continuous. \\
\indent
The strict ordering of $\Gamma$ implies that $\pi_j|_{\Gamma}$ is injective. Moreover, shift-invariance of $\Gamma$ implies that if $x \in \Gamma$, then so is $\tau_{0,l}x = x+l$ for every $l\in \Z$, whence the range of $\pi_j|_{\Gamma}$ is unbounded. Since $\Gamma$ is connected and $\pi_j|_{\Gamma}$ is continuous, its range is both unbounded and connected, that is $\pi_j|_{\Gamma}:\Gamma\to\R$ is surjective. \\
\indent To prove that $(\pi_j|_{\Gamma})^{-1}:\R\to\Gamma$ is continuous, it suffices to realize that $\pi_j:\R^{\Z^d}\to\R$ is an open map, i.e. that it sends open sets to open sets.  This holds because the topology of pointwise convergence is generated by open sets $U\subset \R^{\Z^d}$ for which $\pi_k(U)=V$ with $V\subset \R$ an open subset, while $\pi_l(U)=\mathbb{R}$ for all $l\neq k$. For such $U$, it is clear that $\pi_j(U)$ is open.
\end{proof}
\noindent Lemma \ref{projection} thus says that a ghost circle $\Gamma$ is homeomorphic to $\R$. It should be remarked though that, because $\Gamma$ is invariant under the vertical shift $\tau_{0,1}$, and the gradient flow $\Psi_t$ is equivariant with respect to $\tau_{0,1}$, it makes sense to identify every element $x\in \Gamma$ with $\tau_{0,1}x=x+1\in \Gamma$. The quotient $\Gamma/\Z\cong \R/\Z$ is a genuine topological circle. This identification is sometimes understood in this paper. 
\\
\indent The name {\it ghost circle} refers to the fact that $\Gamma/\Z$ may not consist of ``physically relevant'' configurations, i.e. globally stationary solutions. But, being a compact one-dimensional object consisting of orbits of a formal gradient flow, it has a good chance of containing such solutions. In fact, the following proposition serves as a first motivation to study ghost circles.
\begin{proposition}\label{solution}
Every ghost circle $\Gamma\subset \R^{\Z^d}$ contains a globally stationary solution.
\end{proposition}
\noindent Since $\Gamma$ is a closed, flow-invariant subset of some $\mathcal{B}_{\omega}$, the proof of this proposition is identical to that of Theorem \ref{existenceGole}. Moreover, we remark that when $\Gamma$ contains at least one global minimizer, say $x$, then it automatically contains the entire Aubry-Mather set $\mathcal{M}(x)$. \\
\indent In the following two sections we first of all show that under generic conditions, ghost circles of rational rotation vectors exist and then we will prove a compactness result for ghost circles which will allow us to take limits and obtain ghost circles of irrational rotation vectors.

\section{Morse approximations and periodic ghost circles} \label{PMGC}
In this section, we will prove two technical results. The first is that the local potentials $S_j$ can be perturbed, in a way that will be made precise, so that the periodic action $W_{p,q}:\X_{p,q}\to\R$ becomes a Morse function. \\ \indent
The second result of this section says that whenever $W_{p,q}:\X_{p,q}\to \R$ is a Morse function, then there exists a ghost circle $\Gamma\subset \X_{p,q}$. \\
\indent Together with the results of Section \ref{convergence}, this will imply that any collection of local potentials admits a ghost circle of arbitrary rotation vector.

\subsection{Existence of Morse approximations} \label{morse_approximations}
Let $\omega\in\mathbb{Q}^d$ be a rational rotation vector and let $(p_1,q_1), \ldots, (p_d,q_d)$ be a set of principal periods for $\omega$. Recall that in Section \ref{classAM} we defined the periodic action function $W_{p,q}:\X_{p,q}\to\R$ by $W_{p,q}(x)=\sum_{j\in B_p}S_j(x)$. \\
\indent One says that $W_{p,q}: \X_{p,q}\to \R$ is a {\it Morse function} if at its critical points its Hessian is nondegenate. In other words, if $\nabla W_{p,q}(x)=0$ implies that $D^2W_{p,q}(x)$ is invertible, where $D^2W_{p,q}(x)$ is the symmetric matrix of second derivatives of $W_{p,q}$ evaluated at $x$. By the implicit function theorem, every critical point $x$ of a Morse function is isolated. Moreover, each of these critical points can be assigned an index $i(x)$ which equals the dimension of the unstable manifold of $x$, considered as an equilibrium point for the negative gradient flow $\frac{dx}{dt}=-\nabla W_{p,q}(x)$. \\ 
\indent We remark here that for arbitrary local potentials $S_j$, the periodic action $W_{p,q}$  is not automatically a Morse function. A simple example of a non-Morse action function arises in the Frenkel-Kontorova model without local potential, for which 
$$W_{p,q}(x) = \sum_{j\in B_{p}}\frac{1}{8d}\sum_{||i-j||=1}(x_i-x_j)^2\ .$$ 
This action function satisfies $W_{p,q}(x + t)=W_{p,q}(x)$ for all $t\in \R$, so that its second derivative is everywhere degenerate. In fact, it has a one-parameter family of stationary points, and thus none of those is isolated. Nevertheless, in this subsection we will prove the following theorem:
\begin{theorem}
\label{morsification s}
Let $S_j:\R^{\Z^d}\to \R$ be local potentials that satisfy conditions A-E. Let $\omega \in \Q^d$ and let $(p_1, q_1), \ldots, (p_d, q_d)$ be principal periods $\omega$, that is $\overline{\X}_{\omega}=\X_{p,q}$. Then there exists a sequence of local potentials $S_j^n$ with the following properties:
\begin{itemize}
\item[{\bf 1.}] The $S_j^n$ satisfy conditions A-E.
\item[{\bf 2.}] The range of interaction of the $S_j^n$ is uniformly bounded in $n$.
\item[{\bf 3.}] For every $n$, the periodic action $W_{p,q}^n=\sum_{j\in B_p} S_j^n$ is a Morse function on $\X_{p,q}$.
\item[{\bf 4.}] The gradients converge uniformly: $\lim_{n\to \infty} \nabla S_j^n = \nabla S_j$ uniformly in $C^1(\R^{\Z^d})$
\item[{\bf 5.}] The potentials converge uniformly on compacts: $\lim_{n\to \infty} S_j^n = S_j$ uniformly on $\mathcal{B}_{p,q}$.
\end{itemize}
\end{theorem}

\noindent In dimension $d=1$, this theorem was proved by Gol\'e \cite{gole92}, \cite{gole01} in the context of twist maps. His proof does not generalize to dimensions $d>1$ or to general monotone variational problems in dimension $d=1$, because it explicitly exploits the interpretation of $S_j(x)=S(x_j, x_{j+1})$ as the generating function of a twist map of the annulus, see Appendix A. \\
\indent Our proof in higher dimensions is different, and it is based on Lemma \ref{tauaction} and ideas from equivariant Morse theory. We start by making Lemma \ref{tauaction} a bit more quantitative:
 \begin{lemma}\label{properly}
Let $\omega\in \Q^d$. Then the $\tau$-action of $(\Z^d\times \Z)/I_{\omega}$ on $\overline{\X}_{\omega}$ is properly discontinuous. More precisely, when $(k,l)$ represents a nontrivial element of $(\Z^d\times \Z)/I_{\omega}$ and $x\in\overline{\X}_{\omega}$, then
 $$|\tau_{k,l}x-x|_1: = \sum_{i \in B_p}|(\tau_{k,l}x-x)_i|\geq 1\ .$$
 \end{lemma}
 \begin{proof}
Let $(p,q)$ be principal periods for $\omega$, i.e. $\omega=-p^{-T}q$, and write $n:=|\det p \ \!|$. We notice that for an arbitrary $k\in \Z^d$ it holds that $-n\langle k, \omega\rangle = \langle k, |\det p \ \!|p^{-T}q\rangle \in \Z$ and thus that $(n k, - n\langle\omega,k\rangle) \in I_{\omega}$. Hence, writing $nl =$ $-n\langle \omega, k\rangle$ $+ nl + n \langle \omega, k\rangle$, we see that $\tau_{k,l}^{n}x = \tau_{n k, n l}x =\tau_{0,n \left( l +  \langle \omega, k\rangle\right)}x$. Thus, $|\tau^n_{k,l}x-x|_1=n^2 \left| l +  \langle \omega, k\rangle\right|$. 
\\
\indent Now if $(k,l)$ represents a nontrivial element of $(\Z^d\times \Z)/I_{\omega}$, then $n\cdot | l +  \langle \omega, k\rangle| \geq 1$, and hence we have that $|\tau_{k,l}^{n}x - x|_1\geq n$. We claim that this implies that $|\tau_{k,l}x-x|_1\geq 1$. This follows from the fact that $\tau_{k,l}^{j+1}x-\tau^j_{k,l}x=\tau_{k,0}(\tau_{k,l}^jx-\tau_{k.l}^{j-1}x)$ and thus, by induction, that $|\tau_{k,l}^{j+1}x-\tau_{k,l}^jx|_1=|\tau_{k,l}x-x|_1$. Therefore, $|\tau_{k,l}^nx-x|_1\leq |\tau^{n}_{k,l}x-\tau_{k,l}^{n-1}x|_1 + \ldots + |\tau_{k,l}x-x|_1=n|\tau_{k,l}x-x|_1$, which means that $|\tau_{k,l}x-x|_1\geq 1$.
 \end{proof} 
\noindent With Lemma \ref{properly} at hand, one can prove that the quotient $\overline{\X}_{\omega}/(\Z^d\times \Z)$ is a smooth manifold. An arbitrary $\Z^d\times \Z$-invariant function $f:\overline{\X}_{\omega} \to \R$ descends to this quotient and can hence be perturbed into a shift-invariant Morse function $f^{\varepsilon}$. Instead of providing this rather standard construction from equivariant Morse theory, let us prove this latter fact directly here:
\begin{theorem}\label{morsificationtheorem}
Let $\omega\in \Q^d$ and let $p,q$ be principal periods for $\omega$. When $f:\X_{p,q}\to \R$ is an $m\geq 2$ times continuously differentiable shift-invariant function, then for every $\varepsilon>0$ there exists a shift-invariant Morse function $f^{\varepsilon}:\X_{p,q}\to \R$ with 
$$||f-f^{\varepsilon}||_{C^m(E(N))} \leq \frac{\varepsilon}{(1+N^2)^2}\ \mbox{for every} \ N>0\ .$$
Here,
$$E(N):= \{x\in \X_{p,q}\  | \ |x_i-x_k|\geq N \ \mbox{for some} \ i\neq k \ \mbox{with} \ i,k \in B_p \}\ .$$ 
\end{theorem}
\begin{proof}
Let $c:=\frac{1}{4|\det p\ \!|}$ and define the discrete collection of configurations 
$$G_{p,q}:=\{x:\Z^d \to c\cdot \Z \ |\ \tau_{p_j,q_j}x=x \ \mbox{for} \ j=1,\ldots, d\ \} \subset \X_{p,q}\ . $$
We first of all remark that it is clear that $\tau_{k,l}G_{p,q}=G_{p,q}$. For $x\in G_{p,q}$, let us now define the balls 
$$B_r(x):=\{ y \in \X_{p,q}\ | \ |y - x|_1 : = \sum_{i \in B_p}|x_i - y_i|<r\ \}\ . $$ 
Then we have that $\tau_{k,l}B_r(x)=B_r(\tau_{k,l}x)$, because $|\tau_{k,l}x-\tau_{k,l}y|_1=|x-y|_1$, that is the norm $|\cdot|_1$ on $\X_{p,q}$ is shift-invariant. \\
\indent Moreover, if $y\in \X_{p,q}$, then there must be an element $x\in G_{p,q}$ with $|x_i-y_i|\leq \frac{1}{2}c$ for all $i$, that is for which $|x-y|_1\leq \frac{1}{8}$. In other words, $\X_{p,q}=\bigcup_{x\in G_{p,q}} B_r(x)$ when $r>\frac{1}{8}$. On the other hand, Lemma \ref{properly} implies that when $r<\frac{1}{2}$, then $B_{r}(x)\cap B_{r}(\tau_{k,l}x) = \emptyset$ unless $\langle \omega,k \rangle + l =0$. \\
\indent
This proves that for $\frac{1}{8}<r<\frac{1}{2}$, the collection $\{B_{r}(x)\}_{x\in G_{p,q}}$ forms a shift-invariant covering of $\X_{p,q}$ on which $(\Z^d\times \Z)/I_{\omega}$ acts ``properly discontinuously''.\\
\indent Finally, we let $\phi:\X_{p,q}\to [0,1]$ be a $C^{\infty}$ bump function with the properties that $\phi \equiv 0$ outside $B_{\frac{1}{2}}(0)$ and $\phi \equiv 1$ on $B_{\frac{1}{4}}(0)$. Let's say that $||\phi||_{C^m(\X_{p,q})}\leq E$. \\
\indent After these preparations, we are ready to construct the perturbation $f^{\varepsilon}$ of $f$. This is done by enumerating $G_{p,q}= \{x^1, x^2, \ldots \}$ and defining it inductively. \\
\indent So let us assume that $f^{\varepsilon}_{n-1}$ is $\tau$-invariant, satisfies the Morse property on the union 
$\bigcup_{1\leq i \leq n-1}B_{\frac{1}{4}}(x^i)$ and
fulfills the estimates 
$||f-f^{\varepsilon}_{n-1}||_{C^m(E(N))} \leq \frac{\varepsilon(1-2^{-(n-1)})}{(1+N^2)^2}$. 
\\
\indent  We now want $\alpha_n\in \R^{B_p}$ to be a vector so that $x\mapsto f_{n-1}^{\varepsilon}(x)+\langle \alpha_n, x\rangle$ is Morse on $B_{\frac{1}{4}}(x^n)$. Such $\alpha_n$'s are dense in $\R^{B_p}$ by Sard's theorem, see for instance \cite{Hirsch}. 
\\
\indent The function $f^{\varepsilon}_n$ is now defined as the shift-invariant function
$$f^{\varepsilon}_n(x):= f^{\varepsilon}_{n-1}(x)+ \!\!\! \sum_{(k,l)\in (\Z^d\times \Z)/I_{\omega}}\!\!\! \phi(\tau_{k,l}(x-x_n))\langle \alpha_n, \tau_{k,l}(x-x^n)\rangle\ .$$
Because $B_{\frac{1}{2}}(\tau_{k,l}x_n)$ does not intersect $B_{\frac{1}{2}}(\tau_{K,L}x_n)$ unless $(k,l)=(K,L) \!\! \mod I_{\omega}$, we have that at every $x\in \X_{p,q}$, the above sum consists of only one term.  Moreover, $f_n^{\varepsilon}$ is Morse on $B_{\frac{1}{4}}(x^n)$ by construction.  \\
\indent In fact, by choosing $\alpha_n$ small enough, one can make sure that $f^{\varepsilon}_n$ is Morse on the entire union $\bigcup_{1\leq i \leq n}B_{\frac{1}{4}}(x^i)$. This is true because the collection of Morse functions is open in the space of differentiable functions $C^m\left( \bigcup_{1\leq i\leq n-1}B_{\frac{1}{4}}(x^i)\right)$ for $m\geq 2$, see \cite{Hirsch}. \\ 
\indent By choosing $\alpha_n$ even smaller if necessary, we can also arrange that $f_{n-1}^{\varepsilon}-f_n^{\varepsilon}$ has a $C^m(E(N))$-norm less than $\frac{2^{-n}\varepsilon}{(1+N^2)^2}$. This implies that
\begin{align}\nonumber
||f-f^{\varepsilon}_n ||_{C^m(E(n))}  \leq \frac{(1-2^{-(n-1)})\varepsilon}{(1+N^2)^2} + \frac{2^{-n}\varepsilon}{(1+N^2)^2} = \frac{(1-2^{-n})\varepsilon}{(1+N^2)^2} \ .\nonumber
\end{align}
The required $f^{\varepsilon}$ is the limit $f^{\varepsilon}:=\lim_{n\to\infty} f^{\varepsilon}_n$. Not only does this limit satisfy the required estimates, but it also stabilizes pointwise, which shows that it is Morse.
\end{proof}

\noindent We can now complete the proof of Theorem \ref{morsification s}:
\begin{proof}[Proof of Theorem \ref{morsification s}] 
We start by perturbing the $S_j$ so that they satisfy a strict monotonicity criterion. This will then allow us to perturb the potentials once more without risking to destroy monotonicity condition D. Recall that $B_p:=p([0,1]^d)\cap \Z^d$ is a fundamental domain of $p$. It has cardinality $|\det p \ \!|$. Our first perturbation step is now made by defining 
$$\widetilde{S}_j^n(x):=S_j(x)+ \frac{1}{n}\sum_{i,k \in j+ B_p}(x_k-x_i)\arctan(x_k-x_i)\  .$$ 
The strict monotonicity of the $\widetilde{S}_j^n$ follows because $\frac{d^2}{dx^2}\left( x\arctan x \right) =\frac{2}{(1+x^2)^2}$ is strictly positive. Hence,
$$\p_{i,k}\widetilde{S}_j^n(x)\leq - \frac{1}{n}\frac{2}{(1+(x_i-x_k)^2)^2} < 0\ \mbox{for all}  \ i,k\in j+B_p \ \mbox{with} \ i\neq k\ . $$
\noindent By Theorem \ref{morsificationtheorem}, the periodic action $\widetilde{W}_{p,q}^n: \X_{p,q} \to \R$ defined by $\widetilde{W}_{p,q}^n(x):= \sum_{j \in B_p}\widetilde{S}_j^n(x)$ can now be perturbed into a $\tau$-invariant Morse function $W^{n}_{p,q}:\X_{p,q}\to\R$ of the form 
$$W_{p,q}^n(x)=\widetilde{W}_{p,q}^n(x) + F^n(x)\ . $$ 
The perturbation $F^n$ may be chosen so that it satisfies $||F^n||_{C^2(\X_{p,q})} \leq \frac{1}{n}$, $F^n(\tau_{k,l}x)=F^n(x)$ for all $x\in \X_{p,q}$ and all $k,l$ and $|\p_{i,k}F^n(x)|\leq \frac{1}{2n}\frac{1}{(1+(x_i-x_k)^2)^2}$. \\
\indent For $x:\Z^d\to\R$, let us denote by $\left. x \right|_{j+B_p}^{\rm per}$ the $p$-periodic extension of $x|_{j+B_p}$ defined by $\left(x|_{j+B_j}^{\rm per}\right)_i= x_k$, where $k\in j+B_p$ is the unique element of $j+B_p$ equal to $i$ modulo $p(\Z^d)$. Then we can define, for each $j\in \Z^d$ the new local potential 
$$S_j^n(x):= \widetilde{S}_j^n(x) + \frac{1}{|\det p \ |}F^n\left(\left. x \right|_{j+B_p}^{\rm per}\right)\ .$$
We will now prove that these $S_j^n$ satisfy all requirements of Theorem \ref{morsification s}. \\
\indent In fact, condition $A$ and requirement {\bf 2.} hold true because the range of interaction of both the sum $\frac{1}{2n}\sum_{i,k\in j+B_p} (x_k-x_i)\arctan(x_k-x_i)$ and the perturbation $F^n\left(x|_{j+B_p}^{\rm per}\right)$ do not exceed the bounded radius of $B_p$. \\
\indent Condition B holds by definition. Condition $C$ holds because $x\mapsto x\arctan x$ is nonnegative and $|F_n(x)|$ is uniformly bounded. Condition $D$ holds true because $\p_{i,k}S_j^n(x)\leq \p_{i,k}S_j(x)$, as is easy to check. \\
\indent Requirement {\bf 3.} holds because $W_{p,q}^n(x)=\sum_{j\in B_p} S_j^n(x) = \widetilde{W}^n_{p,q}(x)+F^n(x)$ is a Morse function by construction. \\
\indent Requirement {\bf 4.} and condition E are true because both $|\frac{d}{dx}x\arctan x |= |\frac{1}{1+x^2} + \arctan x| \leq 3$ and $|\frac{d^2}{dx^2}x\arctan x|=|\frac{2}{(1+x^2)^2}|\leq 2$ are uniformly bounded and $||F^n||_{C^2(\X_{p,q})} \leq \frac{1}{2n}$, so that $||\nabla S_j^n-\nabla S_j||_{C^1(\R^{\Z^d})}\leq \frac{C}{n}$ for some constant $C$ depending on the dimension $d$, the periodicity $p$ and the range of interaction $r$. \\
\indent Similarly, requirement {\bf 5.} holds true because $\sum_{i,k\in j+B_p} (x_i-x_k)\arctan (x_i-x_k)$ is uniformly bounded on $\mathcal{B}_{p,q}$ and $|F_n(x)|\leq \frac{1}{2n}$ uniformly on $\X_{p,q}$.
\end{proof}

\subsection{Existence of periodic ghost circles for Morse actions}\label{morse_existence}
We will now show that when the local potentials $S_j$ satisfy conditions A-E and are chosen so that $W_{p,q}: \X_{p,q}\to\R$ is a Morse function, then they admit a periodic ghost circle $\Gamma\subset \X_{p,q}$. More precisely, we will prove the following:

\begin{theorem}\label{ghost_circles} 
Let $\omega\in \Q^d$ and let $(p_1, q_1), \ldots, (p_d, q_d)$ be principal periods for $\omega$. Assume moreover that the local potentials $S_j$ are chosen so that $W_{p,q}:\X_{p,q}\to \R$ is a Morse function. Then there exists a $C^1$ ghost circle $\Gamma \subset \X_{p,q}$ for the $S_j$. This ghost circle includes all the global minimizers of $W_{p,q}$. It consists of stationary points of index $0$ and index $1$ and heteroclinic orbits of the negative gradient flow.
\end{theorem}
\noindent The construction of this ghost circle is essentially the same as the construction in dimension $d=1$ provided by Gol\'e \cite{gole92}. We nevertheless decided to provide the proofs. \\
\indent To prove Theorem \ref{ghost_circles}, we need two lemmas and the following definition: 
\begin{definition}
We say that $x\ll y$ are {\it consecutive} index-$0$ stationary configurations if there is no index-$0$ stationary configuration $z$ with $x\ll z \ll y$. 
\end{definition}
\noindent It turns out that when $W_{p,q}$ is a Morse function, then between consecutive index-$0$ stationary configurations we can find another critical point:
\begin{lemma}[Mountain pass theorem]\label{mountain pass}
Assume that $W_{p,q}:\X_{p,q}\to \R$ is Morse and let $x \ll y$ be two consecutive index-$0$ stationary configurations of $-\nabla W_{p,q}$. Then there is an index-$1$ stationary configuration $z$ in between $x$ and $y$. 
\end{lemma}

\begin{proof}
We use a simple variant of the mountain pass theorem, see for instance \cite{evans98}, Section 8.5.1. For this purpose, we let $\mathcal{C}$ be the collection of curves from $x$ to $y$ lying in the order interval $K:=[x,y]$, that is
$$\mathcal{C}:=\{ \gamma:[0,1]\to K \ | \ \gamma(0)=x, \  \gamma(1)=y\ \mbox{and} \ \gamma \ \mbox{is continuous}  \}\ .$$
We now claim that there is a critical point $z \in \mathring{K}$ for which $W_{p,q}(z) = c$, where
$$c:=\inf_{\gamma \in \mathcal{C}} \max_{0\leq t\leq1} W_{p,q}(\gamma(t))\ .$$ 
\noindent To prove our claim, let us define, for $\delta \in \R$, the sub-levelsets 
$$K^{\delta}:=\{x\in K\ | \ W_{p,q}(x)\leq \delta \}\ .$$
These $K^{\delta}$ are invariant under the forward flow of $\frac{dx}{dt} = -\nabla W_{p,q}(x)$. This is true because $K$ is invariant and because $W_{p,q}$ is a Lyapunov function for the gradient flow. \\
\indent Suppose now that there is no critical point $x\ll z \ll y$ with $W_{p,q}(z)=c$. We will show that this leads to a contradiction. \\
\indent We first of all remark that, by the Morse lemma and the fact that $x$ and $y$ have index $0$, it holds that  $c>\max\{W_{p,q}(x), W_{p,q}(y)\}$. Thus, because there are only finitely many critical points in $K$, and none of these except $x$ and $y$ lie in $\p K$, there exists an $\varepsilon>0$ so that the set $K^{c+\varepsilon}\backslash K^{c-\varepsilon}$
does not contain any critical points. \\
\indent This in turn implies, by compactness, that there is a $\sigma>0$ so that $||\nabla W_{p,q}||^2 > \sigma$ on $K^{c+\varepsilon}\backslash K^{c-\varepsilon/2}$. Hence, a solution curve $t\mapsto x(t)$ of the negative gradient flow satisfies $\frac{d}{dt}W_{p,q}(x(t)) = -||\nabla W_{p,q}(x(t))||^2 < -\sigma$ so long as $x(t)\in K^{c+\varepsilon}\backslash K^{c-\varepsilon/2}$. In particular, there is a $T>0$ for which $\Psi_{T}(K^{c+\varepsilon})\subset K^{c-\varepsilon/2}$. 
\\ \indent At the same time, by definition of $c$, there exists a $\gamma\in \mathcal{C}$ with $\gamma([0,1]) \subset K^{c+\varepsilon}$. The curve $\Psi_{T}\circ \gamma\in \mathcal{C}$ then lies entirely in $K^{c-\varepsilon/2}$. This contradicts the definition of $c$ and hence there must be critical points $x\ll z_1, \ldots, z_m \ll y$ with $W_{p,q}(z_i)=c$. \\
\indent It remains to show that at least one of the $z_i$ has index one. In fact,  the argument is a bit subtle. We start by observing the following:
\begin{itemize}
\item[1.] If $x\ll z_i \ll y$ is an index-$0$ critical point with $W_{p,q}(z_i)=c$, then there are $\alpha_i, \beta_i>0$ so that whenever $\gamma\in \mathcal{C}$ intersects $B_{\alpha_i}(z)$, then $\max_{t\in [0,1]} W_{p,q}(\gamma(t))\geq c+\beta_i$. 
\item[2.]  If $x\ll z_i \ll y$ is an index-$\geq 2$ critical point with $W_{p,q}(z_i)=c$, then there is an $\alpha_i>0$ so that whenever $\gamma\in \mathcal{C}$ intersects $B_{\alpha_i}(z_i)$, then $\gamma$ is homotopic to a curve $\tilde \gamma \in \mathcal{C}$ with the property that $\tilde \gamma$ does not intersect $B_{\alpha_i}(z_i)$, while $\max_{t\in [0,1]} W_{p,q}(\tilde \gamma(t))\leq \max_{t\in [0,1]} W_{p,q}(\gamma(t))$. 
\end{itemize}
These statements are easy to prove in local Morse coordinates near the critical point $z_i$. At the same time, by compactness, we have that there exist $\varepsilon, \sigma_1, \sigma_2 >0$ so that $\sigma_1 < ||\nabla W_{p,q}||^2<\sigma_2$ on $K^{c+\varepsilon}\backslash \left(K^{c-\varepsilon/2}\right.$ $\cup B_{\beta_1/2}(z_1)$ $\left. \cup \ldots \cup B_{\beta_m/2}(z_m)\right)$. Using that $||\frac{dx(t)}{dt}||=||\nabla W_{p,q}(x(t))||$ and $\frac{d}{dt} W_{p,q}(x(t)) = -||\nabla W_{p,q}(x(t))||^2$ for solutions of the gradient flow, one can prove quite easily that this implies:
\begin{itemize}
\item[3.]
If $t\mapsto x(t)$ solves $\frac{dx}{dt}=-\nabla W_{p,q}(x)$ and $W_{p,q}(x(0)) \leq c+\varepsilon$ and $x(0)\notin B_{\beta_1}(z_1)\cup$ $\ldots$ $\cup$ $B_{\beta_m}(z_m)$, then for $0\leq t\leq T:=  \min\{\beta_1,\ldots, \beta_m\}/2\sqrt{\sigma_2}$ one has that $x(t)\notin \cup_{i=1}^mB_{\beta_i/2}(z_i)$ and hence $W_{p,q}(x(T)) \leq \max\{c-\varepsilon/2, W_{p,q}(x(0)) - T\sigma_1\}$. 
\end{itemize}
We now use these facts as follows: Let us assume that none of the $z_i$ has index $1$ and let $\gamma_n\in \mathcal{C}$ be a sequence of curves with $c\leq \max_{t\in [0,1]} W_{p,q}(\gamma_n(t)) \leq c+\frac{1}{n}$. By property 1. we know that for large enough $n$, the curve $\gamma_n$ does not intersects $B_{\alpha_i}(z_i)$ for any of the index-$0$ points $z_i$. At the same time, by property 2. we may assume that none of the $\gamma_n$ intersects the $B_{\alpha_i}(z_i)$ for any of the index-$\geq 2$ points $z_i$. Property 3. then implies that for large enough $n$ we have that $(\Psi_{T}\circ \gamma_{n})([0,1]) \subset K^{c-\delta}$ for some $\delta>0$. This contradicts the definition of $c$.
\end{proof}

\noindent The next step is to show that the unstable manifold of the index-$1$ critical point $z$ of Lemma \ref{mountain pass} defines $C^1$ ordered heteroclinic connections to its neighboring index-$0$ critical points $x$ and $y$. This result, in more generality, can also be found in \cite{angenent88}, see Theorem 1 of Chapter 4.

\begin{lemma}\label{heteroclinic}
Let $W_{p,q}:\X_{p,q}\to \R$ be Morse, let $x  \ll y$ be two consecutive index-$0$ stationary configurations of $W_{p,q}$ and let $z$ be an index-$1$ stationary configuration with $x \ll z \ll y$. Then the unstable manifold of $z$ forms strictly ordered heteroclinic connections from $z$ to $x$ and $y$.
\end{lemma}

\begin{proof} We consider the linearization of the negative gradient vector field at $z$, given by the matrix $- D^2W_{p,q}(z)$. The twist condition D and the bound on the second derivatives E together guarantee that there exists a constant $M>0$ so that the symmetric matrix $H:=- D^2W_{p,q}(z) + M\mbox{Id}$ is nonnegative and strictly positive on its diagonal and its two off-diagonals. \\
\indent By the theorem of Perron-Frobenius, $H$ then has to have a unique simple largest eigenvalue $\lambda_{\rm max} + M \in \R_{+}$ and the corresponding eigenvector $e_{\rm max}$ can be chosen strictly positive. Because $z$ is an index-$1$ point, $\lambda_{\rm max}$ is then the unique positive eigenvalue of $-D^2W_{p,q}(z)$ and $e_{\rm max}$ is its strictly positive eigenvector. \\
\indent The unstable manifold $\mathcal{W}^u(z)$ of $z$ is one-dimensional and at $z$ it is tangent to $e_{\rm max}$. In fact, it consists of $z$ and two orbits of the negative gradient flow 
$$\alpha^{\pm}(t)= z \pm e^{\lambda_{\rm max} \cdot t} \: e_{\rm max} + o(e^{\lambda_{\rm max}t})\ \mbox{for} \ t\to - \infty\ .$$ 
In particular we see that close to $z$, the unstable manifold is strictly ordered, because $e_{\rm max}$ is strictly positive. Theorem \ref{monotonicity psi} then implies that the entire $\mathcal{W}^u(z)$ is strictly ordered. Thus, we see that there must be two critical points $z^{-}:=\lim_{t \to \infty}\alpha_-(t)$ and $z^{+}:=\lim_{t \to \infty}\alpha_+(t)$. We claim that $z^-=x$ and $z^+=y$.\\
\indent To prove this, we will show that $z^-$ and $z^+$ are index-$0$ critical points. Our claim then follows because $x\leq z^- \ll z \ll z^+ \leq y$ by monotonicity and because $x$ and $y$ are consecutive index-$0$ points. \\ 
\indent So let us consider the linearization matrix $-D^2W_{p,q}(z^-)$. It also has a unique maximal eigenvalue $\lambda_{\rm max}^-$ and positive eigenvector $e_{\rm max}^-$. We know that $\lim_{t\to \infty} \alpha^-(t)= z^-$ and that $\alpha^-(\R)$ is strictly ordered.  At the same time, because $-D^2W_{p,q}(z^-)$ is symmetric, $e_{\rm max}^-$ is perpendicular to all other eigenvectors of $-D^2W_{p,q}(z^-)$, which implies that none of these other eigenvectors lies in the positive or the negative quadrant.
This means that $\alpha^{-}(t)$ has to approach $z^-$ tangent to $e^-_{\rm max}$, that is $$\alpha^{-}(t)= z^- + e^{\lambda_{\rm max}^-t} e_{\rm max}^- + o(e^{\lambda_{\rm max}^-t})\ \mbox{for} \ t\to \infty \ .$$ 
In particular, $\lambda_{\rm max}^-<0$. But $\lambda^-_{\rm max}$ is the maximal eigenvalue of $-D^2W_{p,q}(z^-)$. This means that all eigenvalues of $D^2W_{p,q}(z^-)$ are positive, i.e. that $z^-$ is an index-$0$ point.\\
\indent A similar argument for $z^+$ finishes the proof.
\end{proof}
\noindent We conclude with a definition and then give the proof of Theorem \ref{ghost_circles}.
\begin{definition}
A nonempty and strictly ordered collection of configurations
$$C_0=\{ \ldots, x_{-1} \ll x_0 \ll x_1 \ll \ldots \} \subset \X_{p,q}$$ is called a {\it maximal index-$0$ skeleton} for the Morse function $W_{p,q}$ if:
\begin{itemize} 
		\item it consists of index-$0$ critical points of $W_{p,q}$
		\item it is shift-invariant: for all $x\in C_0$ and $(k,l) \in \Z^d \times \Z$, it holds that $\tau_{k,l}x \in C_0$
		\item it is maximal: if $y \notin C_0$ is an index-$0$ point, then there is no $i\in \Z$ with $x_i \ll y \ll x_{i+1}$.
	\end{itemize}
\end{definition}
\begin{proof}[Proof of Theorem \ref{ghost_circles}]
We remark that a maximal index-$0$ skeleton in general is not unique, but it is not hard to see that a maximal index-$0$ skeleton exists if $W_{p,q}$ is Morse. \\
 \indent Indeed, one can construct one by starting with the strictly ordered, shift-invariant collection $C_0^0=\{ \ldots, \bar x_{-1}, \bar x_0, \bar x_1, \ldots \}$ of all the global minimizers of $W_{p,q}$. We note that $C_0^0$ is discrete because $W_{p,q}$ is a Morse function.\\
 \indent If there exists an index-$0$ point $x\notin C^0_0$ with the property that $\bar x_{i}\ll x \ll \bar x_{i}$ for some $i$, then one augments $C_0^0$ by the $\tau$-orbit of this $x$, thus obtaining the strictly ordered, shift-invariant and discrete collection 
 $$C_0^1:= C_0^0 \cup\{\tau_{k,l}x\ | \ (k,l)\in\Z^d \times \Z\}\ . $$ 
 One keeps on adding $\tau$-orbits of index-$0$ points this way. The Morse property of $W_{p,q}$ guarantees that the number of index-$0$ points between $\bar x_0$ and $\bar x_0 +1$ is finite, which implies that this process stops after finitely many steps.\\
 \indent The maximality of an index-$0$ maximal skeleton $C_0=\{ \ldots, x_{-1} \ll x_0 \ll x_1 \ll \ldots \}$ just means that the pairs $x_i, x_{i+1}$ are consecutive index-$0$ points. The Mountain Pass Lemma guarantees that between these consecutive elements, there is an index-$1$ critical point $z_i$, while Lemma \ref{heteroclinic} says that the unstable manifold of this $z_i$ defines strictly ordered heteroclinic connections from $z_i$ to $x_i$ and $x_{i+1}$. \\
 \indent If we choose the $z_i$ in such a way that $\overline{C_0}:=\{\ldots, x_{-1}\ll z_{-1}\ll x_0 \ll z_0\ll x_1 \ll z_1 \ll \ldots\}$ is shift-invariant, then the union of $\overline{C_0}$ and these heteroclinic connections is a ghost circle $\Gamma$. The construction above shows that $\Gamma$ may be assumed to contain all global minimizers of $W_{p,q}$.
\\ 
\indent  It only remains to show that this $\Gamma$ is $C^1$. This is clear except at the critical points. But in the proof of the Lemma \ref{heteroclinic}, we have seen that at the critical points, the heteroclinic connections are tangent to the dominant eigenvector. This eigenvector is simple and hence, $\Gamma$ is $C^1$ also at critical points.
\end{proof}

\section{Convergence of ghost circles}\label{convergence}
Section \ref{PMGC} was devoted to the construction of periodic ghost circles for action functions that satisfy the Morse property. In this section we will prove the existence of periodic ghost circles for arbitrary action functions. In turn, this will then imply the existence of ghost circles with irrational rotation vectors. These results follow from a compactness theorem for ghost circles that we will prove below. Before we can formulate it, let us specify what it means for a sequence of ghost circles to converge:

\begin{definition}[Convergence of ghost circles]
We say that a sequence of ghost circles $\Gamma_n$ converges to a ghost circle $\Gamma_{\infty}$, if for every $\xi\in\R$, the sequence of configurations $x^n(\xi) \in \Gamma_n$ defined by $\pi_0(x^n(\xi))=\xi$ converges pointwise to the configuration $x^{\infty}(\xi) \in \Gamma_{\infty}$ defined by $\pi_{0}(x^{\infty}(\xi))=\xi$.
\end{definition}
\noindent Thus, if $\Gamma_n\to\Gamma_{\infty}$ as $n\to\infty$ then $\Gamma_{\infty}$ consists of pointwise limits of elements of the $\Gamma_n$. \\
\indent Before stating the most important results of this section, let us make a few simple observations concerning convergence of ghost circles. First of all, one can observe that if $\lim_{n\to\infty}\omega_n= \omega_{\infty}$ and if a sequence of ghost circles $\Gamma_n\subset \mathcal{B}_{\omega_n}$ converges to $\Gamma_{\infty}$, then it must be true that $\Gamma_{\infty}\subset \mathcal{B}_{\omega_{\infty}}$. This follows from the continuity of the rotation vector as a function on $\mathcal{B}$ and the fact that the rotation vector of a ghost circle is defined as the rotation vector of any of its elements.\\
\indent The second remark is that if the $\Gamma_n\subset \overline{\mathcal{B}}_{\omega}$ are periodic ghost circles with the same rational rotation vector, and $\Gamma_{n}\to\Gamma_{\infty}$, then $\Gamma_{\infty} \subset  \overline{\mathcal{B}}_{\omega}$ is periodic as well. This follows because $\overline{\mathcal{B}}_{\omega}$ is a closed subset of $\R^{\Z^d}$.\\ 
\indent Our compactness result now is the following:
\begin{theorem}\label{compactnessghostcircle}
Let $\omega_n\in K\subset \R^d$ be a sequence of rotation vectors contained in a compact set $K$ and converging to $\omega_{\infty}\in K\subset \R^d$ and let $S_j^n:\R^{\Z^d}\to\R$ be a sequence of local potentials such that $\nabla S_j^n$ converge to $\nabla S_j^{\infty}$ uniformly in $C^{1}(\R^{\Z^d})$. Finally, let $\Gamma_n$ be a sequence of ghost circles for the $S_j^n$ of rotation vector $\omega_n$. Then there exists a ghost circle $\Gamma_{\infty}$ for the $S_j^{\infty}$ of rotation vector $\omega_{\infty}$ and a subsequence $\{n_j\}_{j=1}^{\infty}$ such that  $\lim_{j\to\infty}\Gamma_{n_j}=\Gamma_{\infty}$.\\
\indent If moreover $\lim_{n\to\infty} S_j^n = S_j^{\infty}$ uniformly in $C^0(\mathcal{B}_K)$ and if every $\Gamma_n$ contains a global minimizer, then also $\Gamma_{\infty}$ contains a global minimizer.
\end{theorem}
\noindent Before proving this compactness result, let us formulate its two main implications:
\begin{theorem}\label{existenceperiodic}
Let $\omega\in\mathbb{Q}^d$ be arbitrary and let the local potentials $S_j$ be given. Then there exists a periodic ghost circle $\Gamma_{\omega}\subset \overline{\mathcal{B}}_{\omega}$ for the $S_j$. This $\Gamma_{\omega}$ may be chosen so that it contains a global minimizer.
\end{theorem}
\begin{proof}
Given $\omega\in\mathbb{Q}^d$ and any local potentials $S_j$, choose principal periods $(p,q)$ for $\omega$. By Theorem \ref{morsification s} we can choose a sequence of local potentials $S^n_j$ such that $\lim_{n\to\infty}\nabla S_j^n=\nabla S_j$ uniformly in $C^{1}(\R^{\Z^d})$ and $\lim_{n\to\infty} S_j^n = S_j$ uniformly in $C^0(\mathcal{B}_K)$, while at the same time $W_{p,q}^n:\X_{p,q}\to \R$ is a Morse function. Then, by Theorem \ref{ghost_circles}, there is a ghost circle $\Gamma_n\subset \mathcal{B}_{p,q}=\overline{\mathcal{B}}_{\omega}$ for the local potentials $S_j^n$ that contains a minimizer of $W_{p,q}^n$. By Theorem \ref{compactnessghostcircle}, a subsequence of the $\Gamma_n$ converges to a ghost circle $\Gamma_{\omega}\subset \mathcal{B}_{p,q}= \overline{\mathcal{B}}_{\omega}$ for the local potentials $S_j$. By the second conclusion of Theorem \ref{compactnessghostcircle}, $\Gamma_{\omega}$ contains a global minimizer.
\indent 
\end{proof}
\begin{theorem}\label{existencequasiperiodic}
Let $\omega\in\mathbb{R}^d\backslash \mathbb{Q}^d$ and let the local potentials $S_j$ be given. Then there exists a ghost circle $\Gamma_{\omega}\subset \overline{\mathcal{B}}_{\omega}$ for the $S_j$. This $\Gamma_{\omega}$ may be chosen so that it contains the entire Aubry-Mather set of rotation vector $\omega$.
\end{theorem}
\begin{proof}
Given $\omega\in\mathbb{R}^d$ and local potentials $S_j$, choose a sequence $\omega_n\in \mathbb{Q}^d$ such that $\lim_{n\to\infty} \omega_n =\omega$. By Theorem \ref{existenceperiodic}, there is a periodic ghost circle $\Gamma_n\subset \overline{\mathcal{B}}_{\omega_n}$ for the local potentials $S_j$ that contains at least one global minimizer. By Theorem \ref{compactnessghostcircle}, a subsequence of the $\Gamma_n$ converges to a ghost circle $\Gamma_{\omega}\subset \overline{\mathcal{B}}_{\omega}$. \\ 
\indent The requirement for the second conclusion of Theorem \ref{compactnessghostcircle} is trivially valid, so that $\Gamma_{\omega}$ contains a global minimizer, say $x$. Being closed and shift-invariant, this implies that $\Gamma_{\omega}$ contains the entire Aubry-Mather set $\mathcal{M}(x)$. 
\end{proof}
\noindent Before proving Theorem \ref{compactnessghostcircle}, we remark that if $\Gamma_n$ is an arbitrary sequence of ghost circles for the local potentials $S_j^n$ and with rotation vectors $\omega_n$ in a compact set $K$, then for every $\xi\in \R$ the sequence of configurations $x^{n}(\xi)\in \Gamma_n$ has a subsequence that converges pointwise. This just follows from the compactness of $\mathcal{B}_{K} \cap \{ x \in \R^{\Z^d}\ | \ \pi_0(x)=\xi\}$. The problem is to show that this subsequence can be chosen independent of $\xi$ and that the collection of limit configurations $\{ \lim_{n\to\infty} x^n(\xi)\ | \ \xi  \in \R \}$ forms a ghost circle for the $S_j=\lim_{n\to\infty}S_j^n$.
\\ \mbox{}
\\
\noindent We will now make some preparations for the proof of Theorem \ref{compactnessghostcircle}. To start with, we define for a given ghost circle $\Gamma$, the map
$$T^{\Gamma}:\R \to \Gamma \ \ \mbox{by}\  \ T^{\Gamma}:=\Psi_{-1} \circ (\pi_0|_{\Gamma})^{-1} \! ,\ \mbox{that is:} \ \ T^{\Gamma}_k(\xi) \! := \left(T^{\Gamma}(\xi)\right)_k \!\! = \left( \pi_k\circ \Psi_{-1} \circ (\pi_0|_{\Gamma})^{-1}\right)(\xi)\ .$$ 
Here, $\Psi_{-1}: \X\to\X$ denotes the time-$-1$ flow of $\frac{dx}{dt}=-\nabla W(x)$. By Theorem \ref{psi}, $T^{\Gamma}$ is a homeomorphism, being the composition of two homeomorphisms. Moreover, it is ``pointwise Lipschitz continuous'':

\begin{lemma}\label{translations quasi}
Let $K\subset \mathbb{R}^d$ be a compact set and $\Gamma \subset \mathcal{B}_{K}=\cup_{\omega\in K} \mathcal{B}_{\omega}$ a ghost circle with rotation vector $\omega\in K$ for the local potentials $S_j$ satisfying conditions A-E. Then, for every $k\in \Z^d$, there is a constant $\Lambda_{||k||}>0$, depending only on $K$ and $||k||$ such that 
$$\left|T_k^{\Gamma}(\xi) - T^{\Gamma}_k(\nu) \right| \leq \Lambda_{||k||} |\xi-\nu|.$$
\end{lemma}

\begin{proof}
Let $\xi, \nu\in \R$ and denote $X=(\pi_0|_{\Gamma})^{-1}(\xi)$ and $Y=(\pi_0|_{\Gamma})^{-1}(\nu)$. Assume that $\xi<\nu$, whence $X\ll Y$. Denote by $\Psi_t$ the time-$t$ flow of $-\nabla W$, with $W:=\sum_{j\in\Z^d} S_j$. Since $\Gamma$ is forward and backward invariant under $\Psi$, both $\Psi_{-1}(X)$ and $\Psi_{-1}(Y)$ lie in $\Gamma$ and satisfy $\Psi_{-1}(X)\ll\Psi_{-1}(Y)$. Now we apply the parabolic Harnack inequality of Theorem \ref{UHIP} to  $t=1$, $i=0$, $x=\Psi_{-1}(X)$ and $y=\Psi_{-1}(Y)$, to find that there is an $L>0$ depending only on $K$ and $||k||$ such that
$$T^{\Gamma}_k(\nu) - T^{\Gamma}_k(\xi) = (\Psi_{-1} Y)_k - (\Psi_{-1} X)_k \leq \frac{1}{L} ( Y_0-X_0 )= \frac{1}{L}\left(\nu- \xi \right) \ .$$
A similar argument in the case that $\xi>\nu$ finishes the proof.
\end{proof}

\noindent We remark here that we see no reason why the maps $\pi_k\circ (\pi_0|_{\Gamma})^{-1}$ should be uniformly Lipschitz continuous. This is why we study the maps $\pi_k\circ \Psi_{-1} \circ (\pi_0|_{\Gamma})^{-1}$ instead.

\begin{definition}
We say that a sequence of maps $T^n:\R\to\R^{\Z^d}$ converges {\it pointwise uniformly} to a map $T^{\infty}:\R\to\R^{\Z^d}$ as $n\to\infty$ if for every $k\in \Z^d$ the sequence of maps $T_k^n:=\pi_k\circ T^n:\R\to\R$ converges uniformly to $T_k^{\infty}:=\pi_k\circ T^{\infty}$ as $n\to\infty$.
\end{definition}

\begin{corollary}\label{uniform translations}
Let $K\subset \R^d$ be a compact set. Assume that for every $n\in\N$, we are given a rotation vector $\omega_n\in K$, local potentials $S_j^n$ satisfying conditions A-E and ghost circles $\Gamma_n \subset \mathcal{B}_{\omega_n}$ for the local potentials $S_j^n$. \\ \indent Then there is a subsequence $\{n_j\}_{j\in \N}\subset \N$ with the property that the maps $T^{\Gamma_{n_j}}:\R\to \R$ converge pointwise uniformly on $\R$, say $T^{\Gamma_{n_j}}\to T^{\infty}$ as $j\to \infty$.  Each limit map $T^{\infty}_k:=\pi_k \circ T^{\infty}: \R\to\R$, is non-decreasing, surjective and Lipschitz continuous.
\end{corollary}

\begin{proof}
Fix a $k\in \Z^d$. By Lemma \ref{translations quasi}, the maps $T^{\Gamma_n}_k:=\pi_k\circ T^{\Gamma_n}$ are uniformly-in-$n$ Lipschitz continuous  with Lipschitz constant $\Lambda_{||k||}$. Moreover, by the definition of a ghost circle they are $1$-periodic and increasing.
To see that they are uniformly bounded on compacts, we then just have to note that Proposition \ref{solution} implies that $T_0^{\Gamma}(0)\in[-1,1]$, while Lemma \ref{sequence} and the fact that $(\Psi_{-1}\circ (\pi_0^{\Gamma})^{-1})(\xi)$ is a Birkhoff sequence then imply that $T_k^{\Gamma}(0)\in [-2-||K||\cdot ||k||, 2 + ||K||\cdot ||k||]$, where $||K||:=\max_{\omega\in K}||\omega||$. \\ 
\indent Thus, the theorem of Arzel\`{a}-Ascoli guarantees that there exists a uniformly convergent subsequence $T_k^{n_{j,k}} \to T^{\infty}_k$ for $j\to \infty$. Clearly, $T_k^{\infty}$ is nondecreasing and Lipschitz continuous with Lipschitz constant $\Lambda_{||k||}$. \\ 
\indent  Let $j\mapsto k_j,  \N\to\Z^d$ be a denumeration of $\Z^d$. Then the diagonal sequence $\{n_{j}\}_{j\in \N} \subset \N$ defined by $n_j:=n_{j,k_j}$ has the property that $T^{\Gamma_{n_j}} \to T^{\infty}$ pointwise uniformly as $j\to\infty$.
\end{proof}
 
\begin{theorem}[Convergence of ghost circles]\label{qpghostcircletheorem}
Let $\Gamma_n$ be a sequence of ghost circles for the local potentials $S_j^n$. Assume that there are local potential functions $S_j^\infty$ such that $\nabla S_j^n\to \nabla S_j^{\infty}$ uniformly in $C^1(\R^{\Z^d})$ and that the maps $T^{\Gamma_n}$ converge pointwise uniformly. Then there is a ghost circle $\Gamma_{\infty}$ for the local potentials $S_j^{\infty}$ such that $\Gamma_n\to \Gamma_{\infty}$ as $n\to\infty$.\\
\indent Moreover, when $\Gamma_n$ contains a global minimizer $x_n$ and $\lim_{n\to\infty} S_j^n = S_j^{\infty}$ uniformly in $C^0(\mathcal{B}_K)$, then $\Gamma_{\infty}$ contains a global minimizer as well. 
\end{theorem}
\noindent Theorem \ref{compactnessghostcircle} now follows directly from Corollary \ref{uniform translations} and Theorem \ref{qpghostcircletheorem}. \\
\indent Before we prove Theorem \ref{qpghostcircletheorem}, let us recall that if $\Gamma_{\infty}=\lim_{n\to\infty} \Gamma_n$ exists, then it must be equal to
$$\Gamma_{\infty}:=\{ x^{\infty}(\xi) := \lim_{n\to\infty} x^n(\xi) \ \mbox{pointwise}\ | \ \xi \in \R \}\ .$$
At this point, it is of course not clear whether the limit $\lim_{n\to\infty} x^n(\xi)$ exists for every $\xi\in \R$. To see that it does under the conditions of Theorem \ref{qpghostcircletheorem}, we note that $x^n(\xi)=\Psi_1^n(T^n(\xi))$, so that
$$\lim_{n\to\infty} x^n(\xi) = \lim_{n\to \infty} \Psi_1^n(T^n(\xi)) \ \mbox{exists and is equal to}\ \Psi_1^{\infty}(T^{\infty}(\xi)) .$$
This is true because on the one hand, according to Corollary \ref{continuousflow}, $\Psi_1^n\to\Psi_1^{\infty}$ uniformly in the topology of pointwise convergence, while on the other hand it holds that for every $\xi\in\R$, the sequence of configurations $T^{\Gamma_n}(\xi)\in \mathcal{B}$ converges pointwise to the configuration $T^{\infty}(\xi)$ as $n\to\infty$, because $T^{\Gamma_n}\to T^{\infty}$ pointwise uniformly. Thus we find that under the conditions of Theorem \ref{qpghostcircletheorem}, $\Gamma_{\infty}:=\lim_{n\to\infty}\Gamma_n\subset \mathcal{B}_{\omega_{\infty}}$ is well defined and moreover that, if $\Gamma_{\infty}$ is a ghost circle, then $T^{\Gamma_{\infty}}=T^{\infty}$.
 \\ \indent We will now show that $\Gamma_{\infty}$ is in fact a ghost circle for the local potentials $S_j^{\infty}$:

\begin{proof}[Proof of Theorem \ref{qpghostcircletheorem}] We first check that $\Gamma_{\infty}$ has the properties required for a ghost circle:
\begin{itemize}
\item[1.] {\bf Closedness:} Let $x^{\infty}(\xi_m) \in \Gamma_{\infty}$ be a sequence of configurations that converges pointwise. This implies that the $\xi_m$ converge, say to $\xi$. We now want to show that $\lim_{m\to\infty} x^{\infty}(\xi_m)=x^{\infty}(\xi)\in \Gamma_{\infty}$ pointwise. This follows because $\lim_{m\to\infty} x^{\infty}(\xi_m)= \lim_{m\to\infty} \Psi_1^{\infty}(T^{\infty}(\xi_m)) = \Psi_1^{\infty}(\lim_{m\to\infty} T^{\infty}(\xi_m)) = \Psi_1^{\infty}(T^{\infty}(\xi)) = x^{\infty}(\xi)$. All these limits are pointwise. We have used that $\Psi_1^{\infty}$ is continuous for pointwise convergence and that $\lim_{m\to\infty} T^{\infty}(\xi_m)=T^{\infty}(\xi)$ pointwise. 
\item[2.] {\bf Connectedness:} We note that $\Gamma_{\infty}= \Psi_1^{\infty}(T^{\infty}(\R))$, so it is the image under a continuous map of a connected set, hence connected.  
\item[3.] {\bf Strict ordering:} Suppose $\xi<\nu$. Recall that $T^{\infty}$ is nondecreasing, so $T^{\infty}(\xi) \leq T^{\infty}(\nu)$. We remark that $T^{\infty}(\xi)$ cannot equal $T^{\infty}(\nu)$, because this would imply that $\xi = \pi_0(x^{\infty}(\xi)) = \pi_0( \Psi_1^{\infty}(T^{\infty}(\xi))) = \pi_0(\Psi_1^{\infty}(T^{\infty}(\nu)) = \pi_0(x^{\infty}(\nu) ) = \nu$. Thus $T^{\infty}(\xi)< T^{\infty}(\nu)$. \\ \indent
The strict monotonicity of the negative gradient flow, then implies that $x^{\infty}(\xi) = \Psi_1^{\infty}(T^{\infty}(\xi)) \ll \Psi_1^{\infty}(T^{\infty}(\nu)) = x^{\infty}(\nu) $. 
\item[4.] {\bf Shift-invariance:} Let $x^{\infty} \in \Gamma_{\infty}$, that is $x^{\infty}=\lim_{n\to\infty} x^n$ with $x^n \in \Gamma_n$ and $\pi_0(x^n)=\pi_0(x^{\infty})$. Let $k\in \Z^d$ and $l\in \Z$ be given. 
We want to show that $\tau_{k,l}x^{\infty}\in \Gamma_{\infty}$, that is we want to show that $\tau_{k,l}x^{\infty} = \lim_{n\to \infty} y^n$ with $y^n\in \Gamma_n$ such that $\pi_0(y^n)= \pi_0(\tau_{k,l}x^{\infty})$. We prove this by writing $$\lim_{n\to \infty} \left(\tau_{k,l}x^{\infty} - y^n \right) = \lim_{n\to \infty} \left(\tau_{k,l}x^{\infty} - \tau_{k,l}x^n\right) + \lim_{n\to\infty} \left( \tau_{k,l}x^n- y^n \right)$$ 
and showing that both limits on the right hand side vanish. \\
\indent The first limit is zero because, by Lemma \ref{lipschitz translations}, $\tau_{k,l}$ is continuous in the topology of pointwise convergence. Thus we have that $\lim_{n\to\infty} \left(\tau_{k,l}x^{\infty} - \tau_{k,l}x^n \right)= \lim_{n\to\infty} \tau_{k,0}(x^{\infty} - x^n) = 0$. \\
\indent For the second limit, we realize that $\tau_{k,l}x^n\in \Gamma_n$ because $\Gamma_n$ is shift-invariant and we observe that $\pi_0(\tau_{k,l}x^n) =x^n_k + l$. Because $T^{\Gamma_n}\to T^{\infty}$ pointwise uniformly, we moreover know that $\lim_{n\to\infty} T^{\Gamma_n}(x^n_k+l) = T^{\infty}(x^{\infty}_k + l)$ pointwise. Thus, by the uniform convergence of the $\Psi_1^n$ to $\Psi_1^{\infty}$,
$$\lim_{n\to\infty}\!\left( \tau_{k,l}x^n \right)\!=\! \lim_{n\to\infty} \Psi_1^n(T^n(x_k^n+l))\! =\! \Psi_1^{\infty}(T^{\infty}(x^{\infty}_k + l)) \!=\!  \lim_{n\to\infty}\! \Psi_1^{\infty}(T^{n}(x^{\infty}_k + l))\! =\! \lim_{n\to\infty}\! y^n . $$ 
\item[5.] {\bf Flow-invariance:} This is proved in a similar way as shift-invariance. So, let $x^{\infty} \in \Gamma_{\infty}$, that is $x^{\infty}=\lim_{n\to\infty} x^n$ with $x^n \in \Gamma_n$ and $\pi_0(x^n)=\pi_0(x^{\infty})$. Let $t\in \R$ be given. 
We want to show that $\Psi_t^{\infty} x^{\infty}\in \Gamma_{\infty}$, that is we want to show that $\Psi_t^{\infty} x^{\infty} = \lim_{n\to \infty} y^n$ with $y^n\in \Gamma_n$ such that $\pi_0(y^n)= \pi_0(\Psi_t^{\infty} x^{\infty})$. We prove this by writing $$\lim_{n\to \infty} \left(\Psi_t^{\infty} x^{\infty} - y^n \right) = \lim_{n\to \infty} \left(\Psi_t^{\infty} x^{\infty} - \Psi_t^n x^n\right) + \lim_{n\to\infty} \left( \Psi^n_t x^n- y^n \right)$$ 
and showing that both limits on the right hand side vanish. \\
\indent The first limit is zero because, by Theorem \ref{psi}, $\Psi_t^n$ converges to $\Psi_t^{\infty}$ uniformly. Thus we have that $\lim_{n\to\infty} \Psi_t^n x^{n} = \Psi_t^{\infty} x^{\infty}$. \\
\indent For the second limit, we realize that $\Psi_t^n x^n\in \Gamma^n$ because $\Gamma_n$ is flow-invariant and we observe that $\lim_{n\to\infty} \pi_0(\Psi_t^n x^n) = \pi_0 (\Psi_t^{\infty}(x^{\infty}))$ because $\pi_0$ is continuous for pointwise convergence. Because $T^{\Gamma_n}\to T^{\infty}$ pointwise uniformly, we therefore know that $\lim_{n\to\infty} T^{\Gamma_n}(\pi_0(\Psi^{\infty}_t(x^n))) = T^{\infty}(\pi_0(\Psi_t^{\infty} x^{\infty}))$ pointwise. Thus, by the uniform convergence of the $\Psi_1^n$ to $\Psi_1^{\infty}$,
$$\lim_{n\to\infty}\! \Psi_t^n x^n \!=\! \lim_{n\to\infty} \Psi^n_1(T^{\Gamma_n}(\pi_0(\Psi^n_tx^n)))\! =\! \lim_{n\to\infty} \Psi_1^{\infty}(T^{\Gamma_n}(\pi_0(\Psi_t^{\infty} x^{\infty}))) \!=\!  \lim_{n\to\infty}\! y^n . $$ 
\end{itemize}
We finish the proof of Theorem \ref{qpghostcircletheorem} by proving that when each $\Gamma_n$ contains a minimizer and $\lim_{n\to\infty} S_j^n=S_j^{\infty}$ uniformly in $C^0(\mathcal{B}_K)$, then also $\Gamma_{\infty}$ contains a minimizer: \\ 
\mbox{} \\
\noindent {\bf Minimizing property:} Suppose that every ghost circle $\Gamma_n$ contains a minimizer $x^n=x^n(\xi_n)$. This means that for every finite subset $B\subset \Z^d$ and every $y:\Z^d\to \R$ with finite support in $\mathring{B}^{(r)}$ it holds that 
\begin{align}\label{minprop}
W_B^n(x^n+y) - W_B^n(x^n) \geq 0 \ , \ \mbox{where}\ W_B^n(x):=\sum_{j\in B}S_j^n(x)\ .
\end{align}
By compactness of $\mathcal{B}_K/\Z$, a subsequence of the $x^n(\xi_n)$ converges pointwise, say to $x^{\infty}=\lim_{j\to\infty} x^{n_j}(\xi_{n_j})$. Moreover, $\lim_{n\to \infty}W_B^n = W_B^{\infty}$ uniformly in $C^0(\mathcal{B}_K)$. Taking the limit of equation (\ref{minprop}) as $n_j\to \infty$ then shows that $W_B^{\infty}(x^{\infty}+y)-W_B^{\infty}(x^{\infty})\geq 0$. In other words, $x^{\infty}$ is a global minimizer.\\
\indent  It remains to prove that $x^{\infty} \in \Gamma_{\infty}$. This holds because $x^{\infty}=\lim_{j\to\infty}x^{n_j}(\xi_{n_j}) = \lim_{j\to\infty}\Psi_1^n(T^n(\xi_{n_j})) = \Psi_1^{\infty}(T^{\infty}(\xi_{\infty})) = x^{\infty}(\xi_{\infty})$,  where $\xi_{\infty}:=\lim_{j\to\infty} \xi_{n_j}$.
\end{proof}

\section{Gap solutions}\label{gap_solutions}
In this final section we examine the situation that an Aubry-Mather set $\mathcal{M}(x)\subset \mathcal{B}_{\omega}$ has a gap, that is when there are elements $y^-, y^+\in \mathcal{M}(x)$ with $y^-\ll y^+$ such that $[y^-, y^+]$ does not contain any elements of $\mathcal{M}(x)$ other than $y^-$ and $y^+$. This situation occurs when $\omega\in \Q^d$ or when $\omega\in \R^d\backslash \Q^d$ and $\mathcal{M}(x)$ is a Cantor set. \\
\indent The main result of this section is Theorem \ref{gapsolution} below, which states that either $[y^-, y^+]$ admits a foliation by global minimizers, or there exists at least one stationary configuration $z\in [y^-, y^+]$ that is not a global minimizer. This result is more precise than the result of \cite{llave-valdinoci07}, that says that a gap must contain at least one stationary solution. Moreover, the proof below is more geometric, as it makes use of ghost circles.\\
\indent We start with the following theorem, which is a refinement of a result by Moser \cite{moser89}. It says that when a gap admits a foliation by stationary points, then all of them are minimizing. Recall that every Aubry-Mather set is contained in a ghost circle. 
\begin{theorem}\label{moserfoliation}
Let $[y^-, y^+]$ be a gap in the Aubry-Mather set $\mathcal{M}(x)$ and let $\Gamma$ be a ghost circle so that $\mathcal{M}(x)\subset \Gamma$. If $\Gamma^{[y^-,y^+]}:=\Gamma \cap [y^-, y^+]$ consists of stationary configurations only, then all of them are global minimizers.
\end{theorem}
\begin{proof}
Assume that $\Gamma^{[y^-,y^+]}$ consists of stationary points only but that $w\in \Gamma^{[y^-, y^+]}$ is not a global minimizer. Then there is a finite subset $B \subset \Z^d$ and a $z \in \R^{\Z^d}$ with $\mbox{supp}(z) \subset \mathring{B}^{(r)}$ such that $W_B(w+z)<W_B(w)$. Because the function $z\mapsto W_B(w+z)$ is coercive, it attains its minimum, let's say at a $Z$ with support in $\mathring{B}^{(r)}$. By assumption $Z\neq 0$. Let's say there is an $i\in \mathring{B}^{(r)}$ for which $Z_i>0$. In the case that $Z_i<0$ the proof is similar. We now claim that $Z$ can be chosen so that $w+Z \leq y^+$. \\
\indent To prove this claim, we remark that when $m:=\min\{w+Z, y^{+}\}$ and $M:=\max\{w+Z, y^{+}\}$, then $W_B(w+Z)+W_B(y^+)\geq W_B(m)+W_B(M)$, as in the proof of Lemma \ref{maxminprop}. Because both $(w+Z) - m$ and $y^+ - M$ are supported in $\mathring{B}^{(r)}$ and both $w+Z$ and $y^+$ minimize $W_B$ with respect to variations supported in $\mathring{B}^{(r)}$, it must therefore hold that $W_B(w+Z)=W_B(m)=W_B(w+\min\{Z,y^{+}-w\})$.\\
\indent The next step is to define $y:=\inf\{\tilde y \in \Gamma \ | \ \tilde y \gg w+Z\}$. Because $w, y, y^+ \in \Gamma$, $w+Z\leq y^+$ and $Z_i>0$, it now holds that $w \ll y \leq y^+$. At the same time, because $\Gamma$ is connected, $y$ touches $w+Z$. That is: there is an $i\in \mathring{B}^{(r)}$ so that $y_i=w_i+Z_i$, while $z_k+W_k=z_k< y_k$ for all $k\notin \mathring{B}^{(r)}$. We claim that this is impossible. \\
\indent To prove this, choose such an $i\in \mathring{B}^{(r)}$ at which $y_i=w_i+Z_i$ and a $k\in \Z^d$ for which $||i-k||=1$. Then, because $y$ is a global stationary point and $w+Z$ is stationary for $W_B$ with respect to variations in $\mathring{B}^{(r)}$, it must be true that 
\begin{align}\nonumber
0  =  \p_iW_B(y)- & \p_iW_B(w+Z) = \!\!
 \sum_{j\in B, l\in \Z^d}\!\! \left( \int_0^1 \!\! \p_{i,l}S_j(t y+(1-t)(w+Z))\ \! dt\right) \cdot (y_l- w_l-Z_l) \\\nonumber 
 \geq &  \int_0^1\!\! \p_{i,k}S_i(ty +(1-t)(w+Z))\ \! dt \cdot (y_k-w_k-Z_k) \ .
 \end{align} 
Here, the inequality holds because $y_i-w_i-Z_i=0$ and $\p_{i,l}S_j\leq 0$ when $i\neq l$ and $y_l-w_l-Z_l\geq 0$ for all $l$. The twist condition that $\p_{i,k}S_i<0$ then guarantees that $y_k=w_k+Z_k$. By induction, one then finds that there is a $k\notin \mathring{B}^{(r)}$ for which $y_k=w_k+Z_k=w_k$. This is a contradiction.  
\end{proof}

\noindent We will now show that when $\Gamma^{[y^-,y^+]}$ does not consist of only stationary points, i.e. minimizers, then it contains at least one non-minimizing stationary point. We do this by finding a stationary point of a ``renormalized action'' function $W_{[y_-, y^+]}: [y^-, y^+] \to [0,\infty)$. In order to define $W_{[y^-,y^+]}$, we need the following well-known technical result that states, when applied to $\mathcal{M}=\mathcal{M}(x)$, that the gaps of an Aubry-Mather set are uniformly summable:
\begin{theorem}\label{l1}
Let $\mathcal{M}$ be any strictly ordered, shift-invariant collection of configurations of rotation vector $\omega \in \R^d$. Let $x, y\in \mathcal{M}$ be so that $x\ll y$ and assume that there exists no $z \in \mathcal{M}$ with $x\ll z \ll y$. Denote $H_{\omega}:= \{i\in\Z^d\ | \ \langle \omega, i\rangle \in \Z \}$. Then
$$\sum_{i\in\Z^d/ H_\omega} |y_i - x_i| \leq 1\ .$$
\end{theorem}
\begin{proof}
We start by remarking that our assumptions on $\mathcal{M}$ imply that $x$ and $y$ are Birkhoff. Now, let $i$ and $j$ be representatives of different equivalent classes of $\Z^d/H_{\omega}$ and let $m, n \in \Z$ be arbitrary. Then, $\langle \omega, j-i\rangle + m-n \neq 0$ and hence, by Proposition \ref{numbertheory}, either $\tau_{j-i, m-n}y \ll y$ or $\tau_{j-i, m-n}x \gg x$. In the first case, actually $\tau_{j-i,m-n}y \leq x$ because there is no element of $\mathcal{M}$ between $x$ and $y$. Evaluating the latter inequality at $i$, we then obtain $y_j +m \leq x_i +n$. In the second case, one finds that $\tau_{j-i,m-n}x\geq y$, whence $x_j+m \geq y_i +n$. In both cases we find that $(x_j+m, y_j+m) \cap (x_i+n, y_i+n) = \emptyset$. \\
\indent For $\xi\in\R$, denote by $[\xi]:=\max \{n\in \Z\ | \ n\leq \xi \}$ and define $\bar x = \inf_{i} (x_i-[x_i])\in [0,1)$. Then, clearly $x_i-[x_i]\geq \bar x$. We claim that $y_i-[x_i] \leq \bar x + 1$. To prove this, note that $x_i-[x_i]\leq x_j-[x_j]+1$ and hence, by the Birkhoff property, $x\ll \tau_{j-i}x+[x_i]-[x_j] +1$. The assumption that $x$ and $y$ are consecutive elements of $\mathcal{M}$ then implies that $y\leq \tau_{j-i}x+[x_i]-[x_j]+1$, that is, $y_i-[x_i]\leq x_j-[x_j]+1$. Hence, $y_i-[x_i]\leq \bar x +1$. \\
\indent
This yields, denoting by $|A|$ the Lebesgue measure of a set $A\subset \R$:
$$\sum_{i\in\Z^d/H_\omega} \!\!\! |y_i - x_i| =\!\!\! \sum_{i\in \Z^d/H_\omega}\!\!\! |(x_i-[x_i], y_i-[x_i])| = \left|\bigcup_{i\in\Z^d/H_{\omega}} \!\!\! (x_i-[x_i], y_i-[x_i]) \right| \leq |(\bar x, \bar x + 1)| = 1\ .$$
\end{proof}

\noindent For  \textit{rationally independent} rotation vectors, for which $H_{\omega}=\{0\}$, Theorem \ref{l1} was stated for the first time by Moser \cite{moser86}.\\
\indent For a gap $[y^-, y^+]$, with $y^-, y^+\in \mathcal{B}_{\omega}$, let us define
$$\overline{[y^-,y^+]}:=\{y\in[y^-, y^+] \ | \ \tau_{k,l}y=y \ \mbox{if} \ \langle \omega, k \rangle + l =0\} .$$
It is not hard to see that $\overline{[y^-, y^+]}\subset\mathcal{B}_{\omega}$. Namely, when $y^-< y < y^+$ and $\langle \omega, k\rangle +l> 0$, then $\tau_{k,l}y > \tau_{k,l}y^-\geq y^+>y$, where the second inequality holds because $[y^-, y^+]$ is a gap. Similarly, $\tau_{k,l}y<y$ when $\langle \omega, k\rangle + l< 0$. Hence, $y$ is Birkhoff once $\tau_{k,l}y=y$ for all $k,l$ with $\langle \omega, k\rangle + l =0$.
We are now ready to define the renormalized action function:
\begin{definition}
When $[y^-, y^+]$ is a gap, we define $W_{[y^-,y^+]}:\overline{[y^-, y^+]}\to [0,\infty)$ by
$$W_{[y^-,y^+]}(y):=\sum_{j\in \Z^d/H_{\omega}} \left(S_j(y)-S_j(y^-)\right) .$$
\end{definition}
\begin{proposition}\label{well_defined}
For every $y\in \overline{[y^-, y^+]}$, the sum $W_{[y^-,y^+]}(y)$ is absolutely convergent. Moreover, $W_{[y^-, y^+]}$ is continuous with respect to pointwise convergence.\end{proposition}
\begin{proof}
The compactness of $\mathcal{B}_{\omega}$ implies that there is a constant $D>0$ so that $|\p_kS_j|\leq D$ uniformly on $\overline{[y^-, y^+]}$. Thus, we compute that for $y^1,y^2\in \overline{[y^-,y^+]}$, 
\begin{align} 
\nonumber
& |W_{[y^-,y^+]}(y^1) - W_{[y^-,y^+]}(y^2)| \leq \sum_{j\in \Z^d/H_{\omega}} |S_j(y^1)-S_j(y^2)| \leq \\ \nonumber 
 \sum_{j\in \Z^d/H_{\omega}} & \sum_{||k-j||\leq r} \left| \int_0^1\p_k S_j(t y^1 + (1-t)y^2)\ \! dt\right|  |y^1_k-y^2_k| \leq (2r+1)^d D 
 \!\! \sum_{k\in \Z^d/H_{\omega}} \!\! |y^1_k-y^2_k| \ .
\end{align}
First of all, this implies that $W_{[y^-,y^+]}(y)$ is well-defined and converges absolutely for $y\in \overline{[y^-,y^+]}$, because $W_{[y^-,y^+]}(y^-)=0$, by definition, and $\sum_{k\in \Z^d/H_{\omega}}|y_k-y_k^-|\leq 1$ by Theorem \ref{l1}. \\ 
\indent Secondly, it is now clear that $W_{[y^-, y^+]}$ is continuous for pointwise convergence, because when $y^n\in \overline{[y^-,y^+]}$ is a sequence of configurationsconverging pointwise, say to $y^{\infty}$, then $\lim_{n\to\infty} \sum_{j\in \Z^d/H_{\omega}} |y^n_j-y^{\infty}_j|=0$, as is quite easy to prove, so that $\lim_{n\to\infty} W_{[y^-,y^+]}(y^n) = W_{[y^-,y^+]}(y^{\infty})$.
\end{proof}
\noindent The next result is harder to prove:
\begin{theorem}
We have $W_{[y^-,y^+]}\geq 0$. Moreover, if a configuration $y \in \overline{[y^-, y^+]}$ is a global minimizer, then $W_{[y^-,y^+]}(y)=0$. In particular, $W_{[y^-,y^+]}(y^-)=W_{[y^-,y^+]}(y^+)=0$. 
\end{theorem}
\begin{proof}
 The proof of this theorem is similar to that of Theorem \ref{converse}. We will sketch it here. \\
\indent  In fact, we will show that when $y^1, y^2\in \overline{[y^-,y^+]}$ and $y^1$ is a global minimizer, then $W_{[y^-,y^+]}(y^1)\leq W_{[y^-, y^+]}(y^2)$. Applied to $y^1=y^-$, this shows that $W_{[y^-,y^+]}\geq 0$, whereas when applied to $y^2=y^-$, it shows that $W_{[y^-,y^+]}(y)=0$ if $y$ is a global minimizer. \\ 
\indent So let $y^1, y^2\in \overline{[y^-, y^+]}$ and suppose that $\varepsilon:= W_{[y^-,y^+]}(y^1)-W_{[y^-, y^+]}(y^2)>0$. It suffices to show that this implies that $y^1$ is not a global minimizer. To prove this, let $(p_1, q_1), \ldots, (p_c, q_c)$ be principal periods for $\omega$ and let $B_{p}\subset \Z^d$ be the fundamental domain for $\Z^d/H_{\omega}$ defined by 
$$B_{p}:=\{i\in \Z^d\ | \ 0\leq \langle i, p_j\rangle <1 \ \mbox{for all} \ 1 \leq j \leq c\  \}\ .$$
Then $W_{[y^-,y^+]}=\left. W_{B_{p}}\right|_{\overline{[y^-,y^+]}}$ and $W_{nB_{p}}(y^1)-W_{nB_{p}}(y^2)=\varepsilon \cdot n^c$. \\
\indent Now define for every $n>r+2$ the configuration $y^n\in [y^-, y^+]$ by $$y^n_i:=\left\{ \begin{array}{ll} y^2_i & \mbox{if}\  i\in \mathring{B}_{np}^{(r)}, \\ y^1_i & \mbox{otherwise}.\end{array} \right.$$ 
Then $y^n-y^1$ is clearly supported in the $r$-interior of $B_{np}$ and it is not too hard to show, using the compactness of $[y^-, y^+]$, the uniform $l_1$-bound on $[y^-, y^+]$ and the argument of Theorem \ref{converse}, that there is a constant $E>0$ so that 
$$W_{nB_{p}}(y^1)-W_{nB_{p}}(y^n) \! =\! \left(W_{nB_{p}}(y^1) - W_{nB_{p}}(y^2) \right) + \left( W_{nB_{p}}(y^2)  -W_{nB_{p}}(y^n) \right) \! > \! \varepsilon \cdot n^c - E \cdot n^{c-1}\ .$$
This means that $y^1$ is not a global minimizer. 
\end{proof}
\noindent One can in fact also prove a variant of Theorem \ref{global} that says that if $W_{[y^-, y^+]}(y)=0$, then $y$ is a global minimizer. Since we do not need this result in this paper, we will not prove it here.\\
\indent Recall that both $[y^-, y^+]$ and $\overline{[y^-,y^+]}$ are invariant under the forward flow of the negative gradient vector field $-\nabla W$. But a ghost circle is also invariant under the backward flow. This implies that, if $\Gamma$ is a ghost circle and $y^-, y^+\in \Gamma$ are the endpoints of a gap in an Aubry-Mather set contained in $\Gamma$, then $\Gamma^{[y^-,y^+]}$ is both forward and backward invariant under the negative gradient flow. In order to prove that $y^-$ and $y^+$ are not the only fixed points in $[y^-,y^+]$, we will now show that $W_{[y^-,y^+]}$ acts as a Lyapunov function:

\begin{lemma}\label{lyapunov}
Let $y\in \Gamma^{[y^-,y^+]}$ and denote by $t\mapsto \Psi_t$ the flow of $-\nabla W$. Then $t\mapsto W_{[y^-, y^+]}(\Psi_ty)$ is continuously differentiable and
$$\left. \frac{d}{dt}\right|_{t=0} W_{[y^-, y^+]}(\Psi_ty) = - \!\! \sum_{i\in\Z^d/H_{\omega}}\!\! \left(\p_iW(y)\right)^2.$$ 
\end{lemma}

\begin{proof} Let us denote, for convenience, $||\nabla W(y)||^2:=\sum_{i\in \Z^d/H_{\omega}}(\p_iW(y))^2$. We will begin by showing that the function $y\mapsto ||\nabla W(y)||^2$ is absolutely convergent and continuous on $\overline{[y^-, y^+]}$. This is proved by interpolation, as in the proof of Proposition \ref{well_defined}. More precisely, if $|\p_iS_j|\leq D$ and $|\p_{i,k}S_j|\leq C$ uniformly on $\overline{[y^-,y^+]}$, then $|\p_iW|\leq (2r+1)^dD$, so that
\begin{align}
\left| ||\nabla W(y^1)||^2 - ||\nabla W(y^2)||^2 \right|  \leq & \sum_{i\in\Z^d/H_{\omega}} |\p_iW(y^1) + \p_iW(y^2)| \cdot |\p_iW(y^1) - \p_iW(y^2)| \nonumber \\ \nonumber \leq
\sum_{i\in \Z^d/H_{\omega}} \!\!  2(2r+1)^dD  \!\! \sum_{||j-i||\leq r}& \sum_{||k-j||\leq r} \! \left| \int_0^1 \! \p_{i,k}S_j(\tau y^1\!+\!(1-\tau)y^2)d\tau\right| \cdot |y^1_k-y^2_k| \\ \nonumber \leq \ 2CD&(2r+1)^{3d}  \!\! \sum_{k\in\Z^d/H_{\omega}}\! |y^1_k-y^2_k|\ .
\end{align}
Applied to $y^1=y^-$ and $y^2=y$ and combined with Theorem \ref{l1}, this implies that $||\nabla W(y)||^2$ is absolutely convergent, because $||\nabla W(y^-)||^2=0$. The continuity for pointwise convergence follows from the argument given in Proposition \ref{well_defined}. In particular, we now know that $t\mapsto -||\nabla W(\Psi_ty)||^2$ is continuous, being the composition of two continuous functions.\\ \indent 
The next step is to denote $y(t):=\Psi_ty$, for $y\in \Gamma^{[y^-, y^+]}$, and to observe that 
$$S_j(y(t))-S_j(y) =\int_0^t \frac{d}{d\tau}S_j(y(\tau))d\tau  = -\sum_{||k-j||\leq r} \int_0^t \p_kS_j(y(\tau)) \cdot\p_k W(y(\tau))d\tau.$$
Summing this over $j\in \Z^d/H_{\omega}$, we obtain because of the absolute convergence, that
$$W_{[y^-,y^+]}(y(t)) - W_{[y^-,y^+]}(y)= - \int_0^t ||\nabla W(y(\tau))||^2d\tau .$$
\noindent In particular, because the integrand is continuous, we find that $\left. \frac{d}{dt}\right|_{t=0} W_{[y^-, y^+]}(\Psi_ty) = \lim_{t\to 0} \frac{1}{t}\int_0^t||\nabla W(y(\tau))||^2d\tau = -||\nabla W(y)||^2$.
\end{proof}

\noindent Proposition \ref{well_defined} and Lemma \ref{lyapunov} combined now lead to the main result of this section:
 \begin{theorem}\label{gapsolution}
Let $[y^-,y^+]$ be a gap in the Aubry-Mather set $\mathcal{M}(x)$ and let $\Gamma$ be a ghost circle so that $\mathcal{M}(x)\subset \Gamma$. Then either $\Gamma^{[y^-,y^+]}= \Gamma\cap [y^-,y^+]$ consists of global minimizers only, or there is at least one stationary point $y \in \Gamma^{[y^-,y^+]}$ that is not a global minimizer.
\end{theorem}

\begin{proof}
By Proposition \ref{well_defined}, we have that $W_{[y^-,y^+]}\geq 0$. When $\left.W_{[y^-,y^+]}\right|_{\Gamma^{[y^-,y^+]}} \equiv 0$, then the flow-invariance of $\Gamma^{[y^-,y^+]}$ implies that $W_{[y^-,y^+]}(\Psi_ty)=0$ for all $y\in \Gamma^{[y^-,y^+]}$ and all $t\in \R$. By Proposition \ref{lyapunov}, we then have that $||\nabla W(y)||^2=0$. That is, $\Gamma^{[y^-,y^+]}$ consists of stationary points only, and hence by Theorem \ref{moserfoliation}, it consists of global minimizers only.\\
\indent Because $\Gamma^{[y^-,y^+]}$ is compact and $W_{[y^-,y^+]}$ is continuous, the other possibility is that $\left.W_{[y^-,y^+]}\right|_{\Gamma^{[y^-,y^+]}}$ assumes a positive maximum at some point $y\in \Gamma^{[y^-,y^+]}$ with $y^- \ll y \ll y^+$. Proposition \ref{well_defined} implies that this $y$ is not a global minimizer. It is clearly stationary though: if not, then $||\nabla W(y)||^2> 0$, so that by continuity of $t\mapsto ||\nabla W(\Psi_ty)||^2$, we have for each $t<0$ that 
$$W_{[y^-,y^+]}(\Psi_ty)-W_{[y^-,y^+]}(y) = - \int_0^t ||\nabla W(\Psi_{\tau}y)||^2d\tau>0. $$ 
Because $\Psi_ty\in \Gamma^{[y^-,y^+]}$, this contradicts that $y$ is a maximizer of $\left.W_{[y^-,y^+]}\right|_{\Gamma^{[y^-,y^+]}}$.
 \end{proof}
\noindent At this moment it is unclear to us whether a gap in an Aubry-Mather set can be foliated by stationary points - which therefore all have to be nonrecurrent global minimizers.

\appendix	
\section{Twist maps}
Variational monotone recurrence relations do not only arise in statistical mechanics or as discretized PDEs: the case of dimension $d=1$ is relevant for the theory of twist maps of the cylinder. The latter arise for instance in the study of convex billiards and as Poincar\'e maps of Hamiltonian systems. In this short descriptive appendix, we will briefly review these topics. The informed reader can skip this appendix and we refer to \cite{MatherForni} or \cite{gole01} for more detailed proofs of our statements, as well as for a more comprehensive introduction to the topic. \\ 
\indent Let us denote by $\mathbb{A}:= \mathbb{R}/\mathbb{Z} \times \mathbb{R}$ the standard cylinder, with coordinates $(x\!\! \mod 1,y)$ and bundle projection $\pi: \mathbb{A}\to \mathbb{R}/\mathbb{Z}$ given by $(x \!\! \mod 1,y)\mapsto x \!\! \mod 1$. The lift $\tilde \pi: \mathbb{R}^2\to\mathbb{R}$ of $\pi$ to the universal covering spaces sends $(x,y)$ to $x$.
 \\
\indent Recall that a smooth cylinder map $T: \mathbb{A }\to \mathbb{A}$ allows for a lift $\tilde T:\mathbb{R}^2\to\mathbb{R}^2$, with the property that $\tilde T(x,y) \!\! \mod (1,0)= T(x \!\! \mod 1 ,y)$. This implies that $\tilde T(x+1,y)=\tilde T(x,y)+ (n, 0)$, where $n$ is the degree of $T$, and moreover that $\tilde T$ is unique modulo constants of the form $(m,0)$ with $m\in \mathbb{Z}$. 
\begin{definition}
We call a cylinder map $T:\mathbb{A}\to \mathbb{A}$ an {\it exact symplectic positive twist map} if it satisfies conditions {\bf 1}, {\bf 2} and {\bf 3} below.
\end{definition}
\begin{itemize}
\item[{\bf 1.}] {\bf Degree one:} $\tilde T(x+1, y)=\tilde T(x,y) + (1,0)$. 
\end{itemize}
Condition 1 is true if and only if $T$ is homotopic to the identity map of $\mathbb{A}$. 
\begin{itemize}
\item[{\bf 2.}] {\bf Exact symplectic:} The one-form $T^*(ydx)-ydx$ is exact. 
\end{itemize}
Condition 2 implies that there is a so-called {\it generating function}, denoted $s:\mathbb{A}\to\mathbb{R}$, such that $T^*(ydx)-ydx = ds$. We will denote its lift by $\tilde s:\R^2\to\R$. This lift satisfies the identity $\tilde s(x, y)=s(x\!\! \mod 1, y)$. In particular, $\tilde s(x+1,y)=\tilde s(x,y)$.\\
\indent Geometrically, condition 2 can be interpreted as follows: one can show that when condition 1 holds, then condition 2 is true if and only if  $T$ preserves the volume form $dy\wedge dx$ and moreover has the property that the volume enclosed by the cycles $\gamma(t):=(t \!\! \mod 1, 0)$ and its homotopic image $T \circ \gamma$ is equal to zero.\\
\indent
Moreover, conditions 1 and 2 hold if and only if $T$ is a so-called Hamiltonian map, i.e. $T$ is the time-$1$ flow of a time-$1$-periodic Hamiltonian vector field $X_{H(t)}$ on $\mathbb{A}$.\\ 
\indent To formulate the last condition, let us call $\tilde T(x,y) = (X(x,y), Y(x,y))$.
\begin{itemize}
\item[{\bf 3.}] {\bf Positive Twist:} The map $(x,y)\mapsto (x,X(x,y)): \mathbb{R}^2\to\mathbb{R}^2$ is a diffeomorphism. This implies that $\p_yX\neq 0$. We require that $\p_yX>0$.
\end{itemize}
Condition 3 says that $T$ twists each fiber $\{x \!\! \mod 1\}\times \mathbb{R} \subset \mathbb{A}$ around the cylinder $\mathbb{A}$ ``in the positive direction''. We will denote the inverse of the map $(x,y)\mapsto (x, X(x,y))$ by $(x,X) \mapsto (x, y(x,X))$. \\  
\indent In fact, condition 3 allows us to define the function $S:\mathbb{R}^2\to \mathbb{R}$ by $S(x, X) := \tilde s(x, y(x,X))$. The function $S$ is called the {\it generating function} of the twist map $T$. \\ \indent The following well-known theorem is crucial in the theory of twist maps. It states that, in order to find orbits of exact symplectic positive twist maps of the cylinder, one needs to solve a variational monotone recurrence relation in dimension $d=1$. For completeness, we have included a brief proof of this statement.
\begin{theorem} 
The sequence $i\mapsto (x_i, y_i) \in \mathbb{R}^2$ is an orbit of $\tilde T$ if and only if \\ \mbox{} \\
{\bf i)} For all $M,N\in \mathbb{Z}$ with $M<N$ the sequence $i\mapsto x_i$ is a stationary point of the finite action 
$$W_{M, N}(x):= \sum_{i=M+1}^{N} S_i(x), \ \mbox{with} \ S_i(x):= S(x_{i-1},x_i),$$
for variations of $x$ with fixed endpoints $x_M$ and $x_N$. \\ \mbox{} \\
\noindent {\bf ii)} It holds that $y_i=-\p_xS(x_{i},x_{i+1})$ for all $i$.
\\ \mbox{} \\ \noindent
Moreover, one has that $S(x+1, X+1) = S(x, X)$ and $\partial_{x,X}S < 0$.
\end{theorem}
\begin{proof}
Recall the notation $\tilde T(x,y)=(X(x,y), Y(x,y))$, the diffeomorphism $(x,y)\mapsto (x, X(x,y))$ with its inverse $(x,X)\mapsto (x, y(x,X))$ and the definition $S(x,X)=\tilde s(x,y(x,X))$. Then the equality $T^*(ydx)=ydx+d s$ in the coordinate system $(x,y)$ becomes $YdX=ydx + dS$ in the coordinates $(x,X)$, viewing $y=y(x,X)$ and $Y=Y(x,y(x,X))$ as functions of $x$ and $X$. Writing $dS=\p_xS dx + \p_X SdX$, we thus obtain that 
\begin{align}
\label{generating11}
Y=\p_X S \ \mbox{and} \  y=-\p_x S\ . 
\end{align}
Now let $i\mapsto (x_i, y_i)$ be an arbitrary sequence and define $X_i := X(x_{i-1}, y_{i-1})$ and $Y_{i}:=Y(x_{i-1}, y_{i-1})$. Then $(x_i, y_i)$ is an orbit of $\tilde T$ if and only if $x_{i} = X_i$ and $y_{i} = Y_i$ for all $i$. According to formula (\ref{generating11}), we have that $Y_i= Y(x_{i-1}, y_{i-1}) = \p_XS(x_{i-1}, X(x_{i-1}, y_{i-1})) = \p_XS(x_{i-1}, X_i)$ and $y_{i} = -\p_xS(x_i, X(x_i,y_i)) = -\p_xS(x_i, X_{i+1})$. Thus, $x_i=X_i$ and $y_i=Y_i$ if and only if $y_i = -\p_xS(x_i, x_{i+1})$ and 
\begin{align}\label{1drecurrence}
\p_XS(x_{i-1}, x_i) + \p_xS(x_i, x_{i+1}) = 0 \ \mbox{for all} \ i\in \Z. 
\end{align}
Formula (\ref{1drecurrence}) holds if and only if $i\mapsto x_i$ is stationary for all the $W_{N,M}$ defined above. \\
\indent  Because $T$ has degree one, we have that $X(x+1, y)=X(x, y) + 1$. The function $y(x,X)$ is defined implicitly by the relation $X(x,y(x,X))=X$, and therefore we see that $y(x+1, X+1) = y(x, X)$. Thus, the generating function satisfies $S(x+1,X+1)  = \tilde s(x+1,y(x+1,X+1))= \tilde s(x,y(x,X)) = S(x,X)$. \\
\indent
Finally, formula (\ref{generating11}) implies that $\p_{x,X}S = -\p_X y = - (\p_yX)^{-1} < 0$.
\end{proof}
\subsubsection*{Examples of twist maps}
Perhaps the most famous example of an exact symplectic twist map is the {\it Chirikov standard map}. Given a $1$-periodic function $V=V(x)$, it is defined as $T_V: \mathbb{A} \to \mathbb{A}$ by $$T_V(x,y)=(x+y + 2V'(x) \!\!\! \!\mod 1, y+2V'(x))\  .$$
It turns out that its generating function is $S(x,X):= \frac{1}{2}(x-X)^2 + 2V(x)$. In other words, the variational monotone recurrence relation corresponding to $T_V$ is exactly the Frenkel-Kontorova equation in dimension $d=1$, given by $S_i(x)=\frac{1}{4}(x_{i-1}-x_i)^2+ V(x_i)$. By the way, the ``standard'' is to choose $V(x)=\frac{k}{8\pi^2}\cos(2\pi x)$,  for some parameter $k\geq 0$. This produces the map $(x\!\! \mod 1, y )\mapsto (x+y - \frac{k}{2\pi}\sin (2\pi x) \!\! \mod 1, y-\frac{k}{2\pi}\sin (2\pi x))$.\\ 
\indent
Another application of the theory of twist maps arises in the context of {\it convex billiards}, cf. \cite{Tabachnikov}. 
\begin{figure}[ht]  \centering 
 \includegraphics[width=5cm, height=6cm]{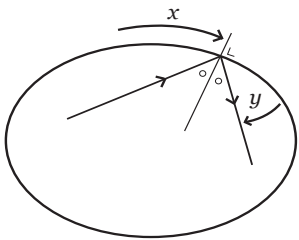}
\renewcommand{\figurename}{\rm \bf \footnotesize Figure} \vspace{-2.5cm}
\caption{\footnotesize An ellipse-shaped convex billiard and part of a billiard trajectory.} 
\label{billiard}
\end{figure}
The configuration space of such a billiard consists of the arclength parameters $x\in \mathbb{R}/\mathbb{Z}$ that describe the position of the billiard ball along the boundary of the billiard at the moment of reflection and angles $y\in (0,\pi)$ measuring the direction of the outgoing billiard trajectory with respect to the tangent line to the billiard at $x$. Then the motion of a billiard ball is described by an exact symplectic positive twist map $T:(x_i, y_i)\mapsto (x_{i+1}, y_{i+1})$. The variational structure of this problem follows as the rule ``angle of incidence $=$ angle of reflection'' is derived from the variational principle that a billiard ball travels along ``shortest paths''. The positive twist condition $\frac{\p x_{k+1}}{\p y_k}>0$ should be obvious from Figure \ref{billiard}.
\\
\indent
Finally, under generic conditions, the Poincar\'e return map of a $2$ degree of freedom Hamiltonian system near an elliptic equilibrium point is an exact symplectic twist map. In this case, the corresponding twist map is actually close to {\it integrable}, so that it allows for the application of various kinds of perturbation theory. Again, we refer to \cite{MatherForni} for more details.  
 
\begin{small}
\bibliographystyle{amsplain}
\bibliography{lattice}
\end{small}

 \end{document}